\documentclass[11pt]{article}
\usepackage[latin1]{inputenc}
\usepackage{epsfig}
\usepackage{color}
\usepackage[british,english]{babel}
\usepackage{amsthm}
\usepackage{amsmath}
\usepackage{amsfonts}
\usepackage{amssymb}
\usepackage{graphicx}
\newtheorem{corollary}{Corollary}[section]

\newtheorem{proposition}[corollary]{Proposition}
\newtheorem{remark}[corollary]{Remark}
\newtheorem{theorem}[corollary]{Theorem}
\newfont{\sBlackboard}{msbm10 scaled 900}

\newcommand{\Om}     {{\Omega}}

\newcommand{\mylabel}[1]{\label{#1}
            \ifx\undefined\stillediting
            \else \fbox{$#1$}\fi }
\newcommand{\BE}{\begin{equation}}

\newcommand{\EEQ}{\end{equation}}
\newcommand{\rfb}[1]{\mbox{\rm
   (\ref{#1})}\ifx\undefined\stillediting\else:\fbox{$#1$}\fi}
\newcommand{\half}   {{\frac{1}{2}}}

\newfont{\Blackboard}{msbm10 scaled 1200}
\newcommand{\bl}[1]{\mbox{\Blackboard #1}}
\newfont{\roma}{cmr10 scaled 1200}

\def\CC{\rm \hbox{C\kern-.56em\raise.4ex
         \hbox{$\scriptscriptstyle |$}\kern+0.5 em }}

\newcommand{\rline}  {{\bl R}}

\newcommand{\mm}    {{\hbox{\hskip 0.5pt}}}
\newcommand{\m}     {{\hbox{\hskip 1pt}}}

\newcommand{\bluff} {{\hbox{\raise 15pt \hbox{\mm}}}}
\makeatletter
\def\section{\@startsection {section}{1}{\z@}{-3.5ex plus -1ex minus
    -.2ex}{2.3ex plus .2ex}{\large\bf}}
\makeatother
\def\be{\begin{equation}}
\def\ee{\end{equation}}

\begin{document}

\thispagestyle{empty}
\title{\bf Feedback boundary stabilization of wave equations with interior delay}
\author{Ka\"{\i}s AMMARI
\thanks{D\'epartement de Math\'ematiques,
Facult\'e des Sciences de Monastir, 5019 Monastir, Tunisie,
e-mail : kais.ammari@fsm.rnu.tn}
\, , Serge NICAISE \thanks{Universit\'e de Valenciennes et du Hainaut Cambr\'esis, Institut des Sciences et Techniques de Valenciennes, 59313 Valenciennes Cedex 9, France, e-mail: snicaise@univ-valenciennes.fr}
\, and \, Cristina PIGNOTTI \thanks{Dipartimento di Matematica Pura e Applicata, Universit\`a di L'Aquila, Via Vetoio, Loc. Coppito, 67010 L'Aquila, Italy, \,
e-mail : pignotti@univaq.it}}
\date{}
\maketitle
{\bf Abstract.} {\small
In this paper we consider a boundary stabilization problem for the wave equation with interior delay.
We prove an exponential stability result under
some Lions geometric condition. The proof of the main result is based on an identity with multipliers
that allows to obtain a uniform decay estimate for  a suitable Lyapunov functional.} \\

\noindent
{\bf Mathematics Subject Classification (2000)}: 35B05, 93D15, 93D20\\
{\bf Keywords}: boundary stabilization, interior delay, wave equations

\section{Introduction}
\setcounter{equation}{0}
We study the boundary stabilization of a wave equation in an open
bounded domain $\Omega$ of $\rline^n, \, n \geq 2$. We denote by
$\partial\Omega$ the boundary of $\Omega$ and we assume that
$\partial \Omega = {\Gamma}_0 \cup {\Gamma}_1,$
where $\Gamma_0,\ \Gamma_1$ are closed subsets of $\partial \Omega$
with $\Gamma_0\cap
\Gamma_1=\emptyset$. Moreover we assume  $meas \Gamma_0>0.$ The
system is given by :
\begin{align}
{}&u_{tt}(x,t)-\Delta u(x,t)  + a u_t(x,t-\tau)= 0 ,  && x\in\Omega,\ t>0,\label{a1}\\
{}&u(x,t)=0,&& x\in  \Gamma_0,\ t> 0\label{a2}\\
{}&\frac{\partial u}{\partial \nu}(x,t)=- k u_t(x,t), && x\in  \Gamma_1,\ t> 0\label{a3}\\
{}&u(x,0) = u_0(x),  \quad  u_t(x,0)=u_1(x), && x\in
\Omega,\label{a4}\\
{}&u_t(x,t)=g(x,t),\quad  && x\in\Omega, \ t\in (-\tau,0), \label{a5}
\end{align}
where $\nu$ stands for the unit normal vector of $\partial \Omega$ pointing towards the exterior of $\Omega$ and $\frac{\partial u}{\partial\nu}$ is the normal
derivative.
Moreover, the constant $\tau >0$ is the time delay,
$a$ and $k$ are two positive
numbers and the initial data are taken in suitable spaces.

Denoting by $m$ the standard multiplier, that is $m(x)=x-x_0,$ we assume
\be
\label{cond stab}
m(x)\cdot\nu(x)\le 0, \quad x\in\Gamma_0,\quad \mbox{and}\quad
m(x)\cdot \nu(x)\ge\delta>0,\quad x\in\Gamma_1\,.
\ee

Delay effects arise in many applications and practical problems and
it is well-known that an arbitrarily small  delay may destabilize a
system which is uniformly asymptotically stable in absence of delay
(see e.g. \cite{Lag1,Lag2,Datko,DLP}).

The exponential stability of
\eqref{a1}--\eqref{a5} with $a=0$ has been studied in \cite{blr}
where it has been shown that the system is exponentially stable if
$\Gamma_1$ satisfies some geometric condition (BLR).
Moreover, if $\tau=0,$  that is in absence of delay, the above problem for any $a>0$ is exponentially stable even if
$k=0$ (see e.g. \cite{Zuazua}, \cite{lebeau}). On the contrary, in presence of a delay term there are instability phenomena.
In fact, as shown in \cite{nicaisepignotti1}, it is possible to find for the above problem in the case $k=0$ a sequence $\{\tau_k\}_k$ of delays with $\tau_k\rightarrow 0$ for which the corresponding solutions $u_k$ have an
increasing  energy.

In \cite{nicaisepignotti1} in order to contrast the destabilizing
effect of the time delay a ``good'' (not delayed) damping term is
introduced in \eqref{a1}. More precisely the problem   considered in
\cite{nicaisepignotti1} is
\begin{align}
{}&u_{tt}(x,t)-\Delta u(x,t)  + a u_t(x,t-\tau)+a_0u_t(x,t)= 0 , && x\in\Omega,\ t>0,\label{N1}\\
{}&u(x,t)=0,&& x\in  \Gamma_0,\ t> 0\label{N2}\\
{}&\frac{\partial u}{\partial \nu}(x,t)=0, && x\in  \Gamma_1,\ t> 0\label{N3}.
\end{align}
with $a_0,a>0$ and initial data in suitable spaces. If $a_0>a$, it
was shown in \cite{nicaisepignotti1} that system
(\ref{N1})-(\ref{N3}) is uniformly exponentially stable, see also
\cite{nicaisepignotti, amman, nicaisevalein} for related results.

In this paper the idea is to contrast the effect of the time delay by using the dissipative boundary feedback (\ref{a3}) (i.e., by giving the control in the
feedback form $-k \, u_t(x,t), \ \ \ x\in \Gamma_1,\ \ t> 0$).
We will show that if the condition \rfb{cond stab} is satisfied (geometric Lions condition), then for any $k>0$ system \eqref{a1}--\eqref{a5}
is exponentially stable for $a$ sufficiently small.

Let $A = - \Delta$ be the unbounded operator in $H = L^2(\Omega)$
with domain
$$H_{1}={\cal D}(A) = \left\{u \in H^2(\Omega), \, u_{|\Gamma_0} = 0, \, \frac{\partial u}{\partial \nu}_{|\Gamma_1} = 0 \right\}.$$
We define by
$$B \in {\cal L}(L^2(\Omega);L^2(\Omega)), \, Bu = B^*u = \sqrt{a} \, u, \forall \, u \in L^2(\Omega),$$
and
$$C \in {\cal L}(L^2(\Gamma_1);H_{-1}), \, Cu = \sqrt{k} \, A_{-1} Nu, \, \forall \, u \in L^2(\Gamma_1),$$
$$
C^*w =\sqrt{k}  w_{|\Gamma_1}, \, \forall \, w \in {\cal D}(A^\half)
= H_\half,$$ where $H_{-1} = ({\cal D}(A))^\prime$ (the duality is
in the sense of $H$), $A_{-1}$ is the extension of $A$ to $H$,
namely for all $h\in H$ and $\varphi \in {\cal D}(A)$, $A_{-1}h$ is
the unique element in $H_{-1}$ such that (see for instance
\cite{tucsnakweiss})
\[
\langle A_{-1}h; \varphi\rangle_{H_{-1}-H_{1}}=\int_\Omega h
A\varphi\, dx.\]

Here and below $N \in {\cal L}(L^2(\Gamma_1);L^2(\Omega)),$ $\forall
\, v \in L^2(\Gamma_1), \, Nv$ is the unique solution (transposition
solution) of
$$\Delta w =0, \,
w_{|\Gamma_0} = 0, \,
\frac{\partial w}{\partial \nu}_{|\Gamma_1} = v.$$

Setting $z(t,\theta) = \dot u(t+\theta), \, \theta \in (-\tau,0)$, the evolution
equation \rfb{a1}--\rfb{a5} is equivalent to
\begin{align}
\label{damped1b} &\ddot{u}(t) + Au(t) + CC^*\dot u(t) + BB^*z(t,-\tau) = 0,\quad  t\geq 0 ,\\
 &\dot z (t,\theta)-
z_\theta(t,\theta) = 0, \,\theta \in (-\tau,0),\quad  t\geq 0 , \\
&  z(t,0) =\dot{u}(t) ,\quad  t\geq 0 ,\label{dampedvb} \\
  & u(0) \m=\m u^0,  \m
\dot u(0)=u^1, \,\, z(0,\theta) = g(\theta ), \,\theta \in
(-\tau,0),\label{outputbbb}
\end{align}
 where $z_\theta =
\partial_\theta z$ and $g \in L^2(-\tau,0;H_{\frac12})$.

To study the well-posedness of the system
\eqref{damped1b}--\eqref{outputbbb}, we write it as an abstract
Cauchy problem in a product  space, and use the semigroup approach.
For this purpose, take the Hilbert space ${\mathcal H}~:=~ H_{\half}
\times H \times L^2(-\tau,0;H_\half)$ and the unbounded linear
operator \be \label{op} {\mathcal  A}_d : {\mathcal D}({\mathcal
A}_d) \subset {\mathcal  H} \longrightarrow {\mathcal  H}, \,
{\mathcal A}_d\left(
\begin{array}{ccc}u_1\\u_2\\z\end{array}\right) = \left(
\begin{array}{l}
u_2 \\
- Au_1 - CC^*u_2 - BB^*z(-\tau) \\
 \partial_\theta z
\end{array}
\right), \ee
where
\be\label{op2}
{\mathcal  D}({\mathcal  A}_d) :=\left\{ (u_1,u_2,z) \in  H_{\half} \times
H_{\half} \times H^1(-\tau,0;H_\half),\right.
\ee
$$\hspace{3cm}\left. \, A u_1 +  \,\\ CC^*u_2 + BB^*z(-\tau) \in H, z(0)= u_2 \right\}.
$$

\begin{proposition}\label{propexistunic}
\begin{enumerate}
\item
The operator $({\mathcal A}_d, {\mathcal  D}({\mathcal  A}_d) )$
generates a strongly continuous semigroup $({\mathcal
T}(t))_{t\geq0}$ on ${\mathcal H}$.
\item
The system
\rfb{a1}--\rfb{a5} is well-posed. More precisely, for every
$(u^0,u^1,g)\in {\mathcal  H}$, the function $w$ given by the first
component of \, ${\mathcal  T}(t)\left(
\begin{array}{ccc}u^0\\u^1\\g\end{array}\right)$ is the mild solution of \rfb{a1}--\rfb{a5}.
In particular, for  $(u^0,u^1,g)\in{\mathcal  D}({\mathcal  A}_d)$, the
problem \rfb{damped1b}--\rfb{outputbbb} admits a unique classical
solution
$$
(u,v,z) \in C([0,+\infty), {\mathcal  D}({\mathcal  A}_d))\cap
C^1([0,+\infty),{\mathcal  H}),
$$
and thus the problem \rfb{a1}--\rfb{a5} admits a unique
classical solution
$$
u \in C^1([-\tau,+\infty),H_\half)\cap C^2([0,+\infty),H).
$$
\end{enumerate}
\end{proposition}

For any regular solution of problem \rfb{a1}--\rfb{a5} we define the energy

$$
E(t):= E_S(t) + \frac{\xi}{2}\int_{t-\tau}^t\int_\Omega u_t^2(x,s)  dx ds \hspace{2cm}
$$
\begin{equation}\label{energy}
\hspace{3cm}
=\frac{1}{2}\int_\Omega\{ \vert \nabla u(x,t)\vert^2+u_t^2(x,t)\}dx
+\frac{\xi}{2}\int_{t-\tau}^t\int_\Omega u_t^2(x,s)  dx ds\,,
\end{equation}
where $\xi$ is a strictly positive real number.

The main result of this paper
is the following.

\begin{theorem} \label{princ}
For any $k>0$
there exist positive constants $a_0, C_1, C_2$ such that
\begin{equation}\label{expestimate}
E(t)\le C_1 e^{-C_2t} E(0),
\end{equation}
for any regular solution of problem
\eqref{a1}-\eqref{a5} with $0\le a<a_0.$
The constants $a_0, C_1, C_2$ are independent of the initial data but
they depend
 on $k$ and on the geometry of
$\Omega.$
\end{theorem}

The opposite problem, that is to contrast the effect of a time delay in the boundary condition
with a velocity term in the wave equation, is still, as far as we know, open and it seems to be
much harder to deal with. This will be the object of a future research.
However, there is a positive answer by Datko, Lagnese and Polis \cite{DLP} in the one dimensional case  for the problem
\begin{align}
{}&u_{tt}(x,t)-u_{xx}(x,t)  +2au_t(x,t)+ a^2 u(x,t)= 0 ,  && 0<x<1,\ t>0,\label{b1}\\
{}&u(0,t)=0,&& t> 0\label{b2}\\
{}&u_x(1,t)=- k u_t(1,t-\tau), &&\ t> 0;\label{b3}\\
\end{align}
with $a,k$ positive real  numbers. Indeed, through a careful
spectral analysis, in \cite{DLP} the authors have  shown that, for
any $a>0,$ if $k$ satisfies
\begin{equation}\label{forDLP}
0<k<\frac{1-e^{-2a}}{1+e^{-2a}},
\end{equation}
then the spectrum of the system \rfb{b1}--\rfb{b3} lies in $Re\omega\le-\beta,$ where $\beta$ is a positive constant depending on the delay $\tau.$

The paper is organized as follows. The second section deals with the
well-posedness of the problem while,  in the last section, we prove
the exponential stability of the delayed system \rfb{a1}--\rfb{a5}
by introducing a suitable Lyapunov functional.
\section{Wellposedness}
\setcounter{equation}{0}
 For the well-posedness of the equivalent equations \rfb{a1}--\rfb{a5} and  \rfb{damped1b}--\rfb{outputbbb},  we show that the operator $({\mathcal  A}_d, {\mathcal  D}({\mathcal  A}_d) )$ defined by \eqref{op}--\eqref{op2} generates a contraction semigroup on the Hilbert space ${\mathcal  H}:= H_{\half} \times H \times
L^2(-\tau,0;H_\half)$.

We introduce in ${\mathcal  H}$ the new inner product
$$
\left<\left(
\begin{array}{ccc}u_1\\u_2\\z_1\end{array}\right),\left(
\begin{array}{ccc}v_1\\v_2\\z_2\end{array}\right)\right>=
\left<u_1,v_1\right>_{H_\half}+\left<u_2,v_2\right>_{H}+
\xi \, \int_{-\tau}^0 \left<B^*z_1(\theta),B^*z_2(\theta)\right>_{L^2(\Om)}
d\theta,
$$
where  $\xi$ is a strictly positive constant.

It can be easily  seen    that ${\mathcal  H}$ endowed with this
inner product is a Hilbert space, and   its associated norm  is
equivalent to the canonical norm of ${\mathcal  H}$.

\begin{proof} (of Proposition \ref{propexistunic})

We show  that there exists a positive constant $c$ such that
 ${\mathcal A}_d-c I$ is dissipative.  Let $\left(
\begin{array}
{ccc}u\\v\\z
\end{array}
\right)\in {\mathcal  D}({\mathcal  A}_d)$, then by the
Cauchy-Schwarz inequality we have
\begin{align*}
\left<{\mathcal  A}_d\left(
\begin{array}{ccc}u\\v\\z\end{array}\right),\left(
\begin{array}{ccc}u\\v\\z\end{array}\right)\right>=&\left<\left(
\begin{array}{l}
v \\
- Au - \, CC^*v - \, BB^*z(-\tau) \\
  \partial_\theta z
\end{array}
\right),\left(
\begin{array}{ccc}u\\v\\z\end{array}\right)\right>\\
=&- <C^*v,C^*v>_{L^2(\Gamma_1)} - <B^*z(-\tau),B^*v>_{L^2(\Om)}\\
&+ \xi \, \int_{-\tau}^0\left<B^*\partial_\theta z(\theta),B^*z(\theta)\right>_{L^2(\Om)} d\theta\\
=&- <B^*z(-\tau),B^*v>_{L^2(\Om)} - \|C^*v\|^2_{L^2(\Gamma_1)} + \frac{\xi}{2} \, \|B^*v\|^2_{L^2(\Om)}\\
&- \frac{\xi}{2} \, \|B^*z(-\tau)\|_{L^2(\Om)}^2\\
\leq&
\left(\frac{1}{2 \xi} + \frac{\xi}{2} \right) \|B^*v\|^2_{L^2(\Om)}.
\end{align*}
Since $B^*$ is bounded from $L^2(\Om)$ into itself, we deduce that there exists $c>0$ such that
\[
\left<{\mathcal  A}_d\left(
\begin{array}{ccc}u\\v\\z\end{array}\right),\left(
\begin{array}{ccc}u\\v\\z\end{array}\right)\right>
\leq c \|v\|^2_{L^2(\Om)}.
\]
This shows that ${\mathcal A}_d-c I$ is dissipative.

Next, we show that $(\lambda I-{\mathcal  A}_d)$ is surjective for
some $\lambda>0$.

Given  a vector $\left(
\begin{array}{ccc}f\\g\\h\end{array}\right)\in{\mathcal  H}$, we need $\left(
\begin{array}{ccc}u\\v\\z\end{array}\right)\in{\mathcal  D}({\mathcal  A}_d)$ such that

$$(\lambda I-{\mathcal  A}_d )\left(
\begin{array}{ccc}u\\v\\z\end{array}\right)=\left(
\begin{array}{ccc}f\\g\\h\end{array}\right).
$$
This  is equivalent to

\be
    \lambda u-v=f,
    \label{1}
    \ee
    \be
     Au+(\lambda+CC^*)v + BB^*z(-\tau)=g,
     \label{2}
     \ee
     \be
    \lambda z-\partial_\theta z=h.
    \label{3}
    \ee
The function $z\in H^1(-\tau,0;H_\half)$ given by
$$
z(\theta)=e^{\lambda\theta}v+\int_\theta^0
e^{\lambda(\theta-\sigma)}h(\sigma)d\sigma
$$
is a solution to the  equation $(3)$ and verifies  $z(0)=v$. By
replacing $v$ from \rfb{1}  and $z$ in the equation \rfb{2} we are
reduced to find $u\in D(A)$  solution of
\begin{equation}\label{CE}
\left(\lambda^2I
+A+\lambda(CC^*+BB^* e^{-\lambda\tau} ) \right)u=k
\end{equation}
with
$$
k=g+\left(\lambda+ CC^* + e^{-\lambda\tau} BB^* \right)f+
BB^*\int_{-\tau}^0e^{-\lambda(\tau+\sigma)}h(\sigma)d\sigma.
$$
We now solve the equation \eqref{CE}. Assuming that $u\in D(A)$
exists and is a solution of \eqref{CE}, then   we have
$$\left<(\lambda^2I +A+\lambda(CC^*+BB^*e^{-\lambda\tau}))u,\zeta\right>=\left<k,\zeta\right>, \forall
\zeta\in H_\half$$ or equivalently \be
\Lambda(u,\zeta)=\left<k,\zeta\right>, \forall \zeta\in H_\half,
\label{CEweak}
    \ee
where
$$
\Lambda(u,\zeta):=\lambda^2\left<u,\zeta\right>+\left<A^\half
u,A^\half\zeta\right>+
\lambda(\left<C^*u,C^*\zeta\right>+e^{-\lambda\tau}
\left<B^*u,B^*\zeta\right>).
$$

Since  $\Lambda$ is a bilinear coercive form on $H_\half$, the
Lax-Milgram lemma leads to the existence and uniqueness of a
solution $u\in H_\half$ to \eqref{CEweak}. Some integrations by
parts allow to show that $u\in D(A)$ and is indeed solution of
equation \eqref{CE}. Consequently, $(\lambda I-{\mathcal A}_d)$ is
surjective and therefore $(\lambda I-({\mathcal A}_d-cI))$ is also
surjective. The density of  ${\cal D}({\mathcal A}_d)$ is clear.
Finally, the Lumer-Phillips theorem leads to the fact that
${\mathcal  A}_d-cI$ generates a strongly continuous semigroup of
contraction in ${\mathcal  H}$, hence ${\mathcal  A}_d$ generates a
strongly continuous semigroup  in ${\mathcal  H}$.
\end{proof}

It is easy to see that if
$u:[-\tau,\infty)\longrightarrow H_\half$ is a classical solution of
\rfb{a1}--\rfb{a5}, then $(u,\dot{u},\dot u(t+\cdot))$ is
the classical solution of the  equation
\rfb{damped1b}--\rfb{outputbbb}.  The first assertion of Proposition \ref{propexistunic} provides the
converse result and then the well-posedness of the evolution
equation \rfb{a1}--\rfb{a5}.

The well-posedness part follows from the first assertion of Proposition \ref{propexistunic}.


\section{Proof of Theorem \ref{princ}}\label{mainproof}
\setcounter{equation}{0}

\begin{proposition}\label{deriven}
For any solution of problem $(\ref{a1})-(\ref{a5})$
the following estimate holds:
\begin{equation}\label{stima1}
E^\prime(t)\le \frac {a+\xi}{2} \int_\Omega u_t^2(x,t)dx +\frac {a-\xi}{2}
\int_\Omega u_t^2(x, t-\tau) dx -k\int_{\Gamma_1}u_t^2(x,t) d\Gamma\,.
\end{equation}
\end{proposition}
\begin{proof} Differentiating (\ref{energy}) we obtain
$$
E^\prime(t)= \int_\Omega \{\nabla u(t)\nabla u_t(t)+u_t(t)u_{tt}(t)\} dx +\frac{\xi}{2}
\int_\Omega u^2_t(t) dx -\frac{\xi}{2}\int_\Omega u^2_t(t-\tau) dx,$$
and then, integrating by parts and using (\ref{a1}), (\ref{a2}), (\ref{a3}),
\begin{equation}\label{Ide}
E^\prime(t)= -a\int_\Omega u_t(t) u_t(t-\tau) dx
-k\int_{\Gamma_1}u_t^2(t) d\Gamma
 +\frac{\xi}{2}
\int_\Omega u^2_t(t) dx -\frac{\xi}{2}\int_\Omega u^2_t(t-\tau) dx\,.
\end{equation}

Applying Cauchy-Schwarz's inequality in (\ref{Ide}) we obtain
(\ref{stima1}).\end{proof}

\begin{proposition}\label{multiplier}
For any regular solution of problem $(\ref{a1})-(\ref{a5})$ and for every $\epsilon >0,$ we have
\begin{equation}\label{stimaM}
\begin{array}{l}
\displaystyle{
\frac{d}{dt}\left\{
\int_\Omega
[2m\cdot \nabla u +(n-1)u] u_t dx
\right\}\le -\int_\Omega u_t^2 (t) dx -\left ( 1-\frac{\epsilon}{2} C(P)\right)
\int_\Omega \vert\nabla u(t)\vert^2 dx}\\
\hspace{2 cm} \displaystyle{
-2a\int_\Omega (m\cdot\nabla u(t))u_t(t-\tau) dx - (n-1)a\int_\Omega u(t) u_t(t-\tau) dx
}\\
\hspace{2 cm} \displaystyle{
+\left[
\Big (\Vert m\Vert_\infty^2 \frac{2}{\delta} +\frac{(n-1)^2}{2\epsilon}\Big ) k^2+\Vert m\Vert_\infty\right ]\int_{\Gamma_1}u_t^2(t) d\Gamma\,,
}
\end{array}
\end{equation}
where $C(P)$ is a sort of Poincar\'e constant, which is  a positive
constant depending on $\Omega$ and independent of the solution $u.$
\end{proposition}

\begin{proof}
Differentiating and integrating by parts we obtain
\begin{equation}\label{M1}
\begin{array}{l}
\displaystyle{
\frac{d}{dt}\left\{
\int_\Omega
[2m\cdot \nabla u +(n-1)u] u_t dx
\right\}=-\int_\Omega\{u_t^2+\vert \nabla u\vert^2\} dx -2a \int_\Omega (m\cdot\nabla u)u_t(t-\tau) dx}\\
\hspace{2 cm}
\displaystyle{
-(n-1)a\int_{\Omega} u(t) u_t(t-\tau) dx +\int_\Gamma (m\cdot\nu )(u_t^2-
\vert\nabla u\vert^2) d\Gamma}\\
\hspace{2 cm}
\displaystyle{
+2\int_\Gamma (m\cdot\nabla u)\frac{\partial u}{\partial\nu} d\Gamma +(n-1)
\int_{\Gamma} u\frac{\partial u}{\partial\nu} d\Gamma\,.
}
\end{array}
\end{equation}
Now, since $u=0$ on $\Gamma_0$ and $m\cdot \nu\le 0$ on $\Gamma_0,$
from (\ref{M1}) we deduce

\begin{equation}\label{M2}
\begin{array}{l}
\displaystyle{
\frac{d}{dt}\left\{
\int_\Omega
[2m\cdot \nabla u +(n-1)u] u_t dx
\right\}\le -\int_\Omega\{u_t^2+\vert \nabla u\vert^2\} dx}\\
\hspace{1.2 cm}\displaystyle{
 -2a \int_\Omega (m\cdot\nabla u)u_t(t-\tau) dx
-(n-1)a\int_{\Omega} u(t) u_t(t-\tau) dx + \Vert m\Vert_\infty\int_{\Gamma_1}u_t^2 d\Gamma}\\
\hspace{1.2 cm}
\displaystyle{-\delta\int_{\Gamma_1} \vert\nabla u\vert^2 d\Gamma
+2\int_{\Gamma_1} (m\cdot\nabla u)\frac{\partial u}{\partial\nu} d\Gamma +(n-1)
\int_{\Gamma_1} u\frac{\partial u}{\partial\nu} d\Gamma\,,
}
\end{array}
\end{equation}
where we have used also $m\cdot\nu \ge \delta$ on $\Gamma_1.$

We can estimate
\begin{equation}\label{M3}
\begin{array}{l}
\displaystyle{
2\int_{\Gamma_1}(m\cdot\nabla u) \frac{\partial u}{\partial\nu} d\Gamma
\le \frac{\delta}{2}\int_{\Gamma_1}\vert\nabla u\vert^2 d\Gamma +2 \frac{\Vert
m\Vert_\infty^2}{\delta}\int_{\Gamma_1}\left\vert \frac{\partial u}{\partial\nu}
\right\vert^2 d\Gamma}\\
\hspace{2 cm} \displaystyle{
 =
 \frac{\delta}{2}\int_{\Gamma_1}\vert\nabla u\vert^2 d\Gamma +2 \frac{\Vert
m\Vert_\infty^2}{\delta}k^2\int_{\Gamma_1}u_t^2(t)d\Gamma\,.}
\end{array}
\end{equation}

Moreover,
\begin{equation}\label{M4}
\begin{array}{l}
\displaystyle{
(n-1)\int_{\Gamma_1}u \frac{\partial u}{\partial\nu} d\Gamma
\le \frac{\epsilon}{2}\int_{\Gamma_1} u^2 d\Gamma + \frac{
(n-1)^2}{2\epsilon}\int_{\Gamma_1}\left ( \frac{\partial u}{\partial\nu}
\right )^2 d\Gamma }
\\
\hspace{2 cm}\displaystyle{
\le
 \frac{\epsilon}{2}C(P)\int_{\Omega}\vert\nabla u\vert^2 dx
 + \frac{(n-1)^2}{2\epsilon}k^2\int_{\Gamma_1}u_t^2(t)d\Gamma\,,}
\end{array}
\end{equation}
where we have used trace inequality and Poincar\'e's theorem.

Substituting (\ref{M3}) and (\ref{M4}) in (\ref{M2}), we obtain
 the estimate (\ref{stimaM}).\end{proof}

\begin{remark}\label{CP}
 {{\rm
In the above proposition  $C(P)$ is the smallest positive constant such that
}$$ \int_{\Gamma_1}\varphi^2(x)d\Gamma\le C(P) \int_\Omega \vert\nabla \varphi (x)\vert^2 dx\,,\ \forall \ \varphi\in H^1_{\Gamma_0}(\Omega)\,.$$
}\end{remark}

\begin{corollary}\label{ancoraCS}
For any regular solution of $(\ref{a1})-(\ref{a5})$
\begin{equation}\label{stimaMM}
\begin{array}{l}
\displaystyle{
\frac{d}{dt}\left\{
\int_\Omega
[2m\cdot \nabla u +(n-1)u] u_t dx
\right\}\le -\int_\Omega u_t^2 (t) dx }\\
\hspace{1.2 cm}\displaystyle{
-\left ( 1-\frac{\epsilon}{2} C(P)- a
\Vert m\Vert_\infty^2 -\frac{a}{2} (n-1)^2 C_0(P)
\right)
\int_\Omega \vert\nabla u(t)\vert^2 dx}
\\
\hspace{1.2 cm} \displaystyle{
+\frac{3}{2}a\int_\Omega u^2_t(t-\tau) dx
+\left[\Big (\Vert m\Vert_\infty^2 \frac{2}{\delta} +\frac{(n-1)^2}{2\epsilon}
\Big ) k^2+\Vert m\Vert_\infty\right ]\int_{\Gamma_1}u_t^2(t) d\Gamma\,,
}
\end{array}
\end{equation}
where $(C_0(P))^{1/2}$ is the so-called Poincar\'e constant.
\end{corollary}

\begin{proof} We apply Cauchy-Schwarz's inequality to the integral in the second line of (\ref{stimaM}).\end{proof}
Now, let us introduce the functional
${\cal S}(t):= \int_\Omega \int_{t-\tau}^te^{s-t} u_t^2 (x,s) ds dx.$

We can easily estimate

\begin{equation}\label{func}
\begin{array}{l}
\displaystyle{
{\cal S}^\prime (t)=
\int_\Omega u_t^2(t) dx -\int_{\Omega}e^{-\tau} u_t^2(t-\tau) dx
-\int_\Omega\int_{t-\tau}^t e^{s-t} u_t^2(x,s) ds dx }\\
\hspace{2 cm}
\displaystyle{
\le \int_\Omega u_t^2(t) dx -e^{-\tau}\int_{\Omega} u_t^2(t-\tau) dx
-e^{-\tau}\int_\Omega\int_{t-\tau}^t  u_t^2(x,s) ds dx }\,.
\end{array}
\end{equation}

Let us introduce the Lyapunov functional
\begin{equation}\label{Lyapunov}
{\cal E}(t):= E(t)+\gamma_1 \int_\Omega
[2m\cdot \nabla u +(n-1)u] u_t dx
+\gamma_2 {\cal S}(t),
\end{equation}
where $\gamma_1, \gamma_2$ are suitable positive small constants
that will be precised later on.

Note that ${\cal E}(t)$ is equivalent to the energy $E(t)$ if
$\gamma_1$ is small enough. In particular, there exists a positive
constant $C_1$ and suitable positive constants $\alpha_1,\alpha_2$
such that
\begin{equation}\label{equiv}
  \alpha_1E(t)\le {\cal E}(t)\le \alpha_2 E(t)\,, \ \ \forall \ 0<\gamma_1,\ \gamma_2\leq C_1\,.
\end{equation}

\begin{proposition}\label{proExp}
For every $k >0$ there exist $a_0, c_1,c_2$ such that for any
solution of problem $(\ref{a1})-(\ref{a5})$ with $0\le a<a_0$ we have
\begin{equation}\label{Lyexp}
{\cal E}(t)\le c_1e^{-c_2 t}{\cal E}(0),\quad t>0.
\end{equation}
The constants $a_0, c_1, c_2$ are independent of the initial data but
they depend
 on $k$ and on the geometry of
$\Omega.$

\end{proposition}

\begin{proof}
Differentiating the Lyapunov functional ${\cal E}$ and using the
propositions above we deduce
\begin{equation}\label{lunga}
\begin{array}{l}
\displaystyle{
{\cal E}^\prime (t)\le \left ( \frac{a+\xi}{2} -\gamma_1+\gamma_2\right )
\int_\Omega u_t^2(x,t) dx -\gamma_2e^{-\tau}\int_\Omega
\int_{t-\tau}^tu_t^2(x,s) ds dx
}\\
\hspace{1.5 cm}
\displaystyle{
+
\left (
\frac{a-\xi}{2}+\frac{3}{2} a \gamma_1-\gamma_2e^{-\tau}
\right )\int_\Omega u_t^2(x,t-\tau) dx
}\\
\hspace{1.5 cm}
\displaystyle{
-\gamma_1
\left ( 1-\frac{\epsilon}{2} C(P)- a
\Vert m\Vert_\infty^2 -\frac{a}{2} (n-1)^2 C_0(P)
\right)
\int_\Omega \vert\nabla u(t)\vert^2 dx}
\\
\hspace{1.5 cm}
\displaystyle{
+\left \{\gamma_1 k^2
\left (
\Vert m\Vert_\infty^2 \frac{2}{\delta} +\frac{(n-1)^2}{2\epsilon}
\right )+\gamma_1\Vert m\Vert_\infty -k
\right\}\int_{\Gamma_1}u_t^2(x,t) dx\,.
}
\end{array}
\end{equation}
For a fixed $k >0$ we want to  chose $\epsilon,
\gamma_1, \gamma_2<C_1$ and  $a$
sufficiently small in order to obtain
\begin{equation}\label{quasi}
{\cal E}'(t)\le -c E(t).
\end{equation}
Applying the second inequality of (\ref{equiv}) estimate (\ref{Lyexp}) easily follows.

To show that (\ref{lunga}) implies (\ref{quasi}) we simply need
that
\begin{eqnarray*}
\frac{a+\xi}{2} -\gamma_1+\gamma_2<0,\\
\frac{a-\xi}{2}+\frac{3}{2} a \gamma_1-\gamma_2e^{-\tau}\leq 0,\\
 1-\frac{\epsilon}{2} C(P)- a \Vert m\Vert_\infty^2
-\frac{a}{2} (n-1)^2 C_0(P)> 0,\\ \gamma_1 k^2 \left ( \Vert
m\Vert_\infty^2 \frac{2}{\delta} +\frac{(n-1)^2}{2\epsilon} \right
)+\gamma_1\Vert m\Vert_\infty-k\leq 0.
\end{eqnarray*}
 These conditions are equivalent to
 \begin{eqnarray}
 \label{serge1}
\frac{a+\xi}{2}<\gamma_1 -\gamma_2,\\
\label{serge3}
 a\left (\frac12+\frac{3}{2}\gamma_1\right )-\frac{\xi}{2} \leq \gamma_2e^{-\tau},\\
 \label{serge4}
   a (\Vert m\Vert_\infty^2
+\frac{1}{2} (n-1)^2 C_0(P))<1-\frac{\epsilon}{2} C(P),\\
\label{serge5} \gamma_1\left[ k^2\Big ( \Vert m\Vert_\infty^2
\frac{2}{\delta} +\frac{(n-1)^2}{2\epsilon} \Big )+\Vert m\Vert_\infty\right]-k\leq 0.
\end{eqnarray}
For any
 $k>0$ this last condition is satisfied for  $\gamma_1$ sufficiently small
(once $\epsilon$ is fixed, see below). It then remains to  the
conditions (\ref{serge1}) to (\ref{serge4}). For the first one, we
need to assume that $\gamma_1>\gamma_2$, while for (\ref{serge4}) we
need to fix $\epsilon$ small enough such that
\begin{equation}\label{scelta}
1-\frac{\epsilon}{2} \, C(P)>0.
\end{equation}
Then we now fix $\gamma_1, \gamma_2$ and $\epsilon$ fulfilling the
above requirements and look at (\ref{serge1}) to (\ref{serge4}) as
conditions on $a$ and $\xi$. These conditions are simply linear
constraints and a simple analysis shows that the set of pairs
$(a,\xi)$ fulfilling these constraints is not empty (see Figure 1).


%
%
%
%
\ifx\figforTeXisloaded\relax\endinput\else\global\let\figforTeXisloaded=\relax\fi
\message{version 1.8.4}
\catcode`\@=11
\ifx\ctr@ln@m\undefined\else%
    \immediate\write16{*** Fig4TeX WARNING : \string\ctr@ln@m\space already defined.}\fi
\def\ctr@ln@m#1{\ifx#1\undefined\else%
    \immediate\write16{*** Fig4TeX WARNING : \string#1 already defined.}\fi}
\ctr@ln@m\ctr@ld@f
\def\ctr@ld@f#1#2{\ctr@ln@m#2#1#2}
\ctr@ld@f\def\ctr@ln@w#1#2{\ctr@ln@m#2\csname#1\endcsname#2}
{\catcode`\/=0 \catcode`/\=12 /ctr@ld@f/gdef/BS@{\}}
\ctr@ld@f\def\ctr@lcsn@m#1{\expandafter\ifx\csname#1\endcsname\relax\else%
    \immediate\write16{*** Fig4TeX WARNING : \BS@\expandafter\string#1\space already defined.}\fi}
\ctr@ld@f\edef\colonc@tcode{\the\catcode`\:}
\ctr@ld@f\edef\semicolonc@tcode{\the\catcode`\;}
\ctr@ld@f\def\t@stc@tcodech@nge{{\let\c@tcodech@nged=\z@%
    \ifnum\colonc@tcode=\the\catcode`\:\else\let\c@tcodech@nged=\@ne\fi%
    \ifnum\semicolonc@tcode=\the\catcode`\;\else\let\c@tcodech@nged=\@ne\fi%
    \ifx\c@tcodech@nged\@ne%
    \immediate\write16{}
    \immediate\write16{!!!=============================================================!!!}
    \immediate\write16{ Fig4TeX WARNING :}
    \immediate\write16{ The category code of some characters has been changed, which will}
    \immediate\write16{ result in an error (message "Runaway argument?").}
    \immediate\write16{ This probably comes from another package that changed the category}
    \immediate\write16{ code after Fig4TeX was loaded. If that proves to be exact, the}
    \immediate\write16{ solution is to exchange the loading commands on top of your file}
    \immediate\write16{ so that Fig4TeX is loaded last. For example, in LaTeX, we should}
    \immediate\write16{ say :}
    \immediate\write16{\BS@ usepackage[french]{babel}}
    \immediate\write16{\BS@ usepackage{fig4tex}}
    \immediate\write16{!!!=============================================================!!!}
    \immediate\write16{}
    \fi}}
\ctr@ld@f\def\FigforTeX{F\kern-.05em i\kern-.05em g\kern-.1em\raise-.14em\hbox{4}\kern-.19em\TeX}
\ctr@ln@w{newdimen}\epsil@n\epsil@n=0.00005pt
\ctr@ln@w{newdimen}\Cepsil@n\Cepsil@n=0.005pt
\ctr@ln@w{newdimen}\dcq@\dcq@=254pt
\ctr@ln@w{newdimen}\PI@\PI@=3.141592pt
\ctr@ln@w{newdimen}\DemiPI@deg\DemiPI@deg=90pt
\ctr@ln@w{newdimen}\PI@deg\PI@deg=180pt
\ctr@ln@w{newdimen}\DePI@deg\DePI@deg=360pt
\ctr@ld@f\chardef\t@n=10
\ctr@ld@f\chardef\c@nt=100
\ctr@ld@f\chardef\@lxxiv=74
\ctr@ld@f\chardef\@xci=91
\ctr@ld@f\mathchardef\@nMnCQn=9949
\ctr@ld@f\chardef\@vi=6
\ctr@ld@f\chardef\@xxx=30
\ctr@ld@f\chardef\@lvi=56
\ctr@ld@f\chardef\@@lxxi=71
\ctr@ld@f\chardef\@lxxxv=85
\ctr@ld@f\mathchardef\@@mmmmlxviii=4068
\ctr@ld@f\mathchardef\@ccclx=360
\ctr@ld@f\mathchardef\@dccxx=720
\ctr@ln@w{newcount}\p@rtent \ctr@ln@w{newcount}\f@ctech \ctr@ln@w{newcount}\result@tent
\ctr@ln@w{newdimen}\v@lmin \ctr@ln@w{newdimen}\v@lmax \ctr@ln@w{newdimen}\v@leur
\ctr@ln@w{newdimen}\result@t\ctr@ln@w{newdimen}\result@@t
\ctr@ln@w{newdimen}\mili@u \ctr@ln@w{newdimen}\c@rre \ctr@ln@w{newdimen}\delt@
\ctr@ld@f\def\degT@rd{0.017453 }  
\ctr@ld@f\def\rdT@deg{57.295779 } 
\ctr@ln@m\v@leurseule
{\catcode`p=12 \catcode`t=12 \gdef\v@leurseule#1pt{#1}}
\ctr@ld@f\def\repdecn@mb#1{\expandafter\v@leurseule\the#1\space}
\ctr@ld@f\def\arct@n#1(#2,#3){{\v@lmin=#2\v@lmax=#3%
    \maxim@m{\mili@u}{-\v@lmin}{\v@lmin}\maxim@m{\c@rre}{-\v@lmax}{\v@lmax}%
    \delt@=\mili@u\m@ech\mili@u%
    \ifdim\c@rre>\@nMnCQn\mili@u\divide\v@lmax\tw@\c@lATAN\v@leur(\z@,\v@lmax)
    \else%
    \maxim@m{\mili@u}{-\v@lmin}{\v@lmin}\maxim@m{\c@rre}{-\v@lmax}{\v@lmax}%
    \m@ech\c@rre%
    \ifdim\mili@u>\@nMnCQn\c@rre\divide\v@lmin\tw@
    \maxim@m{\mili@u}{-\v@lmin}{\v@lmin}\c@lATAN\v@leur(\mili@u,\z@)%
    \else\c@lATAN\v@leur(\delt@,\v@lmax)\fi\fi%
    \ifdim\v@lmin<\z@\v@leur=-\v@leur\ifdim\v@lmax<\z@\advance\v@leur-\PI@%
    \else\advance\v@leur\PI@\fi\fi%
    \global\result@t=\v@leur}#1=\result@t}
\ctr@ld@f\def\m@ech#1{\ifdim#1>1.646pt\divide\mili@u\t@n\divide\c@rre\t@n\m@ech#1\fi}
\ctr@ld@f\def\c@lATAN#1(#2,#3){{\v@lmin=#2\v@lmax=#3\v@leur=\z@\delt@=\tw@ pt%
    \un@iter{0.785398}{\v@lmax<}%
    \un@iter{0.463648}{\v@lmax<}%
    \un@iter{0.244979}{\v@lmax<}%
    \un@iter{0.124355}{\v@lmax<}%
    \un@iter{0.062419}{\v@lmax<}%
    \un@iter{0.031240}{\v@lmax<}%
    \un@iter{0.015624}{\v@lmax<}%
    \un@iter{0.007812}{\v@lmax<}%
    \un@iter{0.003906}{\v@lmax<}%
    \un@iter{0.001953}{\v@lmax<}%
    \un@iter{0.000976}{\v@lmax<}%
    \un@iter{0.000488}{\v@lmax<}%
    \un@iter{0.000244}{\v@lmax<}%
    \un@iter{0.000122}{\v@lmax<}%
    \un@iter{0.000061}{\v@lmax<}%
    \un@iter{0.000030}{\v@lmax<}%
    \un@iter{0.000015}{\v@lmax<}%
    \global\result@t=\v@leur}#1=\result@t}
\ctr@ld@f\def\un@iter#1#2{%
    \divide\delt@\tw@\edef\dpmn@{\repdecn@mb{\delt@}}%
    \mili@u=\v@lmin%
    \ifdim#2\z@%
      \advance\v@lmin-\dpmn@\v@lmax\advance\v@lmax\dpmn@\mili@u%
      \advance\v@leur-#1pt%
    \else%
      \advance\v@lmin\dpmn@\v@lmax\advance\v@lmax-\dpmn@\mili@u%
      \advance\v@leur#1pt%
    \fi}
\ctr@ld@f\def\c@ssin#1#2#3{\expandafter\ifx\csname COS@\number#3\endcsname\relax\c@lCS{#3pt}%
    \expandafter\xdef\csname COS@\number#3\endcsname{\repdecn@mb\result@t}%
    \expandafter\xdef\csname SIN@\number#3\endcsname{\repdecn@mb\result@@t}\fi%
    \edef#1{\csname COS@\number#3\endcsname}\edef#2{\csname SIN@\number#3\endcsname}}
\ctr@ld@f\def\c@lCS#1{{\mili@u=#1\p@rtent=\@ne%
    \relax\ifdim\mili@u<\z@\red@ng<-\else\red@ng>+\fi\f@ctech=\p@rtent%
    \relax\ifdim\mili@u<\z@\mili@u=-\mili@u\f@ctech=-\f@ctech\fi\c@@lCS}}
\ctr@ld@f\def\c@@lCS{\v@lmin=\mili@u\c@rre=-\mili@u\advance\c@rre\DemiPI@deg\v@lmax=\c@rre%
    \mili@u\@@lxxi\mili@u\divide\mili@u\@@mmmmlxviii%
    \edef\v@larg{\repdecn@mb{\mili@u}}\mili@u=-\v@larg\mili@u%
    \edef\v@lmxde{\repdecn@mb{\mili@u}}%
    \c@rre\@@lxxi\c@rre\divide\c@rre\@@mmmmlxviii%
    \edef\v@largC{\repdecn@mb{\c@rre}}\c@rre=-\v@largC\c@rre%
    \edef\v@lmxdeC{\repdecn@mb{\c@rre}}%
    \fctc@s\mili@u\v@lmin\global\result@t\p@rtent\v@leur%
    \let\t@mp=\v@larg\let\v@larg=\v@largC\let\v@largC=\t@mp%
    \let\t@mp=\v@lmxde\let\v@lmxde=\v@lmxdeC\let\v@lmxdeC=\t@mp%
    \fctc@s\c@rre\v@lmax\global\result@@t\f@ctech\v@leur}
\ctr@ld@f\def\fctc@s#1#2{\v@leur=#1\relax\ifdim#2<\@lxxxv\p@\cosser@h\else\sinser@t\fi}
\ctr@ld@f\def\cosser@h{\advance\v@leur\@lvi\p@\divide\v@leur\@lvi%
    \v@leur=\v@lmxde\v@leur\advance\v@leur\@xxx\p@%
    \v@leur=\v@lmxde\v@leur\advance\v@leur\@ccclx\p@%
    \v@leur=\v@lmxde\v@leur\advance\v@leur\@dccxx\p@\divide\v@leur\@dccxx}
\ctr@ld@f\def\sinser@t{\v@leur=\v@lmxdeC\p@\advance\v@leur\@vi\p@%
    \v@leur=\v@largC\v@leur\divide\v@leur\@vi}
\ctr@ld@f\def\red@ng#1#2{\relax\ifdim\mili@u#1#2\DemiPI@deg\advance\mili@u#2-\PI@deg%
    \p@rtent=-\p@rtent\red@ng#1#2\fi}
\ctr@ld@f\def\pr@c@lCS#1#2#3{\ctr@lcsn@m{COS@\number#3 }%
    \expandafter\xdef\csname COS@\number#3\endcsname{#1}%
    \expandafter\xdef\csname SIN@\number#3\endcsname{#2}}
\pr@c@lCS{1}{0}{0}
\pr@c@lCS{0.7071}{0.7071}{45}\pr@c@lCS{0.7071}{-0.7071}{-45}
\pr@c@lCS{0}{1}{90}          \pr@c@lCS{0}{-1}{-90}
\pr@c@lCS{-1}{0}{180}        \pr@c@lCS{-1}{0}{-180}
\pr@c@lCS{0}{-1}{270}        \pr@c@lCS{0}{1}{-270}
\ctr@ld@f\def\invers@#1#2{{\v@leur=#2\maxim@m{\v@lmax}{-\v@leur}{\v@leur}%
    \f@ctech=\@ne\m@inv@rs%
    \multiply\v@leur\f@ctech\edef\v@lv@leur{\repdecn@mb{\v@leur}}%
    \p@rtentiere{\p@rtent}{\v@leur}\v@lmin=\p@\divide\v@lmin\p@rtent%
    \inv@rs@\multiply\v@lmax\f@ctech\global\result@t=\v@lmax}#1=\result@t}
\ctr@ld@f\def\m@inv@rs{\ifdim\v@lmax<\p@\multiply\v@lmax\t@n\multiply\f@ctech\t@n\m@inv@rs\fi}
\ctr@ld@f\def\inv@rs@{\v@lmax=-\v@lmin\v@lmax=\v@lv@leur\v@lmax%
    \advance\v@lmax\tw@ pt\v@lmax=\repdecn@mb{\v@lmin}\v@lmax%
    \delt@=\v@lmax\advance\delt@-\v@lmin\ifdim\delt@<\z@\delt@=-\delt@\fi%
    \ifdim\delt@>\epsil@n\v@lmin=\v@lmax\inv@rs@\fi}
\ctr@ld@f\def\minim@m#1#2#3{\relax\ifdim#2<#3#1=#2\else#1=#3\fi}
\ctr@ld@f\def\maxim@m#1#2#3{\relax\ifdim#2>#3#1=#2\else#1=#3\fi}
\ctr@ld@f\def\p@rtentiere#1#2{#1=#2\divide#1by65536 }
\ctr@ld@f\def\r@undint#1#2{{\v@leur=#2\divide\v@leur\t@n\p@rtentiere{\p@rtent}{\v@leur}%
    \v@leur=\p@rtent pt\global\result@t=\t@n\v@leur}#1=\result@t}
\ctr@ld@f\def\sqrt@#1#2{{\v@leur=#2%
    \minim@m{\v@lmin}{\p@}{\v@leur}\maxim@m{\v@lmax}{\p@}{\v@leur}%
    \f@ctech=\@ne\m@sqrt@\sqrt@@%
    \mili@u=\v@lmin\advance\mili@u\v@lmax\divide\mili@u\tw@\multiply\mili@u\f@ctech%
    \global\result@t=\mili@u}#1=\result@t}
\ctr@ld@f\def\m@sqrt@{\ifdim\v@leur>\dcq@\divide\v@leur\c@nt\v@lmax=\v@leur%
    \multiply\f@ctech\t@n\m@sqrt@\fi}
\ctr@ld@f\def\sqrt@@{\mili@u=\v@lmin\advance\mili@u\v@lmax\divide\mili@u\tw@%
    \c@rre=\repdecn@mb{\mili@u}\mili@u%
    \ifdim\c@rre<\v@leur\v@lmin=\mili@u\else\v@lmax=\mili@u\fi%
    \delt@=\v@lmax\advance\delt@-\v@lmin\ifdim\delt@>\epsil@n\sqrt@@\fi}
\ctr@ld@f\def\extrairelepremi@r#1\de#2{\expandafter\lepremi@r#2@#1#2}
\ctr@ld@f\def\lepremi@r#1,#2@#3#4{\def#3{#1}\def#4{#2}\ignorespaces}
\ctr@ld@f\def\@cfor#1:=#2\do#3{%
  \edef\@fortemp{#2}%
  \ifx\@fortemp\empty\else\@cforloop#2,\@nil,\@nil\@@#1{#3}\fi}
\ctr@ln@m\@nextwhile
\ctr@ld@f\def\@cforloop#1,#2\@@#3#4{%
  \def#3{#1}%
  \ifx#3\Fig@nnil\let\@nextwhile=\Fig@fornoop\else#4\relax\let\@nextwhile=\@cforloop\fi%
  \@nextwhile#2\@@#3{#4}}

\ctr@ld@f\def\@ecfor#1:=#2\do#3{%
  \def\@@cfor{\@cfor#1:=}%
  \edef\@@@cfor{#2}%
  \expandafter\@@cfor\@@@cfor\do{#3}}
\ctr@ld@f\def\Fig@nnil{\@nil}
\ctr@ld@f\def\Fig@fornoop#1\@@#2#3{}
\ctr@ln@m\list@@rg
\ctr@ld@f\def\trtlis@rg#1#2{\def\list@@rg{#1}%
    \@ecfor\p@rv@l:=\list@@rg\do{\expandafter#2\p@rv@l|}}
\ctr@ln@w{newbox}\b@xvisu
\ctr@ln@w{newtoks}\let@xte
\ctr@ln@w{newif}\ifitis@K
\ctr@ln@w{newcount}\s@mme
\ctr@ln@w{newcount}\l@mbd@un \ctr@ln@w{newcount}\l@mbd@de
\ctr@ln@w{newcount}\superc@ntr@l\superc@ntr@l=\@ne        
\ctr@ln@w{newcount}\typec@ntr@l\typec@ntr@l=\superc@ntr@l 
\ctr@ln@w{newdimen}\v@lX  \ctr@ln@w{newdimen}\v@lY  \ctr@ln@w{newdimen}\v@lZ
\ctr@ln@w{newdimen}\v@lXa \ctr@ln@w{newdimen}\v@lYa \ctr@ln@w{newdimen}\v@lZa
\ctr@ln@w{newdimen}\unit@\unit@=\p@ 
\ctr@ld@f\def\unit@util{pt}
\ctr@ld@f\def\ptT@ptps{0.996264 }
\ctr@ld@f\def\ptpsT@pt{1.00375 }
\ctr@ld@f\def\ptT@unit@{1} 
\ctr@ld@f\def\setunit@#1{\def\unit@util{#1}\setunit@@#1:\invers@{\result@t}{\unit@}%
    \edef\ptT@unit@{\repdecn@mb\result@t}}
\ctr@ld@f\def\setunit@@#1#2:{\ifcat#1a\unit@=\@ne#1#2\else\unit@=#1#2\fi}
\ctr@ld@f\def\d@fm@cdim#1#2{{\v@leur=#2\v@leur=\ptT@unit@\v@leur\xdef#1{\repdecn@mb\v@leur}}}
\ctr@ln@w{newif}\ifBdingB@x\BdingB@xtrue
\ctr@ln@w{newdimen}\c@@rdXmin \ctr@ln@w{newdimen}\c@@rdYmin  
\ctr@ln@w{newdimen}\c@@rdXmax \ctr@ln@w{newdimen}\c@@rdYmax
\ctr@ld@f\def\b@undb@x#1#2{\ifBdingB@x%
    \relax\ifdim#1<\c@@rdXmin\global\c@@rdXmin=#1\fi%
    \relax\ifdim#2<\c@@rdYmin\global\c@@rdYmin=#2\fi%
    \relax\ifdim#1>\c@@rdXmax\global\c@@rdXmax=#1\fi%
    \relax\ifdim#2>\c@@rdYmax\global\c@@rdYmax=#2\fi\fi}
\ctr@ld@f\def\b@undb@xP#1{{\Figg@tXY{#1}\b@undb@x{\v@lX}{\v@lY}}}
\ctr@ld@f\def\ellBB@x#1;#2,#3(#4,#5,#6){{\s@uvc@ntr@l\et@tellBB@x%
    \setc@ntr@l{2}\figptell-2::#1;#2,#3(#4,#6)\b@undb@xP{-2}%
    \figptell-2::#1;#2,#3(#5,#6)\b@undb@xP{-2}%
    \c@ssin{\C@}{\S@}{#6}\v@lmin=\C@ pt\v@lmax=\S@ pt%
    \mili@u=#3\v@lmin\delt@=#2\v@lmax\arct@n\v@leur(\delt@,\mili@u)%
    \mili@u=-#3\v@lmax\delt@=#2\v@lmin\arct@n\c@rre(\delt@,\mili@u)%
    \v@leur=\rdT@deg\v@leur\advance\v@leur-\DePI@deg%
    \c@rre=\rdT@deg\c@rre\advance\c@rre-\DePI@deg%
    \v@lmin=#4pt\v@lmax=#5pt%
    \loop\ifdim\v@leur<\v@lmax\ifdim\v@leur>\v@lmin%
    \edef\@ngle{\repdecn@mb\v@leur}\figptell-2::#1;#2,#3(\@ngle,#6)%
    \b@undb@xP{-2}\fi\advance\v@leur\PI@deg\repeat%
    \loop\ifdim\c@rre<\v@lmax\ifdim\c@rre>\v@lmin%
    \edef\@ngle{\repdecn@mb\c@rre}\figptell-2::#1;#2,#3(\@ngle,#6)%
    \b@undb@xP{-2}\fi\advance\c@rre\PI@deg\repeat%
    \resetc@ntr@l\et@tellBB@x}\ignorespaces}
\ctr@ld@f\def\initb@undb@x{\c@@rdXmin=\maxdimen\c@@rdYmin=\maxdimen%
    \c@@rdXmax=-\maxdimen\c@@rdYmax=-\maxdimen}
\ctr@ld@f\def\c@ntr@lnum#1{%
    \relax\ifnum\typec@ntr@l=\@ne%
    \ifnum#1<\z@%
    \immediate\write16{*** Forbidden point number (#1). Abort.}\end\fi\fi%
    \set@bjc@de{#1}}
\ctr@ln@m\objc@de
\ctr@ld@f\def\set@bjc@de#1{\edef\objc@de{@BJ\ifnum#1<\z@ M\romannumeral-#1\else\romannumeral#1\fi}}
\s@mme=\m@ne\loop\ifnum\s@mme>-19
  \set@bjc@de{\s@mme}\ctr@lcsn@m\objc@de\ctr@lcsn@m{\objc@de T}
\advance\s@mme\m@ne\repeat
\s@mme=\@ne\loop\ifnum\s@mme<6
  \set@bjc@de{\s@mme}\ctr@lcsn@m\objc@de\ctr@lcsn@m{\objc@de T}
\advance\s@mme\@ne\repeat
\ctr@ld@f\def\setc@ntr@l#1{\ifnum\superc@ntr@l>#1\typec@ntr@l=\superc@ntr@l%
    \else\typec@ntr@l=#1\fi}
\ctr@ld@f\def\resetc@ntr@l#1{\global\superc@ntr@l=#1\setc@ntr@l{#1}}
\ctr@ld@f\def\s@uvc@ntr@l#1{\edef#1{\the\superc@ntr@l}}
\ctr@ln@m\c@lproscal
\ctr@ld@f\def\c@lproscalDD#1[#2,#3]{{\Figg@tXY{#2}%
    \edef\Xu@{\repdecn@mb{\v@lX}}\edef\Yu@{\repdecn@mb{\v@lY}}\Figg@tXY{#3}%
    \global\result@t=\Xu@\v@lX\global\advance\result@t\Yu@\v@lY}#1=\result@t}
\ctr@ld@f\def\c@lproscalTD#1[#2,#3]{{\Figg@tXY{#2}\edef\Xu@{\repdecn@mb{\v@lX}}%
    \edef\Yu@{\repdecn@mb{\v@lY}}\edef\Zu@{\repdecn@mb{\v@lZ}}%
    \Figg@tXY{#3}\global\result@t=\Xu@\v@lX\global\advance\result@t\Yu@\v@lY%
    \global\advance\result@t\Zu@\v@lZ}#1=\result@t}
\ctr@ld@f\def\c@lprovec#1{%
    \det@rmC\v@lZa(\v@lX,\v@lY,\v@lmin,\v@lmax)%
    \det@rmC\v@lXa(\v@lY,\v@lZ,\v@lmax,\v@leur)%
    \det@rmC\v@lYa(\v@lZ,\v@lX,\v@leur,\v@lmin)%
    \Figv@ctCreg#1(\v@lXa,\v@lYa,\v@lZa)}
\ctr@ld@f\def\det@rm#1[#2,#3]{{\Figg@tXY{#2}\Figg@tXYa{#3}%
    \delt@=\repdecn@mb{\v@lX}\v@lYa\advance\delt@-\repdecn@mb{\v@lY}\v@lXa%
    \global\result@t=\delt@}#1=\result@t}
\ctr@ld@f\def\det@rmC#1(#2,#3,#4,#5){{\global\result@t=\repdecn@mb{#2}#5%
    \global\advance\result@t-\repdecn@mb{#3}#4}#1=\result@t}
\ctr@ld@f\def\getredf@ctDD#1(#2,#3){{\maxim@m{\v@lXa}{-#2}{#2}\maxim@m{\v@lYa}{-#3}{#3}%
    \maxim@m{\v@lXa}{\v@lXa}{\v@lYa}
    \ifdim\v@lXa>\@xci pt\divide\v@lXa\@xci%
    \p@rtentiere{\p@rtent}{\v@lXa}\advance\p@rtent\@ne\else\p@rtent=\@ne\fi%
    \global\result@tent=\p@rtent}#1=\result@tent\ignorespaces}
\ctr@ld@f\def\getredf@ctTD#1(#2,#3,#4){{\maxim@m{\v@lXa}{-#2}{#2}\maxim@m{\v@lYa}{-#3}{#3}%
    \maxim@m{\v@lZa}{-#4}{#4}\maxim@m{\v@lXa}{\v@lXa}{\v@lYa}%
    \maxim@m{\v@lXa}{\v@lXa}{\v@lZa}
    \ifdim\v@lXa>\@lxxiv pt\divide\v@lXa\@lxxiv%
    \p@rtentiere{\p@rtent}{\v@lXa}\advance\p@rtent\@ne\else\p@rtent=\@ne\fi%
    \global\result@tent=\p@rtent}#1=\result@tent\ignorespaces}
\ctr@ld@f\def\FigptintercircB@zDD#1:#2:#3,#4[#5,#6,#7,#8]{{\s@uvc@ntr@l\et@tfigptintercircB@zDD%
    \setc@ntr@l{2}\figvectPDD-1[#5,#8]\Figg@tXY{-1}\getredf@ctDD\f@ctech(\v@lX,\v@lY)%
    \mili@u=#4\unit@\divide\mili@u\f@ctech\c@rre=\repdecn@mb{\mili@u}\mili@u%
    \figptBezierDD-5::#3[#5,#6,#7,#8]%
    \v@lmin=#3\p@\v@lmax=\v@lmin\advance\v@lmax0.1\p@%
    \loop\edef\T@{\repdecn@mb{\v@lmax}}\figptBezierDD-2::\T@[#5,#6,#7,#8]%
    \figvectPDD-1[-5,-2]\n@rmeucCDD{\delt@}{-1}\ifdim\delt@<\c@rre\v@lmin=\v@lmax%
    \advance\v@lmax0.1\p@\repeat%
    \loop\mili@u=\v@lmin\advance\mili@u\v@lmax%
    \divide\mili@u\tw@\edef\T@{\repdecn@mb{\mili@u}}\figptBezierDD-2::\T@[#5,#6,#7,#8]%
    \figvectPDD-1[-5,-2]\n@rmeucCDD{\delt@}{-1}\ifdim\delt@>\c@rre\v@lmax=\mili@u%
    \else\v@lmin=\mili@u\fi\v@leur=\v@lmax\advance\v@leur-\v@lmin%
    \ifdim\v@leur>\epsil@n\repeat\figptcopyDD#1:#2/-2/%
    \resetc@ntr@l\et@tfigptintercircB@zDD}\ignorespaces}
\ctr@ln@m\figptinterlines
\ctr@ld@f\def\inters@cDD#1:#2[#3,#4;#5,#6]{{\s@uvc@ntr@l\et@tinters@cDD%
    \setc@ntr@l{2}\vecunit@{-1}{#4}\vecunit@{-2}{#6}%
    \Figg@tXY{-1}\setc@ntr@l{1}\Figg@tXYa{#3}%
    \edef\A@{\repdecn@mb{\v@lX}}\edef\B@{\repdecn@mb{\v@lY}}%
    \v@lmin=\B@\v@lXa\advance\v@lmin-\A@\v@lYa%
    \Figg@tXYa{#5}\setc@ntr@l{2}\Figg@tXY{-2}%
    \edef\C@{\repdecn@mb{\v@lX}}\edef\D@{\repdecn@mb{\v@lY}}%
    \v@lmax=\D@\v@lXa\advance\v@lmax-\C@\v@lYa%
    \delt@=\A@\v@lY\advance\delt@-\B@\v@lX%
    \invers@{\v@leur}{\delt@}\edef\v@ldelta{\repdecn@mb{\v@leur}}%
    \v@lXa=\A@\v@lmax\advance\v@lXa-\C@\v@lmin%
    \v@lYa=\B@\v@lmax\advance\v@lYa-\D@\v@lmin%
    \v@lXa=\v@ldelta\v@lXa\v@lYa=\v@ldelta\v@lYa%
    \setc@ntr@l{1}\Figp@intregDD#1:{#2}(\v@lXa,\v@lYa)%
    \resetc@ntr@l\et@tinters@cDD}\ignorespaces}
\ctr@ld@f\def\inters@cTD#1:#2[#3,#4;#5,#6]{{\s@uvc@ntr@l\et@tinters@cTD%
    \setc@ntr@l{2}\figvectNVTD-1[#4,#6]\figvectNVTD-2[#6,-1]\figvectPTD-1[#3,#5]%
    \r@pPSTD\v@leur[-2,-1,#4]\edef\v@lcoef{\repdecn@mb{\v@leur}}%
    \figpttraTD#1:{#2}=#3/\v@lcoef,#4/\resetc@ntr@l\et@tinters@cTD}\ignorespaces}
\ctr@ld@f\def\r@pPSTD#1[#2,#3,#4]{{\Figg@tXY{#2}\edef\Xu@{\repdecn@mb{\v@lX}}%
    \edef\Yu@{\repdecn@mb{\v@lY}}\edef\Zu@{\repdecn@mb{\v@lZ}}%
    \Figg@tXY{#3}\v@lmin=\Xu@\v@lX\advance\v@lmin\Yu@\v@lY\advance\v@lmin\Zu@\v@lZ%
    \Figg@tXY{#4}\v@lmax=\Xu@\v@lX\advance\v@lmax\Yu@\v@lY\advance\v@lmax\Zu@\v@lZ%
    \invers@{\v@leur}{\v@lmax}\global\result@t=\repdecn@mb{\v@leur}\v@lmin}%
    #1=\result@t}
\ctr@ln@m\n@rminf
\ctr@ld@f\def\n@rminfDD#1#2{{\Figg@tXY{#2}\maxim@m{\v@lX}{\v@lX}{-\v@lX}%
    \maxim@m{\v@lY}{\v@lY}{-\v@lY}\maxim@m{\global\result@t}{\v@lX}{\v@lY}}%
    #1=\result@t}
\ctr@ld@f\def\n@rminfTD#1#2{{\Figg@tXY{#2}\maxim@m{\v@lX}{\v@lX}{-\v@lX}%
    \maxim@m{\v@lY}{\v@lY}{-\v@lY}\maxim@m{\v@lZ}{\v@lZ}{-\v@lZ}%
    \maxim@m{\v@lX}{\v@lX}{\v@lY}\maxim@m{\global\result@t}{\v@lX}{\v@lZ}}%
    #1=\result@t}
\ctr@ld@f\def\n@rmeucCDD#1#2{\Figg@tXY{#2}\divide\v@lX\f@ctech\divide\v@lY\f@ctech%
    #1=\repdecn@mb{\v@lX}\v@lX\v@lX=\repdecn@mb{\v@lY}\v@lY\advance#1\v@lX}
\ctr@ld@f\def\n@rmeucCTD#1#2{\Figg@tXY{#2}%
    \divide\v@lX\f@ctech\divide\v@lY\f@ctech\divide\v@lZ\f@ctech%
    #1=\repdecn@mb{\v@lX}\v@lX\v@lX=\repdecn@mb{\v@lY}\v@lY\advance#1\v@lX%
    \v@lX=\repdecn@mb{\v@lZ}\v@lZ\advance#1\v@lX}
\ctr@ln@m\n@rmeucSV
\ctr@ld@f\def\n@rmeucSVDD#1#2{{\Figg@tXY{#2}%
    \v@lXa=\repdecn@mb{\v@lX}\v@lX\v@lYa=\repdecn@mb{\v@lY}\v@lY%
    \advance\v@lXa\v@lYa\sqrt@{\global\result@t}{\v@lXa}}#1=\result@t}
\ctr@ld@f\def\n@rmeucSVTD#1#2{{\Figg@tXY{#2}\v@lXa=\repdecn@mb{\v@lX}\v@lX%
    \v@lYa=\repdecn@mb{\v@lY}\v@lY\v@lZa=\repdecn@mb{\v@lZ}\v@lZ%
    \advance\v@lXa\v@lYa\advance\v@lXa\v@lZa\sqrt@{\global\result@t}{\v@lXa}}#1=\result@t}
\ctr@ln@m\n@rmeuc
\ctr@ld@f\def\n@rmeucDD#1#2{{\Figg@tXY{#2}\getredf@ctDD\f@ctech(\v@lX,\v@lY)%
    \divide\v@lX\f@ctech\divide\v@lY\f@ctech%
    \v@lXa=\repdecn@mb{\v@lX}\v@lX\v@lYa=\repdecn@mb{\v@lY}\v@lY%
    \advance\v@lXa\v@lYa\sqrt@{\global\result@t}{\v@lXa}%
    \global\multiply\result@t\f@ctech}#1=\result@t}
\ctr@ld@f\def\n@rmeucTD#1#2{{\Figg@tXY{#2}\getredf@ctTD\f@ctech(\v@lX,\v@lY,\v@lZ)%
    \divide\v@lX\f@ctech\divide\v@lY\f@ctech\divide\v@lZ\f@ctech%
    \v@lXa=\repdecn@mb{\v@lX}\v@lX%
    \v@lYa=\repdecn@mb{\v@lY}\v@lY\v@lZa=\repdecn@mb{\v@lZ}\v@lZ%
    \advance\v@lXa\v@lYa\advance\v@lXa\v@lZa\sqrt@{\global\result@t}{\v@lXa}%
    \global\multiply\result@t\f@ctech}#1=\result@t}
\ctr@ln@m\vecunit@
\ctr@ld@f\def\vecunit@DD#1#2{{\Figg@tXY{#2}\getredf@ctDD\f@ctech(\v@lX,\v@lY)%
    \divide\v@lX\f@ctech\divide\v@lY\f@ctech%
    \Figv@ctCreg#1(\v@lX,\v@lY)\n@rmeucSV{\v@lYa}{#1}%
    \invers@{\v@lXa}{\v@lYa}\edef\v@lv@lXa{\repdecn@mb{\v@lXa}}%
    \v@lX=\v@lv@lXa\v@lX\v@lY=\v@lv@lXa\v@lY%
    \Figv@ctCreg#1(\v@lX,\v@lY)\multiply\v@lYa\f@ctech\global\result@t=\v@lYa}}
\ctr@ld@f\def\vecunit@TD#1#2{{\Figg@tXY{#2}\getredf@ctTD\f@ctech(\v@lX,\v@lY,\v@lZ)%
    \divide\v@lX\f@ctech\divide\v@lY\f@ctech\divide\v@lZ\f@ctech%
    \Figv@ctCreg#1(\v@lX,\v@lY,\v@lZ)\n@rmeucSV{\v@lYa}{#1}%
    \invers@{\v@lXa}{\v@lYa}\edef\v@lv@lXa{\repdecn@mb{\v@lXa}}%
    \v@lX=\v@lv@lXa\v@lX\v@lY=\v@lv@lXa\v@lY\v@lZ=\v@lv@lXa\v@lZ%
    \Figv@ctCreg#1(\v@lX,\v@lY,\v@lZ)\multiply\v@lYa\f@ctech\global\result@t=\v@lYa}}
\ctr@ld@f\def\vecunitC@TD[#1,#2]{\Figg@tXYa{#1}\Figg@tXY{#2}%
    \advance\v@lX-\v@lXa\advance\v@lY-\v@lYa\advance\v@lZ-\v@lZa\c@lvecunitTD}
\ctr@ld@f\def\vecunitCV@TD#1{\Figg@tXY{#1}\c@lvecunitTD}
\ctr@ld@f\def\c@lvecunitTD{\getredf@ctTD\f@ctech(\v@lX,\v@lY,\v@lZ)%
    \divide\v@lX\f@ctech\divide\v@lY\f@ctech\divide\v@lZ\f@ctech%
    \v@lXa=\repdecn@mb{\v@lX}\v@lX%
    \v@lYa=\repdecn@mb{\v@lY}\v@lY\v@lZa=\repdecn@mb{\v@lZ}\v@lZ%
    \advance\v@lXa\v@lYa\advance\v@lXa\v@lZa\sqrt@{\v@lYa}{\v@lXa}%
    \invers@{\v@lXa}{\v@lYa}\edef\v@lv@lXa{\repdecn@mb{\v@lXa}}%
    \v@lX=\v@lv@lXa\v@lX\v@lY=\v@lv@lXa\v@lY\v@lZ=\v@lv@lXa\v@lZ}
\ctr@ln@m\figgetangle
\ctr@ld@f\def\figgetangleDD#1[#2,#3,#4]{\ifps@cri{\s@uvc@ntr@l\et@tfiggetangleDD\setc@ntr@l{2}%
    \figvectPDD-1[#2,#3]\figvectPDD-2[#2,#4]\vecunit@{-1}{-1}%
    \c@lproscalDD\delt@[-2,-1]\figvectNVDD-1[-1]\c@lproscalDD\v@leur[-2,-1]%
    \arct@n\v@lmax(\delt@,\v@leur)\v@lmax=\rdT@deg\v@lmax%
    \ifdim\v@lmax<\z@\advance\v@lmax\DePI@deg\fi\xdef#1{\repdecn@mb{\v@lmax}}%
    \resetc@ntr@l\et@tfiggetangleDD}\ignorespaces\fi}
\ctr@ld@f\def\figgetangleTD#1[#2,#3,#4,#5]{\ifps@cri{\s@uvc@ntr@l\et@tfiggetangleTD\setc@ntr@l{2}%
    \figvectPTD-1[#2,#3]\figvectPTD-2[#2,#5]\figvectNVTD-3[-1,-2]%
    \figvectPTD-2[#2,#4]\figvectNVTD-4[-3,-1]%
    \vecunit@{-1}{-1}\c@lproscalTD\delt@[-2,-1]\c@lproscalTD\v@leur[-2,-4]%
    \arct@n\v@lmax(\delt@,\v@leur)\v@lmax=\rdT@deg\v@lmax%
    \ifdim\v@lmax<\z@\advance\v@lmax\DePI@deg\fi\xdef#1{\repdecn@mb{\v@lmax}}%
    \resetc@ntr@l\et@tfiggetangleTD}\ignorespaces\fi}
\ctr@ld@f\def\figgetdist#1[#2,#3]{\ifps@cri{\s@uvc@ntr@l\et@tfiggetdist\setc@ntr@l{2}%
    \figvectP-1[#2,#3]\n@rmeuc{\v@lX}{-1}\v@lX=\ptT@unit@\v@lX\xdef#1{\repdecn@mb{\v@lX}}%
    \resetc@ntr@l\et@tfiggetdist}\ignorespaces\fi}
\ctr@ld@f\def\Figg@tT#1{\c@ntr@lnum{#1}%
    {\expandafter\expandafter\expandafter\extr@ctT\csname\objc@de\endcsname:%
     \ifnum\B@@ltxt=\z@\ptn@me{#1}\else\csname\objc@de T\endcsname\fi}}
\ctr@ld@f\def\extr@ctT#1,#2,#3/#4:{\def\B@@ltxt{#3}}
\ctr@ld@f\def\Figg@tXY#1{\c@ntr@lnum{#1}%
    \expandafter\expandafter\expandafter\extr@ctC\csname\objc@de\endcsname:}
\ctr@ln@m\extr@ctC
\ctr@ld@f\def\extr@ctCDD#1/#2,#3,#4:{\v@lX=#2\v@lY=#3}
\ctr@ld@f\def\extr@ctCTD#1/#2,#3,#4:{\v@lX=#2\v@lY=#3\v@lZ=#4}
\ctr@ld@f\def\Figg@tXYa#1{\c@ntr@lnum{#1}%
    \expandafter\expandafter\expandafter\extr@ctCa\csname\objc@de\endcsname:}
\ctr@ln@m\extr@ctCa
\ctr@ld@f\def\extr@ctCaDD#1/#2,#3,#4:{\v@lXa=#2\v@lYa=#3}
\ctr@ld@f\def\extr@ctCaTD#1/#2,#3,#4:{\v@lXa=#2\v@lYa=#3\v@lZa=#4}
\ctr@ln@m\t@xt@
\ctr@ld@f\def\figinit#1{\t@stc@tcodech@nge\initpr@lim\Figinit@#1,:\initpss@ttings\ignorespaces}
\ctr@ld@f\def\Figinit@#1,#2:{\setunit@{#1}\def\t@xt@{#2}\ifx\t@xt@\empty\else\Figinit@@#2:\fi}
\ctr@ld@f\def\Figinit@@#1#2:{\if#12 \else\Figs@tproj{#1}\initTD@\fi}
\ctr@ln@w{newif}\ifTr@isDim
\ctr@ld@f\def\UnD@fined{UNDEFINED}
\ctr@ld@f\def\ifundefined#1{\expandafter\ifx\csname#1\endcsname\relax}
\ctr@ln@m\@utoFN
\ctr@ln@m\@utoFInDone
\ctr@ln@m\disob@unit
\ctr@ld@f\def\initpr@lim{\initb@undb@x\figsetmark{}\figsetptname{$A_{##1}$}\def\Sc@leFact{1}%
    \initDD@\figsetroundcoord{yes}\ps@critrue\expandafter\setupd@te\defaultupdate:%
    \edef\disob@unit{\UnD@fined}\edef\t@rgetpt{\UnD@fined}\gdef\@utoFInDone{1}\gdef\@utoFN{0}}
\ctr@ld@f\def\initDD@{\Tr@isDimfalse%
    \ifPDFm@ke%
     \let\Ps@rcerc=\Ps@rcercBz%
     \let\Ps@rell=\Ps@rellBz%
    \fi
    \let\c@lDCUn=\c@lDCUnDD%
    \let\c@lDCDeux=\c@lDCDeuxDD%
    \let\c@ldefproj=\relax%
    \let\c@lproscal=\c@lproscalDD%
    \let\c@lprojSP=\relax%
    \let\extr@ctC=\extr@ctCDD%
    \let\extr@ctCa=\extr@ctCaDD%
    \let\extr@ctCF=\extr@ctCFDD%
    \let\Figp@intreg=\Figp@intregDD%
    \let\Figpts@xes=\Figpts@xesDD%
    \let\n@rmeucSV=\n@rmeucSVDD\let\n@rmeuc=\n@rmeucDD\let\n@rminf=\n@rminfDD%
    \let\pr@dMatV=\pr@dMatVDD%
    \let\ps@xes=\ps@xesDD%
    \let\vecunit@=\vecunit@DD%
    \let\figcoord=\figcoordDD%
    \let\figgetangle=\figgetangleDD%
    \let\figpt=\figptDD%
    \let\figptBezier=\figptBezierDD%
    \let\figptbary=\figptbaryDD%
    \let\figptcirc=\figptcircDD%
    \let\figptcircumcenter=\figptcircumcenterDD%
    \let\figptcopy=\figptcopyDD%
    \let\figptcurvcenter=\figptcurvcenterDD%
    \let\figptell=\figptellDD%
    \let\figptendnormal=\figptendnormalDD%
    \let\figptinterlineplane=\figptinterlineplaneDD%
    \let\figptinterlines=\inters@cDD%
    \let\figptorthocenter=\figptorthocenterDD%
    \let\figptorthoprojline=\figptorthoprojlineDD%
    \let\figptorthoprojplane=\figptorthoprojplaneDD%
    \let\figptrot=\figptrotDD%
    \let\figptscontrol=\figptscontrolDD%
    \let\figptsintercirc=\figptsintercircDD%
    \let\figptsinterlinell=\figptsinterlinellDD%
    \let\figptsorthoprojline=\figptsorthoprojlineDD%
    \let\figptorthoprojplane=\figptorthoprojplaneDD%
    \let\figptsrot=\figptsrotDD%
    \let\figptssym=\figptssymDD%
    \let\figptstra=\figptstraDD%
    \let\figptsym=\figptsymDD%
    \let\figpttraC=\figpttraCDD%
    \let\figpttra=\figpttraDD%
    \let\figptvisilimSL=\figptvisilimSLDD%
    \let\figsetobdist=\figsetobdistDD%
    \let\figsettarget=\figsettargetDD%
    \let\figsetview=\figsetviewDD%
    \let\figvectDBezier=\figvectDBezierDD%
    \let\figvectN=\figvectNDD%
    \let\figvectNV=\figvectNVDD%
    \let\figvectP=\figvectPDD%
    \let\figvectU=\figvectUDD%
    \let\psarccircP=\psarccircPDD%
    \let\psarccirc=\psarccircDD%
    \let\psarcell=\psarcellDD%
    \let\psarcellPA=\psarcellPADD%
    \let\psarrowBezier=\psarrowBezierDD%
    \let\psarrowcircP=\psarrowcircPDD%
    \let\psarrowcirc=\psarrowcircDD%
    \let\psarrowhead=\psarrowheadDD%
    \let\psarrow=\psarrowDD%
    \let\psBezier=\psBezierDD%
    \let\pscirc=\pscircDD%
    \let\pscurve=\pscurveDD%
    \let\psnormal=\psnormalDD%
    }
\ctr@ld@f\def\initTD@{\Tr@isDimtrue\initb@undb@xTD\newt@rgetptfalse\newdis@bfalse%
    \let\c@lDCUn=\c@lDCUnTD%
    \let\c@lDCDeux=\c@lDCDeuxTD%
    \let\c@ldefproj=\c@ldefprojTD%
    \let\c@lproscal=\c@lproscalTD%
    \let\extr@ctC=\extr@ctCTD%
    \let\extr@ctCa=\extr@ctCaTD%
    \let\extr@ctCF=\extr@ctCFTD%
    \let\Figp@intreg=\Figp@intregTD%
    \let\Figpts@xes=\Figpts@xesTD%
    \let\n@rmeucSV=\n@rmeucSVTD\let\n@rmeuc=\n@rmeucTD\let\n@rminf=\n@rminfTD%
    \let\pr@dMatV=\pr@dMatVTD%
    \let\ps@xes=\ps@xesTD%
    \let\vecunit@=\vecunit@TD%
    \let\figcoord=\figcoordTD%
    \let\figgetangle=\figgetangleTD%
    \let\figpt=\figptTD%
    \let\figptBezier=\figptBezierTD%
    \let\figptbary=\figptbaryTD%
    \let\figptcirc=\figptcircTD%
    \let\figptcircumcenter=\figptcircumcenterTD%
    \let\figptcopy=\figptcopyTD%
    \let\figptcurvcenter=\figptcurvcenterTD%
    \let\figptinterlineplane=\figptinterlineplaneTD%
    \let\figptinterlines=\inters@cTD%
    \let\figptorthocenter=\figptorthocenterTD%
    \let\figptorthoprojline=\figptorthoprojlineTD%
    \let\figptorthoprojplane=\figptorthoprojplaneTD%
    \let\figptrot=\figptrotTD%
    \let\figptscontrol=\figptscontrolTD%
    \let\figptsintercirc=\figptsintercircTD%
    \let\figptsorthoprojline=\figptsorthoprojlineTD%
    \let\figptsorthoprojplane=\figptsorthoprojplaneTD%
    \let\figptsrot=\figptsrotTD%
    \let\figptssym=\figptssymTD%
    \let\figptstra=\figptstraTD%
    \let\figptsym=\figptsymTD%
    \let\figpttraC=\figpttraCTD%
    \let\figpttra=\figpttraTD%
    \let\figptvisilimSL=\figptvisilimSLTD%
    \let\figsetobdist=\figsetobdistTD%
    \let\figsettarget=\figsettargetTD%
    \let\figsetview=\figsetviewTD%
    \let\figvectDBezier=\figvectDBezierTD%
    \let\figvectN=\figvectNTD%
    \let\figvectNV=\figvectNVTD%
    \let\figvectP=\figvectPTD%
    \let\figvectU=\figvectUTD%
    \let\psarccircP=\psarccircPTD%
    \let\psarccirc=\psarccircTD%
    \let\psarcell=\psarcellTD%
    \let\psarcellPA=\psarcellPATD%
    \let\psarrowBezier=\psarrowBezierTD%
    \let\psarrowcircP=\psarrowcircPTD%
    \let\psarrowcirc=\psarrowcircTD%
    \let\psarrowhead=\psarrowheadTD%
    \let\psarrow=\psarrowTD%
    \let\psBezier=\psBezierTD%
    \let\pscirc=\pscircTD%
    \let\pscurve=\pscurveTD%
    }
\ctr@ld@f\def\un@v@ilable#1{\immediate\write16{*** The macro #1 is not available in the current context.}}
\ctr@ld@f\def\figinsert#1{{\def\t@xt@{#1}\relax%
    \ifx\t@xt@\empty\ifnum\@utoFInDone>\z@\Figinsert@\DefGIfilen@me,:\fi%
    \else\expandafter\FiginsertNu@#1 :\fi}\ignorespaces}
\ctr@ld@f\def\FiginsertNu@#1 #2:{\def\t@xt@{#1}\relax\ifx\t@xt@\empty\def\t@xt@{#2}%
    \ifx\t@xt@\empty\ifnum\@utoFInDone>\z@\Figinsert@\DefGIfilen@me,:\fi%
    \else\FiginsertNu@#2:\fi\else\expandafter\FiginsertNd@#1 #2:\fi}
\ctr@ld@f\def\FiginsertNd@#1#2:{\ifcat#1a\Figinsert@#1#2,:\else%
    \ifnum\@utoFInDone>\z@\Figinsert@\DefGIfilen@me,#1#2,:\fi\fi}
\ctr@ln@m\Sc@leFact
\ctr@ld@f\def\Figinsert@#1,#2:{\def\t@xt@{#2}\ifx\t@xt@\empty\xdef\Sc@leFact{1}\else%
    \X@rgdeux@#2\xdef\Sc@leFact{\@rgdeux}\fi%
    \Figdisc@rdLTS{#1}{\t@xt@}\@psfgetbb{\t@xt@}%
    \v@lX=\@psfllx\p@\v@lX=\ptpsT@pt\v@lX\v@lX=\Sc@leFact\v@lX%
    \v@lY=\@psflly\p@\v@lY=\ptpsT@pt\v@lY\v@lY=\Sc@leFact\v@lY%
    \b@undb@x{\v@lX}{\v@lY}%
    \v@lX=\@psfurx\p@\v@lX=\ptpsT@pt\v@lX\v@lX=\Sc@leFact\v@lX%
    \v@lY=\@psfury\p@\v@lY=\ptpsT@pt\v@lY\v@lY=\Sc@leFact\v@lY%
    \b@undb@x{\v@lX}{\v@lY}%
    \ifPDFm@ke\Figinclud@PDF{\t@xt@}{\Sc@leFact}\else%
    \v@lX=\c@nt pt\v@lX=\Sc@leFact\v@lX\edef\F@ct{\repdecn@mb{\v@lX}}%
    \ifx\TeXturesonMacOSltX\special{postscriptfile #1 vscale=\F@ct\space hscale=\F@ct}%
    \else\includegraphics{#1}\fi\fi%
    \message{[\t@xt@]}\ignorespaces}
\ctr@ld@f\def\Figdisc@rdLTS#1#2{\expandafter\Figdisc@rdLTS@#1 :#2}
\ctr@ld@f\def\Figdisc@rdLTS@#1 #2:#3{\def#3{#1}\relax\ifx#3\empty\expandafter\Figdisc@rdLTS@#2:#3\fi}
\ctr@ld@f\def\figinsertE#1{\FiginsertE@#1,:\ignorespaces}
\ctr@ld@f\def\FiginsertE@#1,#2:{{\def\t@xt@{#2}\ifx\t@xt@\empty\xdef\Sc@leFact{1}\else%
    \X@rgdeux@#2\xdef\Sc@leFact{\@rgdeux}\fi%
    \Figdisc@rdLTS{#1}{\t@xt@}\pdfximage{\t@xt@}%
    \setbox\Gb@x=\hbox{\pdfrefximage\pdflastximage}%
    \v@lX=\z@\v@lY=-\Sc@leFact\dp\Gb@x\b@undb@x{\v@lX}{\v@lY}%
    \advance\v@lX\Sc@leFact\wd\Gb@x\advance\v@lY\Sc@leFact\dp\Gb@x%
    \advance\v@lY\Sc@leFact\ht\Gb@x\b@undb@x{\v@lX}{\v@lY}%
    \v@lX=\Sc@leFact\wd\Gb@x\pdfximage width \v@lX {\t@xt@}%
    \rlap{\pdfrefximage\pdflastximage}\message{[\t@xt@]}}\ignorespaces}
\ctr@ld@f\def\X@rgdeux@#1,{\edef\@rgdeux{#1}}
\ctr@ln@m\figpt
\ctr@ld@f\def\figptDD#1:#2(#3,#4){\ifps@cri\c@ntr@lnum{#1}%
    {\v@lX=#3\unit@\v@lY=#4\unit@\Fig@dmpt{#2}{\z@}}\ignorespaces\fi}
\ctr@ld@f\def\Fig@dmpt#1#2{\def\t@xt@{#1}\ifx\t@xt@\empty\def\B@@ltxt{\z@}%
    \else\expandafter\gdef\csname\objc@de T\endcsname{#1}\def\B@@ltxt{\@ne}\fi%
    \expandafter\xdef\csname\objc@de\endcsname{\ifitis@vect@r\C@dCl@svect%
    \else\C@dCl@spt\fi,\z@,\B@@ltxt/\the\v@lX,\the\v@lY,#2}}
\ctr@ld@f\def\C@dCl@spt{P}
\ctr@ld@f\def\C@dCl@svect{V}
\ctr@ln@m\c@@rdYZ
\ctr@ln@m\c@@rdY
\ctr@ld@f\def\figptTD#1:#2(#3,#4){\ifps@cri\c@ntr@lnum{#1}%
    \def\c@@rdYZ{#4,0,0}\extrairelepremi@r\c@@rdY\de\c@@rdYZ%
    \extrairelepremi@r\c@@rdZ\de\c@@rdYZ%
    {\v@lX=#3\unit@\v@lY=\c@@rdY\unit@\v@lZ=\c@@rdZ\unit@\Fig@dmpt{#2}{\the\v@lZ}%
    \b@undb@xTD{\v@lX}{\v@lY}{\v@lZ}}\ignorespaces\fi}
\ctr@ln@m\Figp@intreg
\ctr@ld@f\def\Figp@intregDD#1:#2(#3,#4){\c@ntr@lnum{#1}%
    {\result@t=#4\v@lX=#3\v@lY=\result@t\Fig@dmpt{#2}{\z@}}\ignorespaces}
\ctr@ld@f\def\Figp@intregTD#1:#2(#3,#4){\c@ntr@lnum{#1}%
    \def\c@@rdYZ{#4,\z@,\z@}\extrairelepremi@r\c@@rdY\de\c@@rdYZ%
    \extrairelepremi@r\c@@rdZ\de\c@@rdYZ%
    {\v@lX=#3\v@lY=\c@@rdY\v@lZ=\c@@rdZ\Fig@dmpt{#2}{\the\v@lZ}%
    \b@undb@xTD{\v@lX}{\v@lY}{\v@lZ}}\ignorespaces}
\ctr@ln@m\figptBezier
\ctr@ld@f\def\figptBezierDD#1:#2:#3[#4,#5,#6,#7]{\ifps@cri{\s@uvc@ntr@l\et@tfigptBezierDD%
    \FigptBezier@#3[#4,#5,#6,#7]\Figp@intregDD#1:{#2}(\v@lX,\v@lY)%
    \resetc@ntr@l\et@tfigptBezierDD}\ignorespaces\fi}
\ctr@ld@f\def\figptBezierTD#1:#2:#3[#4,#5,#6,#7]{\ifps@cri{\s@uvc@ntr@l\et@tfigptBezierTD%
    \FigptBezier@#3[#4,#5,#6,#7]\Figp@intregTD#1:{#2}(\v@lX,\v@lY,\v@lZ)%
    \resetc@ntr@l\et@tfigptBezierTD}\ignorespaces\fi}
\ctr@ld@f\def\FigptBezier@#1[#2,#3,#4,#5]{\setc@ntr@l{2}%
    \edef\T@{#1}\v@leur=\p@\advance\v@leur-#1pt\edef\UNmT@{\repdecn@mb{\v@leur}}%
    \figptcopy-4:/#2/\figptcopy-3:/#3/\figptcopy-2:/#4/\figptcopy-1:/#5/%
    \l@mbd@un=-4 \l@mbd@de=-\thr@@\p@rtent=\m@ne\c@lDecast%
    \l@mbd@un=-4 \l@mbd@de=-\thr@@\p@rtent=-\tw@\c@lDecast%
    \l@mbd@un=-4 \l@mbd@de=-\thr@@\p@rtent=-\thr@@\c@lDecast\Figg@tXY{-4}}
\ctr@ln@m\c@lDCUn
\ctr@ld@f\def\c@lDCUnDD#1#2{\Figg@tXY{#1}\v@lX=\UNmT@\v@lX\v@lY=\UNmT@\v@lY%
    \Figg@tXYa{#2}\advance\v@lX\T@\v@lXa\advance\v@lY\T@\v@lYa%
    \Figp@intregDD#1:(\v@lX,\v@lY)}
\ctr@ld@f\def\c@lDCUnTD#1#2{\Figg@tXY{#1}\v@lX=\UNmT@\v@lX\v@lY=\UNmT@\v@lY\v@lZ=\UNmT@\v@lZ%
    \Figg@tXYa{#2}\advance\v@lX\T@\v@lXa\advance\v@lY\T@\v@lYa\advance\v@lZ\T@\v@lZa%
    \Figp@intregTD#1:(\v@lX,\v@lY,\v@lZ)}
\ctr@ld@f\def\c@lDecast{\relax\ifnum\l@mbd@un<\p@rtent\c@lDCUn{\l@mbd@un}{\l@mbd@de}%
    \advance\l@mbd@un\@ne\advance\l@mbd@de\@ne\c@lDecast\fi}
\ctr@ld@f\def\figptmap#1:#2=#3/#4/#5/{\ifps@cri{\s@uvc@ntr@l\et@tfigptmap%
    \setc@ntr@l{2}\figvectP-1[#4,#3]\Figg@tXY{-1}%
    \pr@dMatV/#5/\figpttra#1:{#2}=#4/1,-1/%
    \resetc@ntr@l\et@tfigptmap}\ignorespaces\fi}
\ctr@ln@m\pr@dMatV
\ctr@ld@f\def\pr@dMatVDD/#1,#2;#3,#4/{\v@lXa=#1\v@lX\advance\v@lXa#2\v@lY%
    \v@lYa=#3\v@lX\advance\v@lYa#4\v@lY\Figv@ctCreg-1(\v@lXa,\v@lYa)}
\ctr@ld@f\def\pr@dMatVTD/#1,#2,#3;#4,#5,#6;#7,#8,#9/{%
    \v@lXa=#1\v@lX\advance\v@lXa#2\v@lY\advance\v@lXa#3\v@lZ%
    \v@lYa=#4\v@lX\advance\v@lYa#5\v@lY\advance\v@lYa#6\v@lZ%
    \v@lZa=#7\v@lX\advance\v@lZa#8\v@lY\advance\v@lZa#9\v@lZ%
    \Figv@ctCreg-1(\v@lXa,\v@lYa,\v@lZa)}
\ctr@ln@m\figptbary
\ctr@ld@f\def\figptbaryDD#1:#2[#3;#4]{\ifps@cri{\edef\list@num{#3}\extrairelepremi@r\p@int\de\list@num%
    \s@mme=\z@\@ecfor\c@ef:=#4\do{\advance\s@mme\c@ef}%
    \edef\listec@ef{#4,0}\extrairelepremi@r\c@ef\de\listec@ef%
    \Figg@tXY{\p@int}\divide\v@lX\s@mme\divide\v@lY\s@mme%
    \multiply\v@lX\c@ef\multiply\v@lY\c@ef%
    \@ecfor\p@int:=\list@num\do{\extrairelepremi@r\c@ef\de\listec@ef%
           \Figg@tXYa{\p@int}\divide\v@lXa\s@mme\divide\v@lYa\s@mme%
           \multiply\v@lXa\c@ef\multiply\v@lYa\c@ef%
           \advance\v@lX\v@lXa\advance\v@lY\v@lYa}%
    \Figp@intregDD#1:{#2}(\v@lX,\v@lY)}\ignorespaces\fi}
\ctr@ld@f\def\figptbaryTD#1:#2[#3;#4]{\ifps@cri{\edef\list@num{#3}\extrairelepremi@r\p@int\de\list@num%
    \s@mme=\z@\@ecfor\c@ef:=#4\do{\advance\s@mme\c@ef}%
    \edef\listec@ef{#4,0}\extrairelepremi@r\c@ef\de\listec@ef%
    \Figg@tXY{\p@int}\divide\v@lX\s@mme\divide\v@lY\s@mme\divide\v@lZ\s@mme%
    \multiply\v@lX\c@ef\multiply\v@lY\c@ef\multiply\v@lZ\c@ef%
    \@ecfor\p@int:=\list@num\do{\extrairelepremi@r\c@ef\de\listec@ef%
           \Figg@tXYa{\p@int}\divide\v@lXa\s@mme\divide\v@lYa\s@mme\divide\v@lZa\s@mme%
           \multiply\v@lXa\c@ef\multiply\v@lYa\c@ef\multiply\v@lZa\c@ef%
           \advance\v@lX\v@lXa\advance\v@lY\v@lYa\advance\v@lZ\v@lZa}%
    \Figp@intregTD#1:{#2}(\v@lX,\v@lY,\v@lZ)}\ignorespaces\fi}
\ctr@ld@f\def\figptbaryR#1:#2[#3;#4]{\ifps@cri{%
    \v@leur=\z@\@ecfor\c@ef:=#4\do{\maxim@m{\v@lmax}{\c@ef pt}{-\c@ef pt}%
    \ifdim\v@lmax>\v@leur\v@leur=\v@lmax\fi}%
    \ifdim\v@leur<\p@\f@ctech=\@M\else\ifdim\v@leur<\t@n\p@\f@ctech=\@m\else%
    \ifdim\v@leur<\c@nt\p@\f@ctech=\c@nt\else\ifdim\v@leur<\@m\p@\f@ctech=\t@n\else%
    \f@ctech=\@ne\fi\fi\fi\fi%
    \def\listec@ef{0}%
    \@ecfor\c@ef:=#4\do{\sc@lec@nvRI{\c@ef pt}\edef\listec@ef{\listec@ef,\the\s@mme}}%
    \extrairelepremi@r\c@ef\de\listec@ef\figptbary#1:#2[#3;\listec@ef]}\ignorespaces\fi}
\ctr@ld@f\def\sc@lec@nvRI#1{\v@leur=#1\p@rtentiere{\s@mme}{\v@leur}\advance\v@leur-\s@mme\p@%
    \multiply\v@leur\f@ctech\p@rtentiere{\p@rtent}{\v@leur}%
    \multiply\s@mme\f@ctech\advance\s@mme\p@rtent}
\ctr@ln@m\figptcirc
\ctr@ld@f\def\figptcircDD#1:#2:#3;#4(#5){\ifps@cri{\s@uvc@ntr@l\et@tfigptcircDD%
    \c@lptellDD#1:{#2}:#3;#4,#4(#5)\resetc@ntr@l\et@tfigptcircDD}\ignorespaces\fi}
\ctr@ld@f\def\figptcircTD#1:#2:#3,#4,#5;#6(#7){\ifps@cri{\s@uvc@ntr@l\et@tfigptcircTD%
    \setc@ntr@l{2}\c@lExtAxes#3,#4,#5(#6)\figptellP#1:{#2}:#3,-4,-5(#7)%
    \resetc@ntr@l\et@tfigptcircTD}\ignorespaces\fi}
\ctr@ln@m\figptcircumcenter
\ctr@ld@f\def\figptcircumcenterDD#1:#2[#3,#4,#5]{\ifps@cri{\s@uvc@ntr@l\et@tfigptcircumcenterDD%
    \setc@ntr@l{2}\figvectNDD-5[#3,#4]\figptbaryDD-3:[#3,#4;1,1]%
                  \figvectNDD-6[#4,#5]\figptbaryDD-4:[#4,#5;1,1]%
    \resetc@ntr@l{2}\inters@cDD#1:{#2}[-3,-5;-4,-6]%
    \resetc@ntr@l\et@tfigptcircumcenterDD}\ignorespaces\fi}
\ctr@ld@f\def\figptcircumcenterTD#1:#2[#3,#4,#5]{\ifps@cri{\s@uvc@ntr@l\et@tfigptcircumcenterTD%
    \setc@ntr@l{2}\figvectNTD-1[#3,#4,#5]%
    \figvectPTD-3[#3,#4]\figvectNVTD-5[-1,-3]\figptbaryTD-3:[#3,#4;1,1]%
    \figvectPTD-4[#4,#5]\figvectNVTD-6[-1,-4]\figptbaryTD-4:[#4,#5;1,1]%
    \resetc@ntr@l{2}\inters@cTD#1:{#2}[-3,-5;-4,-6]%
    \resetc@ntr@l\et@tfigptcircumcenterTD}\ignorespaces\fi}
\ctr@ln@m\figptcopy
\ctr@ld@f\def\figptcopyDD#1:#2/#3/{\ifps@cri{\Figg@tXY{#3}%
    \Figp@intregDD#1:{#2}(\v@lX,\v@lY)}\ignorespaces\fi}
\ctr@ld@f\def\figptcopyTD#1:#2/#3/{\ifps@cri{\Figg@tXY{#3}%
    \Figp@intregTD#1:{#2}(\v@lX,\v@lY,\v@lZ)}\ignorespaces\fi}
\ctr@ln@m\figptcurvcenter
\ctr@ld@f\def\figptcurvcenterDD#1:#2:#3[#4,#5,#6,#7]{\ifps@cri{\s@uvc@ntr@l\et@tfigptcurvcenterDD%
    \setc@ntr@l{2}\c@lcurvradDD#3[#4,#5,#6,#7]\edef\Sprim@{\repdecn@mb{\result@t}}%
    \figptBezierDD-1::#3[#4,#5,#6,#7]\figpttraDD#1:{#2}=-1/\Sprim@,-5/%
    \resetc@ntr@l\et@tfigptcurvcenterDD}\ignorespaces\fi}
\ctr@ld@f\def\figptcurvcenterTD#1:#2:#3[#4,#5,#6,#7]{\ifps@cri{\s@uvc@ntr@l\et@tfigptcurvcenterTD%
    \setc@ntr@l{2}\figvectDBezierTD -5:1,#3[#4,#5,#6,#7]%
    \figvectDBezierTD -6:2,#3[#4,#5,#6,#7]\vecunit@TD{-5}{-5}%
    \edef\Sprim@{\repdecn@mb{\result@t}}\figvectNVTD-1[-6,-5]%
    \figvectNVTD-5[-5,-1]\c@lproscalTD\v@leur[-6,-5]%
    \invers@{\v@leur}{\v@leur}\v@leur=\Sprim@\v@leur\v@leur=\Sprim@\v@leur%
    \figptBezierTD-1::#3[#4,#5,#6,#7]\edef\Sprim@{\repdecn@mb{\v@leur}}%
    \figpttraTD#1:{#2}=-1/\Sprim@,-5/\resetc@ntr@l\et@tfigptcurvcenterTD}\ignorespaces\fi}
\ctr@ld@f\def\c@lcurvradDD#1[#2,#3,#4,#5]{{\figvectDBezierDD -5:1,#1[#2,#3,#4,#5]%
    \figvectDBezierDD -6:2,#1[#2,#3,#4,#5]\vecunit@DD{-5}{-5}%
    \edef\Sprim@{\repdecn@mb{\result@t}}\figvectNVDD-5[-5]\c@lproscalDD\v@leur[-6,-5]%
    \invers@{\v@leur}{\v@leur}\v@leur=\Sprim@\v@leur\v@leur=\Sprim@\v@leur%
    \global\result@t=\v@leur}}
\ctr@ln@m\figptell
\ctr@ld@f\def\figptellDD#1:#2:#3;#4,#5(#6,#7){\ifps@cri{\s@uvc@ntr@l\et@tfigptell%
    \c@lptellDD#1::#3;#4,#5(#6)\figptrotDD#1:{#2}=#1/#3,#7/%
    \resetc@ntr@l\et@tfigptell}\ignorespaces\fi}
\ctr@ld@f\def\c@lptellDD#1:#2:#3;#4,#5(#6){\c@ssin{\C@}{\S@}{#6}\v@lmin=\C@ pt\v@lmax=\S@ pt%
    \v@lmin=#4\v@lmin\v@lmax=#5\v@lmax%
    \edef\Xc@mp{\repdecn@mb{\v@lmin}}\edef\Yc@mp{\repdecn@mb{\v@lmax}}%
    \setc@ntr@l{2}\figvectC-1(\Xc@mp,\Yc@mp)\figpttraDD#1:{#2}=#3/1,-1/}
\ctr@ld@f\def\figptellP#1:#2:#3,#4,#5(#6){\ifps@cri{\s@uvc@ntr@l\et@tfigptellP%
    \setc@ntr@l{2}\figvectP-1[#3,#4]\figvectP-2[#3,#5]%
    \v@leur=#6pt\c@lptellP{#3}{-1}{-2}\figptcopy#1:{#2}/-3/%
    \resetc@ntr@l\et@tfigptellP}\ignorespaces\fi}
\ctr@ln@m\@ngle
\ctr@ld@f\def\c@lptellP#1#2#3{\edef\@ngle{\repdecn@mb\v@leur}\c@ssin{\C@}{\S@}{\@ngle}%
    \figpttra-3:=#1/\C@,#2/\figpttra-3:=-3/\S@,#3/}
\ctr@ln@m\figptendnormal
\ctr@ld@f\def\figptendnormalDD#1:#2:#3,#4[#5,#6]{\ifps@cri{\s@uvc@ntr@l\et@tfigptendnormal%
    \Figg@tXYa{#5}\Figg@tXY{#6}%
    \advance\v@lX-\v@lXa\advance\v@lY-\v@lYa%
    \setc@ntr@l{2}\Figv@ctCreg-1(\v@lX,\v@lY)\vecunit@{-1}{-1}\Figg@tXY{-1}%
    \delt@=#3\unit@\maxim@m{\delt@}{\delt@}{-\delt@}\edef\l@ngueur{\repdecn@mb{\delt@}}%
    \v@lX=\l@ngueur\v@lX\v@lY=\l@ngueur\v@lY%
    \delt@=\p@\advance\delt@-#4pt\edef\l@ngueur{\repdecn@mb{\delt@}}%
    \figptbaryR-1:[#5,#6;#4,\l@ngueur]\Figg@tXYa{-1}%
    \advance\v@lXa\v@lY\advance\v@lYa-\v@lX%
    \setc@ntr@l{1}\Figp@intregDD#1:{#2}(\v@lXa,\v@lYa)\resetc@ntr@l\et@tfigptendnormal}%
    \ignorespaces\fi}
\ctr@ld@f\def\figptexcenter#1:#2[#3,#4,#5]{\ifps@cri{\let@xte={-}%
    \Figptexinsc@nter#1:#2[#3,#4,#5]}\ignorespaces\fi}
\ctr@ld@f\def\figptincenter#1:#2[#3,#4,#5]{\ifps@cri{\let@xte={}%
    \Figptexinsc@nter#1:#2[#3,#4,#5]}\ignorespaces\fi}
\ctr@ld@f\let\figptinscribedcenter=\figptincenter
\ctr@ld@f\def\Figptexinsc@nter#1:#2[#3,#4,#5]{%
    \figgetdist\LA@[#4,#5]\figgetdist\LB@[#3,#5]\figgetdist\LC@[#3,#4]%
    \figptbaryR#1:{#2}[#3,#4,#5;\the\let@xte\LA@,\LB@,\LC@]}
\ctr@ln@m\figptinterlineplane
\ctr@ld@f\def\figptinterlineplaneDD{\un@v@ilable{figptinterlineplane}}
\ctr@ld@f\def\figptinterlineplaneTD#1:#2[#3,#4;#5,#6]{\ifps@cri{\s@uvc@ntr@l\et@tfigptinterlineplane%
    \setc@ntr@l{2}\figvectPTD-1[#3,#5]\vecunit@TD{-2}{#6}%
    \r@pPSTD\v@leur[-2,-1,#4]\edef\v@lcoef{\repdecn@mb{\v@leur}}%
    \figpttraTD#1:{#2}=#3/\v@lcoef,#4/\resetc@ntr@l\et@tfigptinterlineplane}\ignorespaces\fi}
\ctr@ln@m\figptorthocenter
\ctr@ld@f\def\figptorthocenterDD#1:#2[#3,#4,#5]{\ifps@cri{\s@uvc@ntr@l\et@tfigptorthocenterDD%
    \setc@ntr@l{2}\figvectNDD-3[#3,#4]\figvectNDD-4[#4,#5]%
    \resetc@ntr@l{2}\inters@cDD#1:{#2}[#5,-3;#3,-4]%
    \resetc@ntr@l\et@tfigptorthocenterDD}\ignorespaces\fi}
\ctr@ld@f\def\figptorthocenterTD#1:#2[#3,#4,#5]{\ifps@cri{\s@uvc@ntr@l\et@tfigptorthocenterTD%
    \setc@ntr@l{2}\figvectNTD-1[#3,#4,#5]%
    \figvectPTD-2[#3,#4]\figvectNVTD-3[-1,-2]%
    \figvectPTD-2[#4,#5]\figvectNVTD-4[-1,-2]%
    \resetc@ntr@l{2}\inters@cTD#1:{#2}[#5,-3;#3,-4]%
    \resetc@ntr@l\et@tfigptorthocenterTD}\ignorespaces\fi}
\ctr@ln@m\figptorthoprojline
\ctr@ld@f\def\figptorthoprojlineDD#1:#2=#3/#4,#5/{\ifps@cri{\s@uvc@ntr@l\et@tfigptorthoprojlineDD%
    \setc@ntr@l{2}\figvectPDD-3[#4,#5]\figvectNVDD-4[-3]\resetc@ntr@l{2}%
    \inters@cDD#1:{#2}[#3,-4;#4,-3]\resetc@ntr@l\et@tfigptorthoprojlineDD}\ignorespaces\fi}
\ctr@ld@f\def\figptorthoprojlineTD#1:#2=#3/#4,#5/{\ifps@cri{\s@uvc@ntr@l\et@tfigptorthoprojlineTD%
    \setc@ntr@l{2}\figvectPTD-1[#4,#3]\figvectPTD-2[#4,#5]\vecunit@TD{-2}{-2}%
    \c@lproscalTD\v@leur[-1,-2]\edef\v@lcoef{\repdecn@mb{\v@leur}}%
    \figpttraTD#1:{#2}=#4/\v@lcoef,-2/\resetc@ntr@l\et@tfigptorthoprojlineTD}\ignorespaces\fi}
\ctr@ln@m\figptorthoprojplane
\ctr@ld@f\def\figptorthoprojplaneDD{\un@v@ilable{figptorthoprojplane}}
\ctr@ld@f\def\figptorthoprojplaneTD#1:#2=#3/#4,#5/{\ifps@cri{\s@uvc@ntr@l\et@tfigptorthoprojplane%
    \setc@ntr@l{2}\figvectPTD-1[#3,#4]\vecunit@TD{-2}{#5}%
    \c@lproscalTD\v@leur[-1,-2]\edef\v@lcoef{\repdecn@mb{\v@leur}}%
    \figpttraTD#1:{#2}=#3/\v@lcoef,-2/\resetc@ntr@l\et@tfigptorthoprojplane}\ignorespaces\fi}
\ctr@ld@f\def\figpthom#1:#2=#3/#4,#5/{\ifps@cri{\s@uvc@ntr@l\et@tfigpthom%
    \setc@ntr@l{2}\figvectP-1[#4,#3]\figpttra#1:{#2}=#4/#5,-1/%
    \resetc@ntr@l\et@tfigpthom}\ignorespaces\fi}
\ctr@ln@m\figptrot
\ctr@ld@f\def\figptrotDD#1:#2=#3/#4,#5/{\ifps@cri{\s@uvc@ntr@l\et@tfigptrotDD%
    \c@ssin{\C@}{\S@}{#5}\setc@ntr@l{2}\figvectPDD-1[#4,#3]\Figg@tXY{-1}%
    \v@lXa=\C@\v@lX\advance\v@lXa-\S@\v@lY%
    \v@lYa=\S@\v@lX\advance\v@lYa\C@\v@lY%
    \Figv@ctCreg-1(\v@lXa,\v@lYa)\figpttraDD#1:{#2}=#4/1,-1/%
    \resetc@ntr@l\et@tfigptrotDD}\ignorespaces\fi}
\ctr@ld@f\def\figptrotTD#1:#2=#3/#4,#5,#6/{\ifps@cri{\s@uvc@ntr@l\et@tfigptrotTD%
    \c@ssin{\C@}{\S@}{#5}%
    \setc@ntr@l{2}\figptorthoprojplaneTD-3:=#4/#3,#6/\figvectPTD-2[-3,#3]%
    \n@rmeucTD\v@leur{-2}\ifdim\v@leur<\Cepsil@n\Figg@tXYa{#3}\else%
    \edef\v@lcoef{\repdecn@mb{\v@leur}}\figvectNVTD-1[#6,-2]%
    \Figg@tXYa{-1}\v@lXa=\v@lcoef\v@lXa\v@lYa=\v@lcoef\v@lYa\v@lZa=\v@lcoef\v@lZa%
    \v@lXa=\S@\v@lXa\v@lYa=\S@\v@lYa\v@lZa=\S@\v@lZa\Figg@tXY{-2}%
    \advance\v@lXa\C@\v@lX\advance\v@lYa\C@\v@lY\advance\v@lZa\C@\v@lZ%
    \Figg@tXY{-3}\advance\v@lXa\v@lX\advance\v@lYa\v@lY\advance\v@lZa\v@lZ\fi%
    \Figp@intregTD#1:{#2}(\v@lXa,\v@lYa,\v@lZa)\resetc@ntr@l\et@tfigptrotTD}\ignorespaces\fi}
\ctr@ln@m\figptsym
\ctr@ld@f\def\figptsymDD#1:#2=#3/#4,#5/{\ifps@cri{\s@uvc@ntr@l\et@tfigptsymDD%
    \resetc@ntr@l{2}\figptorthoprojlineDD-5:=#3/#4,#5/\figvectPDD-2[#3,-5]%
    \figpttraDD#1:{#2}=#3/2,-2/\resetc@ntr@l\et@tfigptsymDD}\ignorespaces\fi}
\ctr@ld@f\def\figptsymTD#1:#2=#3/#4,#5/{\ifps@cri{\s@uvc@ntr@l\et@tfigptsymTD%
    \resetc@ntr@l{2}\figptorthoprojplaneTD-3:=#3/#4,#5/\figvectPTD-2[#3,-3]%
    \figpttraTD#1:{#2}=#3/2,-2/\resetc@ntr@l\et@tfigptsymTD}\ignorespaces\fi}
\ctr@ln@m\figpttra
\ctr@ld@f\def\figpttraDD#1:#2=#3/#4,#5/{\ifps@cri{\Figg@tXYa{#5}\v@lXa=#4\v@lXa\v@lYa=#4\v@lYa%
    \Figg@tXY{#3}\advance\v@lX\v@lXa\advance\v@lY\v@lYa%
    \Figp@intregDD#1:{#2}(\v@lX,\v@lY)}\ignorespaces\fi}
\ctr@ld@f\def\figpttraTD#1:#2=#3/#4,#5/{\ifps@cri{\Figg@tXYa{#5}\v@lXa=#4\v@lXa\v@lYa=#4\v@lYa%
    \v@lZa=#4\v@lZa\Figg@tXY{#3}\advance\v@lX\v@lXa\advance\v@lY\v@lYa%
    \advance\v@lZ\v@lZa\Figp@intregTD#1:{#2}(\v@lX,\v@lY,\v@lZ)}\ignorespaces\fi}
\ctr@ln@m\figpttraC
\ctr@ld@f\def\figpttraCDD#1:#2=#3/#4,#5/{\ifps@cri{\v@lXa=#4\unit@\v@lYa=#5\unit@%
    \Figg@tXY{#3}\advance\v@lX\v@lXa\advance\v@lY\v@lYa%
    \Figp@intregDD#1:{#2}(\v@lX,\v@lY)}\ignorespaces\fi}
\ctr@ld@f\def\figpttraCTD#1:#2=#3/#4,#5,#6/{\ifps@cri{\v@lXa=#4\unit@\v@lYa=#5\unit@\v@lZa=#6\unit@%
    \Figg@tXY{#3}\advance\v@lX\v@lXa\advance\v@lY\v@lYa\advance\v@lZ\v@lZa%
    \Figp@intregTD#1:{#2}(\v@lX,\v@lY,\v@lZ)}\ignorespaces\fi}
\ctr@ld@f\def\figptsaxes#1:#2(#3){\ifps@cri{\an@lys@xes#3,:\ifx\t@xt@\empty%
    \ifTr@isDim\Figpts@xes#1:#2(0,#3,0,#3,0,#3)\else\Figpts@xes#1:#2(0,#3,0,#3)\fi%
    \else\Figpts@xes#1:#2(#3)\fi}\ignorespaces\fi}
\ctr@ln@m\Figpts@xes
\ctr@ld@f\def\Figpts@xesDD#1:#2(#3,#4,#5,#6){%
    \s@mme=#1\figpttraC\the\s@mme:$x$=#2/#4,0/%
    \advance\s@mme\@ne\figpttraC\the\s@mme:$y$=#2/0,#6/}
\ctr@ld@f\def\Figpts@xesTD#1:#2(#3,#4,#5,#6,#7,#8){%
    \s@mme=#1\figpttraC\the\s@mme:$x$=#2/#4,0,0/%
    \advance\s@mme\@ne\figpttraC\the\s@mme:$y$=#2/0,#6,0/%
    \advance\s@mme\@ne\figpttraC\the\s@mme:$z$=#2/0,0,#8/}
\ctr@ld@f\def\figptsmap#1=#2/#3/#4/{\ifps@cri{\s@uvc@ntr@l\et@tfigptsmap%
    \setc@ntr@l{2}\def\list@num{#2}\s@mme=#1%
    \@ecfor\p@int:=\list@num\do{\figvectP-1[#3,\p@int]\Figg@tXY{-1}%
    \pr@dMatV/#4/\figpttra\the\s@mme:=#3/1,-1/\advance\s@mme\@ne}%
    \resetc@ntr@l\et@tfigptsmap}\ignorespaces\fi}
\ctr@ln@m\figptscontrol
\ctr@ld@f\def\figptscontrolDD#1[#2,#3,#4,#5]{\ifps@cri{\s@uvc@ntr@l\et@tfigptscontrolDD\setc@ntr@l{2}%
    \v@lX=\z@\v@lY=\z@\Figtr@nptDD{-5}{#2}\Figtr@nptDD{2}{#5}%
    \divide\v@lX\@vi\divide\v@lY\@vi%
    \Figtr@nptDD{3}{#3}\Figtr@nptDD{-1.5}{#4}\Figp@intregDD-1:(\v@lX,\v@lY)%
    \v@lX=\z@\v@lY=\z@\Figtr@nptDD{2}{#2}\Figtr@nptDD{-5}{#5}%
    \divide\v@lX\@vi\divide\v@lY\@vi\Figtr@nptDD{-1.5}{#3}\Figtr@nptDD{3}{#4}%
    \s@mme=#1\advance\s@mme\@ne\Figp@intregDD\the\s@mme:(\v@lX,\v@lY)%
    \figptcopyDD#1:/-1/\resetc@ntr@l\et@tfigptscontrolDD}\ignorespaces\fi}
\ctr@ld@f\def\figptscontrolTD#1[#2,#3,#4,#5]{\ifps@cri{\s@uvc@ntr@l\et@tfigptscontrolTD\setc@ntr@l{2}%
    \v@lX=\z@\v@lY=\z@\v@lZ=\z@\Figtr@nptTD{-5}{#2}\Figtr@nptTD{2}{#5}%
    \divide\v@lX\@vi\divide\v@lY\@vi\divide\v@lZ\@vi%
    \Figtr@nptTD{3}{#3}\Figtr@nptTD{-1.5}{#4}\Figp@intregTD-1:(\v@lX,\v@lY,\v@lZ)%
    \v@lX=\z@\v@lY=\z@\v@lZ=\z@\Figtr@nptTD{2}{#2}\Figtr@nptTD{-5}{#5}%
    \divide\v@lX\@vi\divide\v@lY\@vi\divide\v@lZ\@vi\Figtr@nptTD{-1.5}{#3}\Figtr@nptTD{3}{#4}%
    \s@mme=#1\advance\s@mme\@ne\Figp@intregTD\the\s@mme:(\v@lX,\v@lY,\v@lZ)%
    \figptcopyTD#1:/-1/\resetc@ntr@l\et@tfigptscontrolTD}\ignorespaces\fi}
\ctr@ld@f\def\Figtr@nptDD#1#2{\Figg@tXYa{#2}\v@lXa=#1\v@lXa\v@lYa=#1\v@lYa%
    \advance\v@lX\v@lXa\advance\v@lY\v@lYa}
\ctr@ld@f\def\Figtr@nptTD#1#2{\Figg@tXYa{#2}\v@lXa=#1\v@lXa\v@lYa=#1\v@lYa\v@lZa=#1\v@lZa%
    \advance\v@lX\v@lXa\advance\v@lY\v@lYa\advance\v@lZ\v@lZa}
\ctr@ld@f\def\figptscontrolcurve#1,#2[#3]{\ifps@cri{\s@uvc@ntr@l\et@tfigptscontrolcurve%
    \def\list@num{#3}\extrairelepremi@r\Ak@\de\list@num%
    \extrairelepremi@r\Ai@\de\list@num\extrairelepremi@r\Aj@\de\list@num%
    \s@mme=#1\figptcopy\the\s@mme:/\Ai@/%
    \setc@ntr@l{2}\figvectP -1[\Ak@,\Aj@]%
    \@ecfor\Ak@:=\list@num\do{\advance\s@mme\@ne\figpttra\the\s@mme:=\Ai@/\curv@roundness,-1/%
       \figvectP -1[\Ai@,\Ak@]\advance\s@mme\@ne\figpttra\the\s@mme:=\Aj@/-\curv@roundness,-1/%
       \advance\s@mme\@ne\figptcopy\the\s@mme:/\Aj@/%
       \edef\Ai@{\Aj@}\edef\Aj@{\Ak@}}\advance\s@mme-#1\divide\s@mme\thr@@%
       \xdef#2{\the\s@mme}%
    \resetc@ntr@l\et@tfigptscontrolcurve}\ignorespaces\fi}
\ctr@ln@m\figptsintercirc
\ctr@ld@f\def\figptsintercircDD#1[#2,#3;#4,#5]{\ifps@cri{\s@uvc@ntr@l\et@tfigptsintercircDD%
    \setc@ntr@l{2}\let\c@lNVintc=\c@lNVintcDD\Figptsintercirc@#1[#2,#3;#4,#5]%
    \resetc@ntr@l\et@tfigptsintercircDD}\ignorespaces\fi}
\ctr@ld@f\def\figptsintercircTD#1[#2,#3;#4,#5;#6]{\ifps@cri{\s@uvc@ntr@l\et@tfigptsintercircTD%
    \setc@ntr@l{2}\let\c@lNVintc=\c@lNVintcTD\vecunitC@TD[#2,#6]%
    \Figv@ctCreg-3(\v@lX,\v@lY,\v@lZ)\Figptsintercirc@#1[#2,#3;#4,#5]%
    \resetc@ntr@l\et@tfigptsintercircTD}\ignorespaces\fi}
\ctr@ld@f\def\Figptsintercirc@#1[#2,#3;#4,#5]{\figvectP-1[#2,#4]%
    \vecunit@{-1}{-1}\delt@=\result@t\f@ctech=\result@tent%
    \s@mme=#1\advance\s@mme\@ne\figptcopy#1:/#2/\figptcopy\the\s@mme:/#4/%
    \ifdim\delt@=\z@\else%
    \v@lmin=#3\unit@\v@lmax=#5\unit@\v@leur=\v@lmin\advance\v@leur\v@lmax%
    \ifdim\v@leur>\delt@%
    \v@leur=\v@lmin\advance\v@leur-\v@lmax\maxim@m{\v@leur}{\v@leur}{-\v@leur}%
    \ifdim\v@leur<\delt@%
    \divide\v@lmin\f@ctech\divide\v@lmax\f@ctech\divide\delt@\f@ctech%
    \v@lmin=\repdecn@mb{\v@lmin}\v@lmin\v@lmax=\repdecn@mb{\v@lmax}\v@lmax%
    \invers@{\v@leur}{\delt@}\advance\v@lmax-\v@lmin%
    \v@lmax=-\repdecn@mb{\v@leur}\v@lmax\advance\delt@\v@lmax\delt@=.5\delt@%
    \v@lmax=\delt@\multiply\v@lmax\f@ctech%
    \edef\t@ille{\repdecn@mb{\v@lmax}}\figpttra-2:=#2/\t@ille,-1/%
    \delt@=\repdecn@mb{\delt@}\delt@\advance\v@lmin-\delt@%
    \sqrt@{\v@leur}{\v@lmin}\multiply\v@leur\f@ctech\edef\t@ille{\repdecn@mb{\v@leur}}%
    \c@lNVintc\figpttra#1:=-2/-\t@ille,-1/\figpttra\the\s@mme:=-2/\t@ille,-1/\fi\fi\fi}
\ctr@ld@f\def\c@lNVintcDD{\Figg@tXY{-1}\Figv@ctCreg-1(-\v@lY,\v@lX)} 
\ctr@ld@f\def\c@lNVintcTD{{\Figg@tXY{-3}\v@lmin=\v@lX\v@lmax=\v@lY\v@leur=\v@lZ%
    \Figg@tXY{-1}\c@lprovec{-3}\vecunit@{-3}{-3}
    \Figg@tXY{-1}\v@lmin=\v@lX\v@lmax=\v@lY%
    \v@leur=\v@lZ\Figg@tXY{-3}\c@lprovec{-1}}} 
\ctr@ln@m\figptsinterlinell
\ctr@ld@f\def\figptsinterlinellDD#1[#2,#3,#4,#5;#6,#7]{\ifps@cri{\s@uvc@ntr@l\et@tfigptsinterlinellDD%
    \figptcopy#1:/#6/\s@mme=#1\advance\s@mme\@ne\figptcopy\the\s@mme:/#7/%
    \v@lmin=#3\unit@\v@lmax=#4\unit@
    \setc@ntr@l{2}\figptbaryDD-4:[#6,#7;1,1]\figptsrotDD-3=-4,#7/#2,-#5/
    \Figg@tXY{-3}\Figg@tXYa{#2}\advance\v@lX-\v@lXa\advance\v@lY-\v@lYa
    \figvectP-1[-3,-2]\Figg@tXYa{-1}\figvectP-3[-4,#7]\Figptsint@rLE{#1}
    \resetc@ntr@l\et@tfigptsinterlinellDD}\ignorespaces\fi}
\ctr@ld@f\def\figptsinterlinellP#1[#2,#3,#4;#5,#6]{\ifps@cri{\s@uvc@ntr@l\et@tfigptsinterlinellP%
    \figptcopy#1:/#5/\s@mme=#1\advance\s@mme\@ne\figptcopy\the\s@mme:/#6/\setc@ntr@l{2}%
    \figvectP-1[#2,#3]\vecunit@{-1}{-1}\v@lmin=\result@t
    \figvectP-2[#2,#4]\vecunit@{-2}{-2}\v@lmax=\result@t
    \figptbary-4:[#5,#6;1,1]
    \figvectP-3[#2,-4]\c@lproscal\v@lX[-3,-1]\c@lproscal\v@lY[-3,-2]
    \figvectP-3[-4,#6]\c@lproscal\v@lXa[-3,-1]\c@lproscal\v@lYa[-3,-2]
    \Figptsint@rLE{#1}\resetc@ntr@l\et@tfigptsinterlinellP}\ignorespaces\fi}
\ctr@ld@f\def\Figptsint@rLE#1{%
    \getredf@ctDD\f@ctech(\v@lmin,\v@lmax)%
    \getredf@ctDD\p@rtent(\v@lX,\v@lY)\ifnum\p@rtent>\f@ctech\f@ctech=\p@rtent\fi%
    \getredf@ctDD\p@rtent(\v@lXa,\v@lYa)\ifnum\p@rtent>\f@ctech\f@ctech=\p@rtent\fi%
    \divide\v@lmin\f@ctech\divide\v@lmax\f@ctech\divide\v@lX\f@ctech\divide\v@lY\f@ctech%
    \divide\v@lXa\f@ctech\divide\v@lYa\f@ctech%
    \c@rre=\repdecn@mb\v@lXa\v@lmax\mili@u=\repdecn@mb\v@lYa\v@lmin%
    \getredf@ctDD\f@ctech(\c@rre,\mili@u)%
    \c@rre=\repdecn@mb\v@lX\v@lmax\mili@u=\repdecn@mb\v@lY\v@lmin%
    \getredf@ctDD\p@rtent(\c@rre,\mili@u)\ifnum\p@rtent>\f@ctech\f@ctech=\p@rtent\fi%
    \divide\v@lmin\f@ctech\divide\v@lmax\f@ctech\divide\v@lX\f@ctech\divide\v@lY\f@ctech%
    \divide\v@lXa\f@ctech\divide\v@lYa\f@ctech%
    \v@lmin=\repdecn@mb{\v@lmin}\v@lmin\v@lmax=\repdecn@mb{\v@lmax}\v@lmax%
    \edef\G@xde{\repdecn@mb\v@lmin}\edef\P@xde{\repdecn@mb\v@lmax}%
    \c@rre=-\v@lmax\v@leur=\repdecn@mb\v@lY\v@lY\advance\c@rre\v@leur\c@rre=\G@xde\c@rre%
    \v@leur=\repdecn@mb\v@lX\v@lX\v@leur=\P@xde\v@leur\advance\c@rre\v@leur
    \v@lmin=\repdecn@mb\v@lYa\v@lmin\v@lmax=\repdecn@mb\v@lXa\v@lmax%
    \mili@u=\repdecn@mb\v@lX\v@lmax\advance\mili@u\repdecn@mb\v@lY\v@lmin
    \v@lmax=\repdecn@mb\v@lXa\v@lmax\advance\v@lmax\repdecn@mb\v@lYa\v@lmin
    \ifdim\v@lmax>\epsil@n%
    \maxim@m{\v@leur}{\c@rre}{-\c@rre}\maxim@m{\v@lmin}{\mili@u}{-\mili@u}%
    \maxim@m{\v@leur}{\v@leur}{\v@lmin}\maxim@m{\v@lmin}{\v@lmax}{-\v@lmax}%
    \maxim@m{\v@leur}{\v@leur}{\v@lmin}\p@rtentiere{\p@rtent}{\v@leur}\advance\p@rtent\@ne%
    \divide\c@rre\p@rtent\divide\mili@u\p@rtent\divide\v@lmax\p@rtent%
    \delt@=\repdecn@mb{\mili@u}\mili@u\v@leur=\repdecn@mb{\v@lmax}\c@rre%
    \advance\delt@-\v@leur\ifdim\delt@<\z@\else\sqrt@\delt@\delt@%
    \invers@\v@lmax\v@lmax\edef\Uns@rAp{\repdecn@mb\v@lmax}%
    \v@leur=-\mili@u\advance\v@leur-\delt@\v@leur=\Uns@rAp\v@leur%
    \edef\t@ille{\repdecn@mb\v@leur}\figpttra#1:=-4/\t@ille,-3/\s@mme=#1\advance\s@mme\@ne%
    \v@leur=-\mili@u\advance\v@leur\delt@\v@leur=\Uns@rAp\v@leur%
    \edef\t@ille{\repdecn@mb\v@leur}\figpttra\the\s@mme:=-4/\t@ille,-3/\fi\fi}
\ctr@ln@m\figptsorthoprojline
\ctr@ld@f\def\figptsorthoprojlineDD#1=#2/#3,#4/{\ifps@cri{\s@uvc@ntr@l\et@tfigptsorthoprojlineDD%
    \setc@ntr@l{2}\figvectPDD-3[#3,#4]\figvectNVDD-4[-3]\resetc@ntr@l{2}%
    \def\list@num{#2}\s@mme=#1\@ecfor\p@int:=\list@num\do{%
    \inters@cDD\the\s@mme:[\p@int,-4;#3,-3]\advance\s@mme\@ne}%
    \resetc@ntr@l\et@tfigptsorthoprojlineDD}\ignorespaces\fi}
\ctr@ld@f\def\figptsorthoprojlineTD#1=#2/#3,#4/{\ifps@cri{\s@uvc@ntr@l\et@tfigptsorthoprojlineTD%
    \setc@ntr@l{2}\figvectPTD-2[#3,#4]\vecunit@TD{-2}{-2}%
    \def\list@num{#2}\s@mme=#1\@ecfor\p@int:=\list@num\do{%
    \figvectPTD-1[#3,\p@int]\c@lproscalTD\v@leur[-1,-2]%
    \edef\v@lcoef{\repdecn@mb{\v@leur}}\figpttraTD\the\s@mme:=#3/\v@lcoef,-2/%
    \advance\s@mme\@ne}\resetc@ntr@l\et@tfigptsorthoprojlineTD}\ignorespaces\fi}
\ctr@ln@m\figptsorthoprojplane
\ctr@ld@f\def\figptsorthoprojplaneDD{\un@v@ilable{figptsorthoprojplane}}
\ctr@ld@f\def\figptsorthoprojplaneTD#1=#2/#3,#4/{\ifps@cri{\s@uvc@ntr@l\et@tfigptsorthoprojplane%
    \setc@ntr@l{2}\vecunit@TD{-2}{#4}%
    \def\list@num{#2}\s@mme=#1\@ecfor\p@int:=\list@num\do{\figvectPTD-1[\p@int,#3]%
    \c@lproscalTD\v@leur[-1,-2]\edef\v@lcoef{\repdecn@mb{\v@leur}}%
    \figpttraTD\the\s@mme:=\p@int/\v@lcoef,-2/\advance\s@mme\@ne}%
    \resetc@ntr@l\et@tfigptsorthoprojplane}\ignorespaces\fi}
\ctr@ld@f\def\figptshom#1=#2/#3,#4/{\ifps@cri{\s@uvc@ntr@l\et@tfigptshom%
    \setc@ntr@l{2}\def\list@num{#2}\s@mme=#1%
    \@ecfor\p@int:=\list@num\do{\figvectP-1[#3,\p@int]%
    \figpttra\the\s@mme:=#3/#4,-1/\advance\s@mme\@ne}%
    \resetc@ntr@l\et@tfigptshom}\ignorespaces\fi}
\ctr@ln@m\figptsrot
\ctr@ld@f\def\figptsrotDD#1=#2/#3,#4/{\ifps@cri{\s@uvc@ntr@l\et@tfigptsrotDD%
    \c@ssin{\C@}{\S@}{#4}\setc@ntr@l{2}\def\list@num{#2}\s@mme=#1%
    \@ecfor\p@int:=\list@num\do{\figvectPDD-1[#3,\p@int]\Figg@tXY{-1}%
    \v@lXa=\C@\v@lX\advance\v@lXa-\S@\v@lY%
    \v@lYa=\S@\v@lX\advance\v@lYa\C@\v@lY%
    \Figv@ctCreg-1(\v@lXa,\v@lYa)\figpttraDD\the\s@mme:=#3/1,-1/\advance\s@mme\@ne}%
    \resetc@ntr@l\et@tfigptsrotDD}\ignorespaces\fi}
\ctr@ld@f\def\figptsrotTD#1=#2/#3,#4,#5/{\ifps@cri{\s@uvc@ntr@l\et@tfigptsrotTD%
    \c@ssin{\C@}{\S@}{#4}%
    \setc@ntr@l{2}\def\list@num{#2}\s@mme=#1%
    \@ecfor\p@int:=\list@num\do{\figptorthoprojplaneTD-3:=#3/\p@int,#5/%
    \figvectPTD-2[-3,\p@int]%
    \figvectNVTD-1[#5,-2]\n@rmeucTD\v@leur{-2}\edef\v@lcoef{\repdecn@mb{\v@leur}}%
    \Figg@tXYa{-1}\v@lXa=\v@lcoef\v@lXa\v@lYa=\v@lcoef\v@lYa\v@lZa=\v@lcoef\v@lZa%
    \v@lXa=\S@\v@lXa\v@lYa=\S@\v@lYa\v@lZa=\S@\v@lZa\Figg@tXY{-2}%
    \advance\v@lXa\C@\v@lX\advance\v@lYa\C@\v@lY\advance\v@lZa\C@\v@lZ%
    \Figg@tXY{-3}\advance\v@lXa\v@lX\advance\v@lYa\v@lY\advance\v@lZa\v@lZ%
    \Figp@intregTD\the\s@mme:(\v@lXa,\v@lYa,\v@lZa)\advance\s@mme\@ne}%
    \resetc@ntr@l\et@tfigptsrotTD}\ignorespaces\fi}
\ctr@ln@m\figptssym
\ctr@ld@f\def\figptssymDD#1=#2/#3,#4/{\ifps@cri{\s@uvc@ntr@l\et@tfigptssymDD%
    \setc@ntr@l{2}\figvectPDD-3[#3,#4]\Figg@tXY{-3}\Figv@ctCreg-4(-\v@lY,\v@lX)%
    \resetc@ntr@l{2}\def\list@num{#2}\s@mme=#1%
    \@ecfor\p@int:=\list@num\do{\inters@cDD-5:[#3,-3;\p@int,-4]\figvectPDD-2[\p@int,-5]%
    \figpttraDD\the\s@mme:=\p@int/2,-2/\advance\s@mme\@ne}%
    \resetc@ntr@l\et@tfigptssymDD}\ignorespaces\fi}
\ctr@ld@f\def\figptssymTD#1=#2/#3,#4/{\ifps@cri{\s@uvc@ntr@l\et@tfigptssymTD%
    \setc@ntr@l{2}\vecunit@TD{-2}{#4}\def\list@num{#2}\s@mme=#1%
    \@ecfor\p@int:=\list@num\do{\figvectPTD-1[\p@int,#3]%
    \c@lproscalTD\v@leur[-1,-2]\v@leur=2\v@leur\edef\v@lcoef{\repdecn@mb{\v@leur}}%
    \figpttraTD\the\s@mme:=\p@int/\v@lcoef,-2/\advance\s@mme\@ne}%
    \resetc@ntr@l\et@tfigptssymTD}\ignorespaces\fi}
\ctr@ln@m\figptstra
\ctr@ld@f\def\figptstraDD#1=#2/#3,#4/{\ifps@cri{\Figg@tXYa{#4}\v@lXa=#3\v@lXa\v@lYa=#3\v@lYa%
    \def\list@num{#2}\s@mme=#1\@ecfor\p@int:=\list@num\do{\Figg@tXY{\p@int}%
    \advance\v@lX\v@lXa\advance\v@lY\v@lYa%
    \Figp@intregDD\the\s@mme:(\v@lX,\v@lY)\advance\s@mme\@ne}}\ignorespaces\fi}
\ctr@ld@f\def\figptstraTD#1=#2/#3,#4/{\ifps@cri{\Figg@tXYa{#4}\v@lXa=#3\v@lXa\v@lYa=#3\v@lYa%
    \v@lZa=#3\v@lZa\def\list@num{#2}\s@mme=#1\@ecfor\p@int:=\list@num\do{\Figg@tXY{\p@int}%
    \advance\v@lX\v@lXa\advance\v@lY\v@lYa\advance\v@lZ\v@lZa%
    \Figp@intregTD\the\s@mme:(\v@lX,\v@lY,\v@lZ)\advance\s@mme\@ne}}\ignorespaces\fi}
\ctr@ln@m\figptvisilimSL
\ctr@ld@f\def\figptvisilimSLDD{\un@v@ilable{figptvisilimSL}}
\ctr@ld@f\def\figptvisilimSLTD#1:#2[#3,#4;#5,#6]{\ifps@cri{\s@uvc@ntr@l\et@tfigptvisilimSLTD%
    \setc@ntr@l{2}\figvectP-1[#3,#4]\n@rminf{\delt@}{-1}%
    \ifcase\curr@ntproj\v@lX=\cxa@\p@\v@lY=-\p@\v@lZ=\cxb@\p@
    \Figv@ctCreg-2(\v@lX,\v@lY,\v@lZ)\figvectP-3[#5,#6]\figvectNV-1[-2,-3]%
    \or\figvectP-1[#5,#6]\vecunitCV@TD{-1}\v@lmin=\v@lX\v@lmax=\v@lY
    \v@leur=\v@lZ\v@lX=\cza@\p@\v@lY=\czb@\p@\v@lZ=\czc@\p@\c@lprovec{-1}%
    \or\c@ley@pt{-2}\figvectN-1[#5,#6,-2]\fi
    \edef\Ai@{#3}\edef\Aj@{#4}\figvectP-2[#5,\Ai@]\c@lproscal\v@leur[-1,-2]%
    \ifdim\v@leur>\z@\p@rtent=\@ne\else\p@rtent=\m@ne\fi%
    \figvectP-2[#5,\Aj@]\c@lproscal\v@leur[-1,-2]%
    \ifdim\p@rtent\v@leur>\z@\figptcopy#1:#2/#3/%
    \message{*** \BS@ figptvisilimSL: points are on the same side.}\else%
    \figptcopy-3:/#3/\figptcopy-4:/#4/%
    \loop\figptbary-5:[-3,-4;1,1]\figvectP-2[#5,-5]\c@lproscal\v@leur[-1,-2]%
    \ifdim\p@rtent\v@leur>\z@\figptcopy-3:/-5/\else\figptcopy-4:/-5/\fi%
    \divide\delt@\tw@\ifdim\delt@>\epsil@n\repeat%
    \figptbary#1:#2[-3,-4;1,1]\fi\resetc@ntr@l\et@tfigptvisilimSLTD}\ignorespaces\fi}
\ctr@ld@f\def\c@ley@pt#1{\t@stp@r\ifitis@K\v@lX=\cza@\p@\v@lY=\czb@\p@\v@lZ=\czc@\p@%
    \Figv@ctCreg-1(\v@lX,\v@lY,\v@lZ)\Figp@intreg-2:(\wd\Bt@rget,\ht\Bt@rget,\dp\Bt@rget)%
    \figpttra#1:=-2/-\disob@intern,-1/\else\end\fi}
\ctr@ld@f\def\t@stp@r{\itis@Ktrue\ifnewt@rgetpt\else\itis@Kfalse%
    \message{*** \BS@ figptvisilimXX: target point undefined.}\fi\ifnewdis@b\else%
    \itis@Kfalse\message{*** \BS@ figptvisilimXX: observation distance undefined.}\fi%
    \ifitis@K\else\message{*** This macro must be called after \BS@ psbeginfig or after
    having set the missing parameter(s) with \BS@ figset proj()}\fi}
\ctr@ld@f\def\figscan#1(#2,#3){{\s@uvc@ntr@l\et@tfigscan\@psfgetbb{#1}\if@psfbbfound\else%
    \def\@psfllx{0}\def\@psflly{20}\def\@psfurx{540}\def\@psfury{640}\fi\figscan@{#2}{#3}%
    \resetc@ntr@l\et@tfigscan}\ignorespaces}
\ctr@ld@f\def\figscan@#1#2{%
    \unit@=\@ne bp\setc@ntr@l{2}\figsetmark{}%
    \def\minst@p{20pt}%
    \v@lX=\@psfllx\p@\v@lX=\Sc@leFact\v@lX\r@undint\v@lX\v@lX%
    \v@lY=\@psflly\p@\v@lY=\Sc@leFact\v@lY\ifdim\v@lY>\z@\r@undint\v@lY\v@lY\fi%
    \delt@=\@psfury\p@\delt@=\Sc@leFact\delt@%
    \advance\delt@-\v@lY\v@lXa=\@psfurx\p@\v@lXa=\Sc@leFact\v@lXa\v@leur=\minst@p%
    \edef\valv@lY{\repdecn@mb{\v@lY}}\edef\LgTr@it{\the\delt@}%
    \loop\ifdim\v@lX<\v@lXa\edef\valv@lX{\repdecn@mb{\v@lX}}%
    \figptDD -1:(\valv@lX,\valv@lY)\figwriten -1:\hbox{\vrule height\LgTr@it}(0)%
    \ifdim\v@leur<\minst@p\else\figsetmark{\raise-8bp\hbox{$\scriptscriptstyle\triangle$}}%
    \figwrites -1:\@ffichnb{0}{\valv@lX}(6)\v@leur=\z@\figsetmark{}\fi%
    \advance\v@leur#1pt\advance\v@lX#1pt\repeat%
    \def\minst@p{10pt}%
    \v@lX=\@psfllx\p@\v@lX=\Sc@leFact\v@lX\ifdim\v@lX>\z@\r@undint\v@lX\v@lX\fi%
    \v@lY=\@psflly\p@\v@lY=\Sc@leFact\v@lY\r@undint\v@lY\v@lY%
    \delt@=\@psfurx\p@\delt@=\Sc@leFact\delt@%
    \advance\delt@-\v@lX\v@lYa=\@psfury\p@\v@lYa=\Sc@leFact\v@lYa\v@leur=\minst@p%
    \edef\valv@lX{\repdecn@mb{\v@lX}}\edef\LgTr@it{\the\delt@}%
    \loop\ifdim\v@lY<\v@lYa\edef\valv@lY{\repdecn@mb{\v@lY}}%
    \figptDD -1:(\valv@lX,\valv@lY)\figwritee -1:\vbox{\hrule width\LgTr@it}(0)%
    \ifdim\v@leur<\minst@p\else\figsetmark{$\triangleright$\kern4bp}%
    \figwritew -1:\@ffichnb{0}{\valv@lY}(6)\v@leur=\z@\figsetmark{}\fi%
    \advance\v@leur#2pt\advance\v@lY#2pt\repeat}
\ctr@ld@f\let\figscanI=\figscan
\ctr@ld@f\def\figscan@E#1(#2,#3){{\s@uvc@ntr@l\et@tfigscan@E%
    \Figdisc@rdLTS{#1}{\t@xt@}\pdfximage{\t@xt@}%
    \setbox\Gb@x=\hbox{\pdfrefximage\pdflastximage}%
    \edef\@psfllx{0}\v@lY=-\dp\Gb@x\edef\@psflly{\repdecn@mb{\v@lY}}%
    \edef\@psfurx{\repdecn@mb{\wd\Gb@x}}%
    \v@lY=\dp\Gb@x\advance\v@lY\ht\Gb@x\edef\@psfury{\repdecn@mb{\v@lY}}%
    \figscan@{#2}{#3}\resetc@ntr@l\et@tfigscan@E}\ignorespaces}
\ctr@ld@f\def\figshowpts[#1,#2]{{\figsetmark{$\bullet$}\figsetptname{\bf ##1}%
    \p@rtent=#2\relax\ifnum\p@rtent<\z@\p@rtent=\z@\fi%
    \s@mme=#1\relax\ifnum\s@mme<\z@\s@mme=\z@\fi%
    \loop\ifnum\s@mme<\p@rtent\pt@rvect{\s@mme}%
    \ifitis@K\figwriten{\the\s@mme}:(4pt)\fi\advance\s@mme\@ne\repeat%
    \pt@rvect{\s@mme}\ifitis@K\figwriten{\the\s@mme}:(4pt)\fi}\ignorespaces}
\ctr@ld@f\def\pt@rvect#1{\set@bjc@de{#1}%
    \expandafter\expandafter\expandafter\inqpt@rvec\csname\objc@de\endcsname:}
\ctr@ld@f\def\inqpt@rvec#1#2:{\if#1\C@dCl@spt\itis@Ktrue\else\itis@Kfalse\fi}
\ctr@ld@f\def\figshowsettings{{%
    \immediate\write16{====================================================================}%
    \immediate\write16{ Current settings about:}%
    \immediate\write16{ --- GENERAL ---}%
    \immediate\write16{Scale factor and Unit = \unit@util\space (\the\unit@)
     \space -> \BS@ figinit{ScaleFactorUnit}}%
    \immediate\write16{Update mode = \ifpsupdatem@de yes\else no\fi
     \space-> \BS@ psset(update=yes/no) or \BS@ pssetdefault(update=yes/no)}%
    \immediate\write16{ --- PRINTING ---}%
    \immediate\write16{Implicit point name = \ptn@me{i} \space-> \BS@ figsetptname{Name}}%
    \immediate\write16{Point marker = \the\c@nsymb \space -> \BS@ figsetmark{Mark}}%
    \immediate\write16{Print rounded coordinates = \ifr@undcoord yes\else no\fi
     \space-> \BS@ figsetroundcoord{yes/no}}%
    \immediate\write16{ --- GRAPHICAL (general) ---}%
    \immediate\write16{First-level (or primary) settings:}%
    \immediate\write16{ Color = \curr@ntcolor \space-> \BS@ psset(color=ColorDefinition)}%
    \immediate\write16{ Filling mode = \iffillm@de yes\else no\fi
     \space-> \BS@ psset(fillmode=yes/no)}%
    \immediate\write16{ Line join = \curr@ntjoin \space-> \BS@ psset(join=miter/round/bevel)}%
    \immediate\write16{ Line style = \curr@ntdash \space-> \BS@ psset(dash=Index/Pattern)}%
    \immediate\write16{ Line width = \curr@ntwidth
     \space-> \BS@ psset(width=real in PostScript units)}%
    \immediate\write16{Second-level (or secondary) settings:}%
    \immediate\write16{ Color = \sec@ndcolor \space-> \BS@ psset second(color=ColorDefinition)}%
    \immediate\write16{ Line style = \curr@ntseconddash
     \space-> \BS@ psset second(dash=Index/Pattern)}%
    \immediate\write16{ Line width = \curr@ntsecondwidth
     \space-> \BS@ psset second(width=real in PostScript units)}%
    \immediate\write16{Third-level (or ternary) settings:}%
    \immediate\write16{ Color = \th@rdcolor \space-> \BS@ psset third(color=ColorDefinition)}%
    \immediate\write16{ --- GRAPHICAL (specific) ---}%
    \immediate\write16{Arrow-head:}%
    \immediate\write16{ (half-)Angle = \@rrowheadangle
     \space-> \BS@ psset arrowhead(angle=real in degrees)}%
    \immediate\write16{ Filling mode = \if@rrowhfill yes\else no\fi
     \space-> \BS@ psset arrowhead(fillmode=yes/no)}%
    \immediate\write16{ "Outside" = \if@rrowhout yes\else no\fi
     \space-> \BS@ psset arrowhead(out=yes/no)}%
    \immediate\write16{ Length = \@rrowheadlength
     \if@rrowratio\space(not active)\else\space(active)\fi
     \space-> \BS@ psset arrowhead(length=real in user coord.)}%
    \immediate\write16{ Ratio = \@rrowheadratio
     \if@rrowratio\space(active)\else\space(not active)\fi
     \space-> \BS@ psset arrowhead(ratio=real in [0,1])}%
    \immediate\write16{Curve: Roundness = \curv@roundness
     \space-> \BS@ psset curve(roundness=real in [0,0.5])}%
    \immediate\write16{Mesh: Diagonal = \c@ntrolmesh
     \space-> \BS@ psset mesh(diag=integer in {-1,0,1})}%
    \immediate\write16{Flow chart:}%
    \immediate\write16{ Arrow position = \@rrowp@s
     \space-> \BS@ psset flowchart(arrowposition=real in [0,1])}%
    \immediate\write16{ Arrow reference point = \ifcase\@rrowr@fpt start\else end\fi
     \space-> \BS@ psset flowchart(arrowrefpt = start/end)}%
    \immediate\write16{ Line type = \ifcase\fclin@typ@ curve\else polygon\fi
     \space-> \BS@ psset flowchart(line=polygon/curve)}%
    \immediate\write16{ Padding = (\Xp@dd, \Yp@dd)
     \space-> \BS@ psset flowchart(padding = real in user coord.)}%
    \immediate\write16{\space\space\space\space(or
     \BS@ psset flowchart(xpadding=real, ypadding=real) )}%
    \immediate\write16{ Radius = \fclin@r@d
     \space-> \BS@ psset flowchart(radius=positive real in user coord.)}%
    \immediate\write16{ Shape = \fcsh@pe
     \space-> \BS@ psset flowchart(shape = rectangle, ellipse or lozenge)}%
    \immediate\write16{ Thickness = \thickn@ss
     \space-> \BS@ psset flowchart(thickness = real in user coord.)}%
    \ifTr@isDim%
    \immediate\write16{ --- 3D to 2D PROJECTION ---}%
    \immediate\write16{Projection : \typ@proj \space-> \BS@ figinit{ScaleFactorUnit, ProjType}}%
    \immediate\write16{Longitude (psi) = \v@lPsi \space-> \BS@ figset proj(psi=real in degrees)}%
    \ifcase\curr@ntproj\immediate\write16{Depth coeff. (Lambda)
     \space = \v@lTheta \space-> \BS@ figset proj(lambda=real in [0,1])}%
    \else\immediate\write16{Latitude (theta)
     \space = \v@lTheta \space-> \BS@ figset proj(theta=real in degrees)}%
    \fi%
    \ifnum\curr@ntproj=\tw@%
    \immediate\write16{Observation distance = \disob@unit
     \space-> \BS@ figset proj(dist=real in user coord.)}%
    \immediate\write16{Target point = \t@rgetpt \space-> \BS@ figset proj(targetpt=pt number)}%
     \v@lX=\ptT@unit@\wd\Bt@rget\v@lY=\ptT@unit@\ht\Bt@rget\v@lZ=\ptT@unit@\dp\Bt@rget%
    \immediate\write16{ Its coordinates are
     (\repdecn@mb{\v@lX}, \repdecn@mb{\v@lY}, \repdecn@mb{\v@lZ})}%
    \fi%
    \fi%
    \immediate\write16{====================================================================}%
    \ignorespaces}}
\ctr@ln@w{newif}\ifitis@vect@r
\ctr@ld@f\def\figvectC#1(#2,#3){{\itis@vect@rtrue\figpt#1:(#2,#3)}\ignorespaces}
\ctr@ld@f\def\Figv@ctCreg#1(#2,#3){{\itis@vect@rtrue\Figp@intreg#1:(#2,#3)}\ignorespaces}
\ctr@ln@m\figvectDBezier
\ctr@ld@f\def\figvectDBezierDD#1:#2,#3[#4,#5,#6,#7]{\ifps@cri{\s@uvc@ntr@l\et@tfigvectDBezierDD%
    \FigvectDBezier@#2,#3[#4,#5,#6,#7]\v@lX=\c@ef\v@lX\v@lY=\c@ef\v@lY%
    \Figv@ctCreg#1(\v@lX,\v@lY)\resetc@ntr@l\et@tfigvectDBezierDD}\ignorespaces\fi}
\ctr@ld@f\def\figvectDBezierTD#1:#2,#3[#4,#5,#6,#7]{\ifps@cri{\s@uvc@ntr@l\et@tfigvectDBezierTD%
    \FigvectDBezier@#2,#3[#4,#5,#6,#7]\v@lX=\c@ef\v@lX\v@lY=\c@ef\v@lY\v@lZ=\c@ef\v@lZ%
    \Figv@ctCreg#1(\v@lX,\v@lY,\v@lZ)\resetc@ntr@l\et@tfigvectDBezierTD}\ignorespaces\fi}
\ctr@ld@f\def\FigvectDBezier@#1,#2[#3,#4,#5,#6]{\setc@ntr@l{2}%
    \edef\T@{#2}\v@leur=\p@\advance\v@leur-#2pt\edef\UNmT@{\repdecn@mb{\v@leur}}%
    \ifnum#1=\tw@\def\c@ef{6}\else\def\c@ef{3}\fi%
    \figptcopy-4:/#3/\figptcopy-3:/#4/\figptcopy-2:/#5/\figptcopy-1:/#6/%
    \l@mbd@un=-4 \l@mbd@de=-\thr@@\p@rtent=\m@ne\c@lDecast%
    \ifnum#1=\tw@\c@lDCDeux{-4}{-3}\c@lDCDeux{-3}{-2}\c@lDCDeux{-4}{-3}\else%
    \l@mbd@un=-4 \l@mbd@de=-\thr@@\p@rtent=-\tw@\c@lDecast%
    \c@lDCDeux{-4}{-3}\fi\Figg@tXY{-4}}
\ctr@ln@m\c@lDCDeux
\ctr@ld@f\def\c@lDCDeuxDD#1#2{\Figg@tXY{#2}\Figg@tXYa{#1}%
    \advance\v@lX-\v@lXa\advance\v@lY-\v@lYa\Figp@intregDD#1:(\v@lX,\v@lY)}
\ctr@ld@f\def\c@lDCDeuxTD#1#2{\Figg@tXY{#2}\Figg@tXYa{#1}\advance\v@lX-\v@lXa%
    \advance\v@lY-\v@lYa\advance\v@lZ-\v@lZa\Figp@intregTD#1:(\v@lX,\v@lY,\v@lZ)}
\ctr@ln@m\figvectN
\ctr@ld@f\def\figvectNDD#1[#2,#3]{\ifps@cri{\Figg@tXYa{#2}\Figg@tXY{#3}%
    \advance\v@lX-\v@lXa\advance\v@lY-\v@lYa%
    \Figv@ctCreg#1(-\v@lY,\v@lX)}\ignorespaces\fi}
\ctr@ld@f\def\figvectNTD#1[#2,#3,#4]{\ifps@cri{\vecunitC@TD[#2,#4]\v@lmin=\v@lX\v@lmax=\v@lY%
    \v@leur=\v@lZ\vecunitC@TD[#2,#3]\c@lprovec{#1}}\ignorespaces\fi}
\ctr@ln@m\figvectNV
\ctr@ld@f\def\figvectNVDD#1[#2]{\ifps@cri{\Figg@tXY{#2}\Figv@ctCreg#1(-\v@lY,\v@lX)}\ignorespaces\fi}
\ctr@ld@f\def\figvectNVTD#1[#2,#3]{\ifps@cri{\vecunitCV@TD{#3}\v@lmin=\v@lX\v@lmax=\v@lY%
    \v@leur=\v@lZ\vecunitCV@TD{#2}\c@lprovec{#1}}\ignorespaces\fi}
\ctr@ln@m\figvectP
\ctr@ld@f\def\figvectPDD#1[#2,#3]{\ifps@cri{\Figg@tXYa{#2}\Figg@tXY{#3}%
    \advance\v@lX-\v@lXa\advance\v@lY-\v@lYa%
    \Figv@ctCreg#1(\v@lX,\v@lY)}\ignorespaces\fi}
\ctr@ld@f\def\figvectPTD#1[#2,#3]{\ifps@cri{\Figg@tXYa{#2}\Figg@tXY{#3}%
    \advance\v@lX-\v@lXa\advance\v@lY-\v@lYa\advance\v@lZ-\v@lZa%
    \Figv@ctCreg#1(\v@lX,\v@lY,\v@lZ)}\ignorespaces\fi}
\ctr@ln@m\figvectU
\ctr@ld@f\def\figvectUDD#1[#2]{\ifps@cri{\n@rmeuc\v@leur{#2}\invers@\v@leur\v@leur%
    \delt@=\repdecn@mb{\v@leur}\unit@\edef\v@ldelt@{\repdecn@mb{\delt@}}%
    \Figg@tXY{#2}\v@lX=\v@ldelt@\v@lX\v@lY=\v@ldelt@\v@lY%
    \Figv@ctCreg#1(\v@lX,\v@lY)}\ignorespaces\fi}
\ctr@ld@f\def\figvectUTD#1[#2]{\ifps@cri{\n@rmeuc\v@leur{#2}\invers@\v@leur\v@leur%
    \delt@=\repdecn@mb{\v@leur}\unit@\edef\v@ldelt@{\repdecn@mb{\delt@}}%
    \Figg@tXY{#2}\v@lX=\v@ldelt@\v@lX\v@lY=\v@ldelt@\v@lY\v@lZ=\v@ldelt@\v@lZ%
    \Figv@ctCreg#1(\v@lX,\v@lY,\v@lZ)}\ignorespaces\fi}
\ctr@ld@f\def\figvisu#1#2#3{\c@ldefproj\initb@undb@x\xdef\figforTeXFigno{\figforTeXnextFigno}%
    \s@mme=\figforTeXnextFigno\advance\s@mme\@ne\xdef\figforTeXnextFigno{\number\s@mme}%
    \setbox\b@xvisu=\hbox{\ifnum\@utoFN>\z@\figinsert{}\gdef\@utoFInDone{0}\fi\ignorespaces#3}%
    \gdef\@utoFInDone{1}\gdef\@utoFN{0}%
    \v@lXa=-\c@@rdYmin\v@lYa=\c@@rdYmax\advance\v@lYa-\c@@rdYmin%
    \v@lX=\c@@rdXmax\advance\v@lX-\c@@rdXmin%
    \setbox#1=\hbox{#2}\v@lY=-\v@lX\maxim@m{\v@lX}{\v@lX}{\wd#1}%
    \advance\v@lY\v@lX\divide\v@lY\tw@\advance\v@lY-\c@@rdXmin%
    \setbox#1=\vbox{\parindent0mm\hsize=\v@lX\vskip\v@lYa%
    \rlap{\hskip\v@lY\smash{\raise\v@lXa\box\b@xvisu}}%
    \def\t@xt@{#2}\ifx\t@xt@\empty\else\medskip\centerline{#2}\fi}\wd#1=\v@lX}
\ctr@ld@f\def\figDecrementFigno{{\xdef\figforTeXnextFigno{\figforTeXFigno}%
    \s@mme=\figforTeXFigno\advance\s@mme\m@ne\xdef\figforTeXFigno{\number\s@mme}}}
\ctr@ln@w{newbox}\Bt@rget\setbox\Bt@rget=\null
\ctr@ln@w{newbox}\BminTD@\setbox\BminTD@=\null
\ctr@ln@w{newbox}\BmaxTD@\setbox\BmaxTD@=\null
\ctr@ln@w{newif}\ifnewt@rgetpt\ctr@ln@w{newif}\ifnewdis@b
\ctr@ld@f\def\b@undb@xTD#1#2#3{%
    \relax\ifdim#1<\wd\BminTD@\global\wd\BminTD@=#1\fi%
    \relax\ifdim#2<\ht\BminTD@\global\ht\BminTD@=#2\fi%
    \relax\ifdim#3<\dp\BminTD@\global\dp\BminTD@=#3\fi%
    \relax\ifdim#1>\wd\BmaxTD@\global\wd\BmaxTD@=#1\fi%
    \relax\ifdim#2>\ht\BmaxTD@\global\ht\BmaxTD@=#2\fi%
    \relax\ifdim#3>\dp\BmaxTD@\global\dp\BmaxTD@=#3\fi}
\ctr@ld@f\def\c@ldefdisob{{\ifdim\wd\BminTD@<\maxdimen\v@leur=\wd\BmaxTD@\advance\v@leur-\wd\BminTD@%
    \delt@=\ht\BmaxTD@\advance\delt@-\ht\BminTD@\maxim@m{\v@leur}{\v@leur}{\delt@}%
    \delt@=\dp\BmaxTD@\advance\delt@-\dp\BminTD@\maxim@m{\v@leur}{\v@leur}{\delt@}%
    \v@leur=5\v@leur\else\v@leur=800pt\fi\c@ldefdisob@{\v@leur}}}
\ctr@ln@m\disob@intern
\ctr@ln@m\disob@
\ctr@ln@m\divf@ctproj
\ctr@ld@f\def\c@ldefdisob@#1{{\v@leur=#1\ifdim\v@leur<\p@\v@leur=800pt\fi%
    \xdef\disob@intern{\repdecn@mb{\v@leur}}%
    \delt@=\ptT@unit@\v@leur\xdef\disob@unit{\repdecn@mb{\delt@}}%
    \f@ctech=\@ne\loop\ifdim\v@leur>\t@n pt\divide\v@leur\t@n\multiply\f@ctech\t@n\repeat%
    \xdef\disob@{\repdecn@mb{\v@leur}}\xdef\divf@ctproj{\the\f@ctech}}%
    \global\newdis@btrue}
\ctr@ln@m\t@rgetpt
\ctr@ld@f\def\c@ldeft@rgetpt{\newt@rgetpttrue\def\t@rgetpt{CenterBoundBox}{%
    \delt@=\wd\BmaxTD@\advance\delt@-\wd\BminTD@\divide\delt@\tw@%
    \v@leur=\wd\BminTD@\advance\v@leur\delt@\global\wd\Bt@rget=\v@leur%
    \delt@=\ht\BmaxTD@\advance\delt@-\ht\BminTD@\divide\delt@\tw@%
    \v@leur=\ht\BminTD@\advance\v@leur\delt@\global\ht\Bt@rget=\v@leur%
    \delt@=\dp\BmaxTD@\advance\delt@-\dp\BminTD@\divide\delt@\tw@%
    \v@leur=\dp\BminTD@\advance\v@leur\delt@\global\dp\Bt@rget=\v@leur}}
\ctr@ln@m\c@ldefproj
\ctr@ld@f\def\c@ldefprojTD{\ifnewt@rgetpt\else\c@ldeft@rgetpt\fi\ifnewdis@b\else\c@ldefdisob\fi}
\ctr@ld@f\def\c@lprojcav{
    \v@lZa=\cxa@\v@lY\advance\v@lX\v@lZa%
    \v@lZa=\cxb@\v@lY\v@lY=\v@lZ\advance\v@lY\v@lZa\ignorespaces}
\ctr@ln@m\v@lcoef
\ctr@ld@f\def\c@lprojrea{
    \advance\v@lX-\wd\Bt@rget\advance\v@lY-\ht\Bt@rget\advance\v@lZ-\dp\Bt@rget%
    \v@lZa=\cza@\v@lX\advance\v@lZa\czb@\v@lY\advance\v@lZa\czc@\v@lZ%
    \divide\v@lZa\divf@ctproj\advance\v@lZa\disob@ pt\invers@{\v@lZa}{\v@lZa}%
    \v@lZa=\disob@\v@lZa\edef\v@lcoef{\repdecn@mb{\v@lZa}}%
    \v@lXa=\cxa@\v@lX\advance\v@lXa\cxb@\v@lY\v@lXa=\v@lcoef\v@lXa%
    \v@lY=\cyb@\v@lY\advance\v@lY\cya@\v@lX\advance\v@lY\cyc@\v@lZ%
    \v@lY=\v@lcoef\v@lY\v@lX=\v@lXa\ignorespaces}
\ctr@ld@f\def\c@lprojort{
    \v@lXa=\cxa@\v@lX\advance\v@lXa\cxb@\v@lY%
    \v@lY=\cyb@\v@lY\advance\v@lY\cya@\v@lX\advance\v@lY\cyc@\v@lZ%
    \v@lX=\v@lXa\ignorespaces}
\ctr@ld@f\def\Figptpr@j#1:#2/#3/{{\Figg@tXY{#3}\superc@lprojSP%
    \Figp@intregDD#1:{#2}(\v@lX,\v@lY)}\ignorespaces}
\ctr@ln@m\figsetobdist
\ctr@ld@f\def\figsetobdistDD{\un@v@ilable{figsetobdist}}
\ctr@ld@f\def\figsetobdistTD(#1){{\ifcurr@ntPS%
    \immediate\write16{*** \BS@ figsetobdist is ignored inside a
     \BS@ psbeginfig-\BS@ psendfig block.}%
    \else\v@leur=#1\unit@\c@ldefdisob@{\v@leur}\fi}\ignorespaces}
\ctr@ln@m\c@lprojSP
\ctr@ln@m\curr@ntproj
\ctr@ln@m\typ@proj
\ctr@ln@m\superc@lprojSP
\ctr@ld@f\def\Figs@tproj#1{%
    \if#13 \d@faultproj\else\if#1c\d@faultproj%
    \else\if#1o\xdef\curr@ntproj{1}\xdef\typ@proj{orthogonal}%
         \figsetviewTD(\def@ultpsi,\def@ulttheta)%
         \global\let\c@lprojSP=\c@lprojort\global\let\superc@lprojSP=\c@lprojort%
    \else\if#1r\xdef\curr@ntproj{2}\xdef\typ@proj{realistic}%
         \figsetviewTD(\def@ultpsi,\def@ulttheta)%
         \global\let\c@lprojSP=\c@lprojrea\global\let\superc@lprojSP=\c@lprojrea%
    \else\d@faultproj\message{*** Unknown projection. Cavalier projection assumed.}%
    \fi\fi\fi\fi}
\ctr@ld@f\def\d@faultproj{\xdef\curr@ntproj{0}\xdef\typ@proj{cavalier}\figsetviewTD(\def@ultpsi,0.5)%
         \global\let\c@lprojSP=\c@lprojcav\global\let\superc@lprojSP=\c@lprojcav}
\ctr@ln@m\figsettarget
\ctr@ld@f\def\figsettargetDD{\un@v@ilable{figsettarget}}
\ctr@ld@f\def\figsettargetTD[#1]{{\ifcurr@ntPS%
    \immediate\write16{*** \BS@ figsettarget is ignored inside a
     \BS@ psbeginfig-\BS@ psendfig block.}%
    \else\global\newt@rgetpttrue\xdef\t@rgetpt{#1}\Figg@tXY{#1}\global\wd\Bt@rget=\v@lX%
    \global\ht\Bt@rget=\v@lY\global\dp\Bt@rget=\v@lZ\fi}\ignorespaces}
\ctr@ln@m\figsetview
\ctr@ld@f\def\figsetviewDD{\un@v@ilable{figsetview}}
\ctr@ld@f\def\figsetviewTD(#1){\ifcurr@ntPS%
     \immediate\write16{*** \BS@ figsetview is ignored inside a
     \BS@ psbeginfig-\BS@ psendfig block.}\else\Figsetview@#1,:\fi\ignorespaces}
\ctr@ld@f\def\Figsetview@#1,#2:{{\xdef\v@lPsi{#1}\def\t@xt@{#2}%
    \ifx\t@xt@\empty\def\@rgdeux{\v@lTheta}\else\X@rgdeux@#2\fi%
    \c@ssin{\costhet@}{\sinthet@}{#1}\v@lmin=\costhet@ pt\v@lmax=\sinthet@ pt%
    \ifcase\curr@ntproj%
    \v@leur=\@rgdeux\v@lmin\xdef\cxa@{\repdecn@mb{\v@leur}}%
    \v@leur=\@rgdeux\v@lmax\xdef\cxb@{\repdecn@mb{\v@leur}}\v@leur=\@rgdeux pt%
    \relax\ifdim\v@leur>\p@\message{*** Lambda too large ! See \BS@ figset proj() !}\fi%
    \else%
    \v@lmax=-\v@lmax\xdef\cxa@{\repdecn@mb{\v@lmax}}\xdef\cxb@{\costhet@}%
    \ifx\t@xt@\empty\edef\@rgdeux{\def@ulttheta}\fi\c@ssin{\C@}{\S@}{\@rgdeux}%
    \v@lmax=-\S@ pt%
    \v@leur=\v@lmax\v@leur=\costhet@\v@leur\xdef\cya@{\repdecn@mb{\v@leur}}%
    \v@leur=\v@lmax\v@leur=\sinthet@\v@leur\xdef\cyb@{\repdecn@mb{\v@leur}}%
    \xdef\cyc@{\C@}\v@lmin=-\C@ pt%
    \v@leur=\v@lmin\v@leur=\costhet@\v@leur\xdef\cza@{\repdecn@mb{\v@leur}}%
    \v@leur=\v@lmin\v@leur=\sinthet@\v@leur\xdef\czb@{\repdecn@mb{\v@leur}}%
    \xdef\czc@{\repdecn@mb{\v@lmax}}\fi%
    \xdef\v@lTheta{\@rgdeux}}}
\ctr@ld@f\def\def@ultpsi{40}
\ctr@ld@f\def\def@ulttheta{25}
\ctr@ln@m\l@debut
\ctr@ln@m\n@mref
\ctr@ld@f\def\figset#1(#2){\def\t@xt@{#1}\ifx\t@xt@\empty\trtlis@rg{#2}{\Figsetwr@te}
    \else\keln@mde#1|%
    \def\n@mref{pr}\ifx\l@debut\n@mref\ifcurr@ntPS
     \immediate\write16{*** \BS@ figset proj(...) is ignored inside a
     \BS@ psbeginfig-\BS@ psendfig block.}\else\trtlis@rg{#2}{\Figsetpr@j}\fi\else%
    \def\n@mref{wr}\ifx\l@debut\n@mref\trtlis@rg{#2}{\Figsetwr@te}\else
    \immediate\write16{*** Unknown keyword: \BS@ figset #1(...)}%
    \fi\fi\fi\ignorespaces}
\ctr@ld@f\def\Figsetpr@j#1=#2|{\keln@mtr#1|%
    \def\n@mref{dep}\ifx\l@debut\n@mref\Figsetd@p{#2}\else
    \def\n@mref{dis}\ifx\l@debut\n@mref%
     \ifnum\curr@ntproj=\tw@\figsetobdist(#2)\else\Figset@rr\fi\else
    \def\n@mref{lam}\ifx\l@debut\n@mref\Figsetd@p{#2}\else
    \def\n@mref{lat}\ifx\l@debut\n@mref\Figsetth@{#2}\else
    \def\n@mref{lon}\ifx\l@debut\n@mref\figsetview(#2)\else
    \def\n@mref{psi}\ifx\l@debut\n@mref\figsetview(#2)\else
    \def\n@mref{tar}\ifx\l@debut\n@mref%
     \ifnum\curr@ntproj=\tw@\figsettarget[#2]\else\Figset@rr\fi\else
    \def\n@mref{the}\ifx\l@debut\n@mref\Figsetth@{#2}\else
    \immediate\write16{*** Unknown attribute: \BS@ figset proj(..., #1=...).}%
    \fi\fi\fi\fi\fi\fi\fi\fi}
\ctr@ld@f\def\Figsetd@p#1{\ifnum\curr@ntproj=\z@\figsetview(\v@lPsi,#1)\else\Figset@rr\fi}
\ctr@ld@f\def\Figsetth@#1{\ifnum\curr@ntproj=\z@\Figset@rr\else\figsetview(\v@lPsi,#1)\fi}
\ctr@ld@f\def\Figset@rr{\message{*** \BS@ figset proj(): Attribute "\n@mref" ignored, incompatible
    with current projection}}
\ctr@ld@f\def\initb@undb@xTD{\wd\BminTD@=\maxdimen\ht\BminTD@=\maxdimen\dp\BminTD@=\maxdimen%
    \wd\BmaxTD@=-\maxdimen\ht\BmaxTD@=-\maxdimen\dp\BmaxTD@=-\maxdimen}
\ctr@ln@w{newbox}\Gb@x      
\ctr@ln@w{newbox}\Gb@xSC    
\ctr@ln@w{newtoks}\c@nsymb  
\ctr@ln@w{newif}\ifr@undcoord\ctr@ln@w{newif}\ifunitpr@sent
\ctr@ld@f\def\unssqrttw@{0.707106 }
\ctr@ld@f\def\figAst{\raise-1.15ex\hbox{$\ast$}}
\ctr@ld@f\def\figBullet{\raise-1.15ex\hbox{$\bullet$}}
\ctr@ld@f\def\figCirc{\raise-1.15ex\hbox{$\circ$}}
\ctr@ld@f\def\figDiamond{\raise-1.15ex\hbox{$\diamond$}}%
\ctr@ld@f\def\boxit#1#2{\leavevmode\hbox{\vrule\vbox{\hrule\vglue#1%
    \vtop{\hbox{\kern#1{#2}\kern#1}\vglue#1\hrule}}\vrule}}
\ctr@ld@f\def\centertext#1#2{\vbox{\hsize#1\parindent0cm%
    \leftskip=0pt plus 1fil\rightskip=0pt plus 1fil\parfillskip=0pt{#2}}}
\ctr@ld@f\def\lefttext#1#2{\vbox{\hsize#1\parindent0cm\rightskip=0pt plus 1fil#2}}
\ctr@ld@f\def\c@nterpt{\ignorespaces%
    \kern-.5\wd\Gb@xSC%
    \raise-.5\ht\Gb@xSC\rlap{\hbox{\raise.5\dp\Gb@xSC\hbox{\copy\Gb@xSC}}}%
    \kern .5\wd\Gb@xSC\ignorespaces}
\ctr@ld@f\def\b@undb@xSC#1#2{{\v@lXa=#1\v@lYa=#2%
    \v@leur=\ht\Gb@xSC\advance\v@leur\dp\Gb@xSC%
    \advance\v@lXa-.5\wd\Gb@xSC\advance\v@lYa-.5\v@leur\b@undb@x{\v@lXa}{\v@lYa}%
    \advance\v@lXa\wd\Gb@xSC\advance\v@lYa\v@leur\b@undb@x{\v@lXa}{\v@lYa}}}
\ctr@ln@m\Dist@n
\ctr@ln@m\l@suite
\ctr@ld@f\def\@keldist#1#2{\edef\Dist@n{#2}\y@tiunit{\Dist@n}%
    \ifunitpr@sent#1=\Dist@n\else#1=\Dist@n\unit@\fi}
\ctr@ld@f\def\y@tiunit#1{\unitpr@sentfalse\expandafter\y@tiunit@#1:}
\ctr@ld@f\def\y@tiunit@#1#2:{\ifcat#1a\unitpr@senttrue\else\def\l@suite{#2}%
    \ifx\l@suite\empty\else\y@tiunit@#2:\fi\fi}
\ctr@ln@m\figcoord
\ctr@ld@f\def\figcoordDD#1{{\v@lX=\ptT@unit@\v@lX\v@lY=\ptT@unit@\v@lY%
    \ifr@undcoord\ifcase#1\v@leur=0.5pt\or\v@leur=0.05pt\or\v@leur=0.005pt%
    \or\v@leur=0.0005pt\else\v@leur=\z@\fi%
    \ifdim\v@lX<\z@\advance\v@lX-\v@leur\else\advance\v@lX\v@leur\fi%
    \ifdim\v@lY<\z@\advance\v@lY-\v@leur\else\advance\v@lY\v@leur\fi\fi%
    (\@ffichnb{#1}{\repdecn@mb{\v@lX}},\ifmmode\else\thinspace\fi%
    \@ffichnb{#1}{\repdecn@mb{\v@lY}})}}
\ctr@ld@f\def\@ffichnb#1#2{{\def\@@ffich{\@ffich#1(}\edef\n@mbre{#2}%
    \expandafter\@@ffich\n@mbre)}}
\ctr@ld@f\def\@ffich#1(#2.#3){{#2\ifnum#1>\z@.\fi\def\dig@ts{#3}\s@mme=\z@%
    \loop\ifnum\s@mme<#1\expandafter\@ffichdec\dig@ts:\advance\s@mme\@ne\repeat}}
\ctr@ld@f\def\@ffichdec#1#2:{\relax#1\def\dig@ts{#20}}
\ctr@ld@f\def\figcoordTD#1{{\v@lX=\ptT@unit@\v@lX\v@lY=\ptT@unit@\v@lY\v@lZ=\ptT@unit@\v@lZ%
    \ifr@undcoord\ifcase#1\v@leur=0.5pt\or\v@leur=0.05pt\or\v@leur=0.005pt%
    \or\v@leur=0.0005pt\else\v@leur=\z@\fi%
    \ifdim\v@lX<\z@\advance\v@lX-\v@leur\else\advance\v@lX\v@leur\fi%
    \ifdim\v@lY<\z@\advance\v@lY-\v@leur\else\advance\v@lY\v@leur\fi%
    \ifdim\v@lZ<\z@\advance\v@lZ-\v@leur\else\advance\v@lZ\v@leur\fi\fi%
    (\@ffichnb{#1}{\repdecn@mb{\v@lX}},\ifmmode\else\thinspace\fi%
     \@ffichnb{#1}{\repdecn@mb{\v@lY}},\ifmmode\else\thinspace\fi%
     \@ffichnb{#1}{\repdecn@mb{\v@lZ}})}}
\ctr@ld@f\def\figsetroundcoord#1{\expandafter\Figsetr@undcoord#1:\ignorespaces}
\ctr@ld@f\def\Figsetr@undcoord#1#2:{\if#1n\r@undcoordfalse\else\r@undcoordtrue\fi}
\ctr@ld@f\def\Figsetwr@te#1=#2|{\keln@mun#1|%
    \def\n@mref{m}\ifx\l@debut\n@mref\figsetmark{#2}\else
    \immediate\write16{*** Unknown attribute: \BS@ figset (..., #1=...)}%
    \fi}
\ctr@ld@f\def\figsetmark#1{\c@nsymb={#1}\setbox\Gb@xSC=\hbox{\the\c@nsymb}\ignorespaces}
\ctr@ln@m\ptn@me
\ctr@ld@f\def\figsetptname#1{\def\ptn@me##1{#1}\ignorespaces}
\ctr@ld@f\def\FigWrit@L#1:#2(#3,#4){\ignorespaces\@keldist\v@leur{#3}\@keldist\delt@{#4}%
    \C@rp@r@m\def\list@num{#1}\@ecfor\p@int:=\list@num\do{\FigWrit@pt{\p@int}{#2}}}
\ctr@ld@f\def\FigWrit@pt#1#2{\FigWp@r@m{#1}{#2}\Vc@rrect\figWp@si%
    \ifdim\wd\Gb@xSC>\z@\b@undb@xSC{\v@lX}{\v@lY}\fi\figWBB@x}
\ctr@ld@f\def\FigWp@r@m#1#2{\Figg@tXY{#1}%
    \setbox\Gb@x=\hbox{\def\t@xt@{#2}\ifx\t@xt@\empty\Figg@tT{#1}\else#2\fi}\c@lprojSP}
\ctr@ld@f\let\Vc@rrect=\relax
\ctr@ld@f\let\C@rp@r@m=\relax
\ctr@ld@f\def\figwrite[#1]#2{{\ignorespaces\def\list@num{#1}\@ecfor\p@int:=\list@num\do{%
    \setbox\Gb@x=\hbox{\def\t@xt@{#2}\ifx\t@xt@\empty\Figg@tT{\p@int}\else#2\fi}%
    \Figwrit@{\p@int}}}\ignorespaces}
\ctr@ld@f\def\Figwrit@#1{\Figg@tXY{#1}\c@lprojSP%
    \rlap{\kern\v@lX\raise\v@lY\hbox{\unhcopy\Gb@x}}\v@leur=\v@lY%
    \advance\v@lY\ht\Gb@x\b@undb@x{\v@lX}{\v@lY}\advance\v@lX\wd\Gb@x%
    \v@lY=\v@leur\advance\v@lY-\dp\Gb@x\b@undb@x{\v@lX}{\v@lY}}
\ctr@ld@f\def\figwritec[#1]#2{{\ignorespaces\def\list@num{#1}%
    \@ecfor\p@int:=\list@num\do{\Figwrit@c{\p@int}{#2}}}\ignorespaces}
\ctr@ld@f\def\Figwrit@c#1#2{\FigWp@r@m{#1}{#2}%
    \rlap{\kern\v@lX\raise\v@lY\hbox{\rlap{\kern-.5\wd\Gb@x%
    \raise-.5\ht\Gb@x\hbox{\raise.5\dp\Gb@x\hbox{\unhcopy\Gb@x}}}}}%
    \v@leur=\ht\Gb@x\advance\v@leur\dp\Gb@x%
    \advance\v@lX-.5\wd\Gb@x\advance\v@lY-.5\v@leur\b@undb@x{\v@lX}{\v@lY}%
    \advance\v@lX\wd\Gb@x\advance\v@lY\v@leur\b@undb@x{\v@lX}{\v@lY}}
\ctr@ld@f\def\figwritep[#1]{{\ignorespaces\def\list@num{#1}\setbox\Gb@x=\hbox{\c@nterpt}%
    \@ecfor\p@int:=\list@num\do{\Figwrit@{\p@int}}}\ignorespaces}
\ctr@ld@f\def\figwritew#1:#2(#3){\figwritegcw#1:{#2}(#3,0pt)}
\ctr@ld@f\def\figwritee#1:#2(#3){\figwritegce#1:{#2}(#3,0pt)}
\ctr@ld@f\def\figwriten#1:#2(#3){{\def\Vc@rrect{\v@lZ=\v@leur\advance\v@lZ\dp\Gb@x}%
    \Figwrit@NS#1:{#2}(#3)}\ignorespaces}
\ctr@ld@f\def\figwrites#1:#2(#3){{\def\Vc@rrect{\v@lZ=-\v@leur\advance\v@lZ-\ht\Gb@x}%
    \Figwrit@NS#1:{#2}(#3)}\ignorespaces}
\ctr@ld@f\def\Figwrit@NS#1:#2(#3){\let\figWp@si=\FigWp@siNS\let\figWBB@x=\FigWBB@xNS%
    \FigWrit@L#1:{#2}(#3,0pt)}
\ctr@ld@f\def\FigWp@siNS{\rlap{\kern\v@lX\raise\v@lY\hbox{\rlap{\kern-.5\wd\Gb@x%
    \raise\v@lZ\hbox{\unhcopy\Gb@x}}\c@nterpt}}}
\ctr@ld@f\def\FigWBB@xNS{\advance\v@lY\v@lZ%
    \advance\v@lY-\dp\Gb@x\advance\v@lX-.5\wd\Gb@x\b@undb@x{\v@lX}{\v@lY}%
    \advance\v@lY\ht\Gb@x\advance\v@lY\dp\Gb@x%
    \advance\v@lX\wd\Gb@x\b@undb@x{\v@lX}{\v@lY}}
\ctr@ld@f\def\figwritenw#1:#2(#3){{\let\figWp@si=\FigWp@sigW\let\figWBB@x=\FigWBB@xgWE%
    \def\C@rp@r@m{\v@leur=\unssqrttw@\v@leur\delt@=\v@leur%
    \ifdim\delt@=\z@\delt@=\epsil@n\fi}\let@xte={-}\FigWrit@L#1:{#2}(#3,0pt)}\ignorespaces}
\ctr@ld@f\def\figwritesw#1:#2(#3){{\let\figWp@si=\FigWp@sigW\let\figWBB@x=\FigWBB@xgWE%
    \def\C@rp@r@m{\v@leur=\unssqrttw@\v@leur\delt@=-\v@leur%
    \ifdim\delt@=\z@\delt@=-\epsil@n\fi}\let@xte={-}\FigWrit@L#1:{#2}(#3,0pt)}\ignorespaces}
\ctr@ld@f\def\figwritene#1:#2(#3){{\let\figWp@si=\FigWp@sigE\let\figWBB@x=\FigWBB@xgWE%
    \def\C@rp@r@m{\v@leur=\unssqrttw@\v@leur\delt@=\v@leur%
    \ifdim\delt@=\z@\delt@=\epsil@n\fi}\let@xte={}\FigWrit@L#1:{#2}(#3,0pt)}\ignorespaces}
\ctr@ld@f\def\figwritese#1:#2(#3){{\let\figWp@si=\FigWp@sigE\let\figWBB@x=\FigWBB@xgWE%
    \def\C@rp@r@m{\v@leur=\unssqrttw@\v@leur\delt@=-\v@leur%
    \ifdim\delt@=\z@\delt@=-\epsil@n\fi}\let@xte={}\FigWrit@L#1:{#2}(#3,0pt)}\ignorespaces}
\ctr@ld@f\def\figwritegw#1:#2(#3,#4){{\let\figWp@si=\FigWp@sigW\let\figWBB@x=\FigWBB@xgWE%
    \let@xte={-}\FigWrit@L#1:{#2}(#3,#4)}\ignorespaces}
\ctr@ld@f\def\figwritege#1:#2(#3,#4){{\let\figWp@si=\FigWp@sigE\let\figWBB@x=\FigWBB@xgWE%
    \let@xte={}\FigWrit@L#1:{#2}(#3,#4)}\ignorespaces}
\ctr@ld@f\def\FigWp@sigW{\v@lXa=\z@\v@lYa=\ht\Gb@x\advance\v@lYa\dp\Gb@x%
    \ifdim\delt@>\z@\relax%
    \rlap{\kern\v@lX\raise\v@lY\hbox{\rlap{\kern-\wd\Gb@x\kern-\v@leur%
          \raise\delt@\hbox{\raise\dp\Gb@x\hbox{\unhcopy\Gb@x}}}\c@nterpt}}%
    \else\ifdim\delt@<\z@\relax\v@lYa=-\v@lYa%
    \rlap{\kern\v@lX\raise\v@lY\hbox{\rlap{\kern-\wd\Gb@x\kern-\v@leur%
          \raise\delt@\hbox{\raise-\ht\Gb@x\hbox{\unhcopy\Gb@x}}}\c@nterpt}}%
    \else\v@lXa=-.5\v@lYa%
    \rlap{\kern\v@lX\raise\v@lY\hbox{\rlap{\kern-\wd\Gb@x\kern-\v@leur%
          \raise-.5\ht\Gb@x\hbox{\raise.5\dp\Gb@x\hbox{\unhcopy\Gb@x}}}\c@nterpt}}%
    \fi\fi}
\ctr@ld@f\def\FigWp@sigE{\v@lXa=\z@\v@lYa=\ht\Gb@x\advance\v@lYa\dp\Gb@x%
    \ifdim\delt@>\z@\relax%
    \rlap{\kern\v@lX\raise\v@lY\hbox{\c@nterpt\kern\v@leur%
          \raise\delt@\hbox{\raise\dp\Gb@x\hbox{\unhcopy\Gb@x}}}}%
    \else\ifdim\delt@<\z@\relax\v@lYa=-\v@lYa%
    \rlap{\kern\v@lX\raise\v@lY\hbox{\c@nterpt\kern\v@leur%
          \raise\delt@\hbox{\raise-\ht\Gb@x\hbox{\unhcopy\Gb@x}}}}%
    \else\v@lXa=-.5\v@lYa%
    \rlap{\kern\v@lX\raise\v@lY\hbox{\c@nterpt\kern\v@leur%
          \raise-.5\ht\Gb@x\hbox{\raise.5\dp\Gb@x\hbox{\unhcopy\Gb@x}}}}%
    \fi\fi}
\ctr@ld@f\def\FigWBB@xgWE{\advance\v@lY\delt@%
    \advance\v@lX\the\let@xte\v@leur\advance\v@lY\v@lXa\b@undb@x{\v@lX}{\v@lY}%
    \advance\v@lX\the\let@xte\wd\Gb@x\advance\v@lY\v@lYa\b@undb@x{\v@lX}{\v@lY}}
\ctr@ld@f\def\figwritegcw#1:#2(#3,#4){{\let\figWp@si=\FigWp@sigcW\let\figWBB@x=\FigWBB@xgcWE%
    \let@xte={-}\FigWrit@L#1:{#2}(#3,#4)}\ignorespaces}
\ctr@ld@f\def\figwritegce#1:#2(#3,#4){{\let\figWp@si=\FigWp@sigcE\let\figWBB@x=\FigWBB@xgcWE%
    \let@xte={}\FigWrit@L#1:{#2}(#3,#4)}\ignorespaces}
\ctr@ld@f\def\FigWp@sigcW{\rlap{\kern\v@lX\raise\v@lY\hbox{\rlap{\kern-\wd\Gb@x\kern-\v@leur%
     \raise-.5\ht\Gb@x\hbox{\raise\delt@\hbox{\raise.5\dp\Gb@x\hbox{\unhcopy\Gb@x}}}}%
     \c@nterpt}}}
\ctr@ld@f\def\FigWp@sigcE{\rlap{\kern\v@lX\raise\v@lY\hbox{\c@nterpt\kern\v@leur%
    \raise-.5\ht\Gb@x\hbox{\raise\delt@\hbox{\raise.5\dp\Gb@x\hbox{\unhcopy\Gb@x}}}}}}
\ctr@ld@f\def\FigWBB@xgcWE{\v@lZ=\ht\Gb@x\advance\v@lZ\dp\Gb@x%
    \advance\v@lX\the\let@xte\v@leur\advance\v@lY\delt@\advance\v@lY.5\v@lZ%
    \b@undb@x{\v@lX}{\v@lY}%
    \advance\v@lX\the\let@xte\wd\Gb@x\advance\v@lY-\v@lZ\b@undb@x{\v@lX}{\v@lY}}
\ctr@ld@f\def\figwritebn#1:#2(#3){{\def\Vc@rrect{\v@lZ=\v@leur}\Figwrit@NS#1:{#2}(#3)}\ignorespaces}
\ctr@ld@f\def\figwritebs#1:#2(#3){{\def\Vc@rrect{\v@lZ=-\v@leur}\Figwrit@NS#1:{#2}(#3)}\ignorespaces}
\ctr@ld@f\def\figwritebw#1:#2(#3){{\let\figWp@si=\FigWp@sibW\let\figWBB@x=\FigWBB@xbWE%
    \let@xte={-}\FigWrit@L#1:{#2}(#3,0pt)}\ignorespaces}
\ctr@ld@f\def\figwritebe#1:#2(#3){{\let\figWp@si=\FigWp@sibE\let\figWBB@x=\FigWBB@xbWE%
    \let@xte={}\FigWrit@L#1:{#2}(#3,0pt)}\ignorespaces}
\ctr@ld@f\def\FigWp@sibW{\rlap{\kern\v@lX\raise\v@lY\hbox{\rlap{\kern-\wd\Gb@x\kern-\v@leur%
          \hbox{\unhcopy\Gb@x}}\c@nterpt}}}
\ctr@ld@f\def\FigWp@sibE{\rlap{\kern\v@lX\raise\v@lY\hbox{\c@nterpt\kern\v@leur%
          \hbox{\unhcopy\Gb@x}}}}
\ctr@ld@f\def\FigWBB@xbWE{\v@lZ=\ht\Gb@x\advance\v@lZ\dp\Gb@x%
    \advance\v@lX\the\let@xte\v@leur\advance\v@lY\ht\Gb@x\b@undb@x{\v@lX}{\v@lY}%
    \advance\v@lX\the\let@xte\wd\Gb@x\advance\v@lY-\v@lZ\b@undb@x{\v@lX}{\v@lY}}
\ctr@ln@w{newread}\frf@g  \ctr@ln@w{newwrite}\fwf@g
\ctr@ln@w{newif}\ifcurr@ntPS
\ctr@ln@w{newif}\ifps@cri
\ctr@ln@w{newif}\ifUse@llipse
\ctr@ln@w{newif}\ifpsdebugmode \psdebugmodefalse
\ctr@ln@w{newif}\ifPDFm@ke
\ifx\pdfliteral\undefined\else\ifnum\pdfoutput>\z@\PDFm@ketrue\fi\fi
\ctr@ld@f\def\initPDF@rDVI{%
\ifPDFm@ke
 \let\figscan=\figscan@E
 \let\newGr@FN=\newGr@FNPDF
 \ctr@ld@f\def\c@mcurveto{c}
 \ctr@ld@f\def\c@mfill{f}
 \ctr@ld@f\def\c@mgsave{q}
 \ctr@ld@f\def\c@mgrestore{Q}
 \ctr@ld@f\def\c@mlineto{l}
 \ctr@ld@f\def\c@mmoveto{m}
 \ctr@ld@f\def\c@msetgray{g}     \ctr@ld@f\def\c@msetgrayStroke{G}
 \ctr@ld@f\def\c@msetcmykcolor{k}\ctr@ld@f\def\c@msetcmykcolorStroke{K}
 \ctr@ld@f\def\c@msetrgbcolor{rg}\ctr@ld@f\def\c@msetrgbcolorStroke{RG}
 \ctr@ld@f\def\d@fprimarC@lor{\curr@ntcolor\space\curr@ntcolorc@md%
               \space\curr@ntcolor\space\curr@ntcolorc@mdStroke}
 \ctr@ld@f\def\d@fsecondC@lor{\sec@ndcolor\space\sec@ndcolorc@md%
               \space\sec@ndcolor\space\sec@ndcolorc@mdStroke}
 \ctr@ld@f\def\d@fthirdC@lor{\th@rdcolor\space\th@rdcolorc@md%
              \space\th@rdcolor\space\th@rdcolorc@mdStroke}
 \ctr@ld@f\def\c@msetdash{d}
 \ctr@ld@f\def\c@msetlinejoin{j}
 \ctr@ld@f\def\c@msetlinewidth{w}
 \ctr@ld@f\def\f@gclosestroke{\immediate\write\fwf@g{s}}
 \ctr@ld@f\def\f@gfill{\immediate\write\fwf@g{\fillc@md}}
 \ctr@ld@f\def\f@gnewpath{}
 \ctr@ld@f\def\f@gstroke{\immediate\write\fwf@g{S}}
\else
 \let\figinsertE=\figinsert
 \let\newGr@FN=\newGr@FNDVI
 \ctr@ld@f\def\c@mcurveto{curveto}
 \ctr@ld@f\def\c@mfill{fill}
 \ctr@ld@f\def\c@mgsave{gsave}
 \ctr@ld@f\def\c@mgrestore{grestore}
 \ctr@ld@f\def\c@mlineto{lineto}
 \ctr@ld@f\def\c@mmoveto{moveto}
 \ctr@ld@f\def\c@msetgray{setgray}          \ctr@ld@f\def\c@msetgrayStroke{}
 \ctr@ld@f\def\c@msetcmykcolor{setcmykcolor}\ctr@ld@f\def\c@msetcmykcolorStroke{}
 \ctr@ld@f\def\c@msetrgbcolor{setrgbcolor}  \ctr@ld@f\def\c@msetrgbcolorStroke{}
 \ctr@ld@f\def\d@fprimarC@lor{\curr@ntcolor\space\curr@ntcolorc@md}
 \ctr@ld@f\def\d@fsecondC@lor{\sec@ndcolor\space\sec@ndcolorc@md}
 \ctr@ld@f\def\d@fthirdC@lor{\th@rdcolor\space\th@rdcolorc@md}
 \ctr@ld@f\def\c@msetdash{setdash}
 \ctr@ld@f\def\c@msetlinejoin{setlinejoin}
 \ctr@ld@f\def\c@msetlinewidth{setlinewidth}
 \ctr@ld@f\def\f@gclosestroke{\immediate\write\fwf@g{closepath\space stroke}}
 \ctr@ld@f\def\f@gfill{\immediate\write\fwf@g{\fillc@md}}
 \ctr@ld@f\def\f@gnewpath{\immediate\write\fwf@g{newpath}}
 \ctr@ld@f\def\f@gstroke{\immediate\write\fwf@g{stroke}}
\fi}
\ctr@ld@f\def\c@pypsfile#1#2{\c@pyfil@{\immediate\write#1}{#2}}
\ctr@ld@f\def\Figinclud@PDF#1#2{\openin\frf@g=#1\pdfliteral{q #2 0 0 #2 0 0 cm}%
    \c@pyfil@{\pdfliteral}{\frf@g}\pdfliteral{Q}\closein\frf@g}
\ctr@ln@w{newif}\ifmored@ta
\ctr@ln@m\bl@nkline
\ctr@ld@f\def\c@pyfil@#1#2{\def\bl@nkline{\par}{\catcode`\%=12
    \loop\ifeof#2\mored@tafalse\else\mored@tatrue\immediate\read#2 to\tr@c
    \ifx\tr@c\bl@nkline\else#1{\tr@c}\fi\fi\ifmored@ta\repeat}}
\ctr@ld@f\def\keln@mun#1#2|{\def\l@debut{#1}\def\l@suite{#2}}
\ctr@ld@f\def\keln@mde#1#2#3|{\def\l@debut{#1#2}\def\l@suite{#3}}
\ctr@ld@f\def\keln@mtr#1#2#3#4|{\def\l@debut{#1#2#3}\def\l@suite{#4}}
\ctr@ld@f\def\keln@mqu#1#2#3#4#5|{\def\l@debut{#1#2#3#4}\def\l@suite{#5}}
\ctr@ld@f\let\@psffilein=\frf@g 
\ctr@ln@w{newif}\if@psffileok    
\ctr@ln@w{newif}\if@psfbbfound   
\ctr@ln@w{newif}\if@psfverbose   
\@psfverbosetrue
\ctr@ln@m\@psfllx \ctr@ln@m\@psflly
\ctr@ln@m\@psfurx \ctr@ln@m\@psfury
\ctr@ln@m\resetcolonc@tcode
\ctr@ld@f\def\@psfgetbb#1{\global\@psfbbfoundfalse%
\global\def\@psfllx{0}\global\def\@psflly{0}%
\global\def\@psfurx{30}\global\def\@psfury{30}%
\openin\@psffilein=#1\relax
\ifeof\@psffilein\errmessage{I couldn't open #1, will ignore it}\else
   \edef\resetcolonc@tcode{\catcode`\noexpand\:\the\catcode`\:\relax}%
   {\@psffileoktrue \chardef\other=12
    \def\do##1{\catcode`##1=\other}\dospecials \catcode`\ =10 \resetcolonc@tcode
    \loop
       \read\@psffilein to \@psffileline
       \ifeof\@psffilein\@psffileokfalse\else
          \expandafter\@psfaux\@psffileline:. \\%
       \fi
   \if@psffileok\repeat
   \if@psfbbfound\else
    \if@psfverbose\message{No bounding box comment in #1; using defaults}\fi\fi
   }\closein\@psffilein\fi}%
\ctr@ln@m\@psfbblit
\ctr@ln@m\@psfpercent
{\catcode`\%=12 \global\let\@psfpercent=
\ctr@ln@m\@psfaux
\long\def\@psfaux#1#2:#3\\{\ifx#1\@psfpercent
   \def\testit{#2}\ifx\testit\@psfbblit
      \@psfgrab #3 . . . \\%
      \@psffileokfalse
      \global\@psfbbfoundtrue
   \fi\else\ifx#1\par\else\@psffileokfalse\fi\fi}%
\ctr@ld@f\def\@psfempty{}%
\ctr@ld@f\def\@psfgrab #1 #2 #3 #4 #5\\{%
\global\def\@psfllx{#1}\ifx\@psfllx\@psfempty
      \@psfgrab #2 #3 #4 #5 .\\\else
   \global\def\@psflly{#2}%
   \global\def\@psfurx{#3}\global\def\@psfury{#4}\fi}%
\ctr@ld@f\def\PSwrit@cmd#1#2#3{{\Figg@tXY{#1}\c@lprojSP\b@undb@x{\v@lX}{\v@lY}%
    \v@lX=\ptT@ptps\v@lX\v@lY=\ptT@ptps\v@lY%
    \immediate\write#3{\repdecn@mb{\v@lX}\space\repdecn@mb{\v@lY}\space#2}}}
\ctr@ld@f\def\PSwrit@cmdS#1#2#3#4#5{{\Figg@tXY{#1}\c@lprojSP\b@undb@x{\v@lX}{\v@lY}%
    \global\result@t=\v@lX\global\result@@t=\v@lY%
    \v@lX=\ptT@ptps\v@lX\v@lY=\ptT@ptps\v@lY%
    \immediate\write#3{\repdecn@mb{\v@lX}\space\repdecn@mb{\v@lY}\space#2}}%
    \edef#4{\the\result@t}\edef#5{\the\result@@t}}
\ctr@ld@f\def\psaltitude#1[#2,#3,#4]{{\ifcurr@ntPS\ifps@cri%
    \PSc@mment{psaltitude Square Dim=#1, Triangle=[#2 / #3,#4]}%
    \s@uvc@ntr@l\et@tpsaltitude\resetc@ntr@l{2}\figptorthoprojline-5:=#2/#3,#4/%
    \figvectP -1[#3,#4]\n@rminf{\v@leur}{-1}\vecunit@{-3}{-1}%
    \figvectP -1[-5,#3]\n@rminf{\v@lmin}{-1}\figvectP -2[-5,#4]\n@rminf{\v@lmax}{-2}%
    \ifdim\v@lmin<\v@lmax\s@mme=#3\else\v@lmax=\v@lmin\s@mme=#4\fi%
    \figvectP -4[-5,#2]\vecunit@{-4}{-4}\delt@=#1\unit@%
    \edef\t@ille{\repdecn@mb{\delt@}}\figpttra-1:=-5/\t@ille,-3/%
    \figptstra-3=-5,-1/\t@ille,-4/\psline[#2,-5]\s@uvdash{\typ@dash}%
    \pssetdash{\defaultdash}\psline[-1,-2,-3]\pssetdash{\typ@dash}%
    \ifdim\v@leur<\v@lmax\Pss@tsecondSt\psline[-5,\the\s@mme]\Psrest@reSt\fi%
    \PSc@mment{End psaltitude}\resetc@ntr@l\et@tpsaltitude\fi\fi}}
\ctr@ld@f\def\Ps@rcerc#1;#2(#3,#4){\ellBB@x#1;#2,#2(#3,#4,0)%
    \f@gnewpath{\delt@=#2\unit@\delt@=\ptT@ptps\delt@%
    \BdingB@xfalse%
    \PSwrit@cmd{#1}{\repdecn@mb{\delt@}\space #3\space #4\space arc}{\fwf@g}}}
\ctr@ln@m\psarccirc
\ctr@ld@f\def\psarccircDD#1;#2(#3,#4){\ifcurr@ntPS\ifps@cri%
    \PSc@mment{psarccircDD Center=#1 ; Radius=#2 (Ang1=#3, Ang2=#4)}%
    \iffillm@de\Ps@rcerc#1;#2(#3,#4)%
    \f@gfill%
    \else\Ps@rcerc#1;#2(#3,#4)\f@gstroke\fi%
    \PSc@mment{End psarccircDD}\fi\fi}
\ctr@ld@f\def\psarccircTD#1,#2,#3;#4(#5,#6){{\ifcurr@ntPS\ifps@cri\s@uvc@ntr@l\et@tpsarccircTD%
    \PSc@mment{psarccircTD Center=#1,P1=#2,P2=#3 ; Radius=#4 (Ang1=#5, Ang2=#6)}%
    \setc@ntr@l{2}\c@lExtAxes#1,#2,#3(#4)\psarcellPATD#1,-4,-5(#5,#6)%
    \PSc@mment{End psarccircTD}\resetc@ntr@l\et@tpsarccircTD\fi\fi}}
\ctr@ld@f\def\c@lExtAxes#1,#2,#3(#4){%
    \figvectPTD-5[#1,#2]\vecunit@{-5}{-5}\figvectNTD-4[#1,#2,#3]\vecunit@{-4}{-4}%
    \figvectNVTD-3[-4,-5]\delt@=#4\unit@\edef\r@yon{\repdecn@mb{\delt@}}%
    \figpttra-4:=#1/\r@yon,-5/\figpttra-5:=#1/\r@yon,-3/}
\ctr@ln@m\psarccircP
\ctr@ld@f\def\psarccircPDD#1;#2[#3,#4]{{\ifcurr@ntPS\ifps@cri\s@uvc@ntr@l\et@tpsarccircPDD%
    \PSc@mment{psarccircPDD Center=#1; Radius=#2, [P1=#3, P2=#4]}%
    \Ps@ngleparam#1;#2[#3,#4]\ifdim\v@lmin>\v@lmax\advance\v@lmax\DePI@deg\fi%
    \edef\@ngdeb{\repdecn@mb{\v@lmin}}\edef\@ngfin{\repdecn@mb{\v@lmax}}%
    \psarccirc#1;\r@dius(\@ngdeb,\@ngfin)%
    \PSc@mment{End psarccircPDD}\resetc@ntr@l\et@tpsarccircPDD\fi\fi}}
\ctr@ld@f\def\psarccircPTD#1;#2[#3,#4,#5]{{\ifcurr@ntPS\ifps@cri\s@uvc@ntr@l\et@tpsarccircPTD%
    \PSc@mment{psarccircPTD Center=#1; Radius=#2, [P1=#3, P2=#4, P3=#5]}%
    \setc@ntr@l{2}\c@lExtAxes#1,#3,#5(#2)\psarcellPP#1,-4,-5[#3,#4]%
    \PSc@mment{End psarccircPTD}\resetc@ntr@l\et@tpsarccircPTD\fi\fi}}
\ctr@ld@f\def\Ps@ngleparam#1;#2[#3,#4]{\setc@ntr@l{2}%
    \figvectPDD-1[#1,#3]\vecunit@{-1}{-1}\Figg@tXY{-1}\arct@n\v@lmin(\v@lX,\v@lY)%
    \figvectPDD-2[#1,#4]\vecunit@{-2}{-2}\Figg@tXY{-2}\arct@n\v@lmax(\v@lX,\v@lY)%
    \v@lmin=\rdT@deg\v@lmin\v@lmax=\rdT@deg\v@lmax%
    \v@leur=#2pt\maxim@m{\mili@u}{-\v@leur}{\v@leur}%
    \edef\r@dius{\repdecn@mb{\mili@u}}}
\ctr@ld@f\def\Ps@rcercBz#1;#2(#3,#4){\Ps@rellBz#1;#2,#2(#3,#4,0)}
\ctr@ld@f\def\Ps@rellBz#1;#2,#3(#4,#5,#6){%
    \ellBB@x#1;#2,#3(#4,#5,#6)\BdingB@xfalse%
    \c@lNbarcs{#4}{#5}\v@leur=#4pt\setc@ntr@l{2}\figptell-13::#1;#2,#3(#4,#6)%
    \f@gnewpath\PSwrit@cmd{-13}{\c@mmoveto}{\fwf@g}%
    \s@mme=\z@\bcl@rellBz#1;#2,#3(#6)\BdingB@xtrue}
\ctr@ld@f\def\bcl@rellBz#1;#2,#3(#4){\relax%
    \ifnum\s@mme<\p@rtent\advance\s@mme\@ne%
    \advance\v@leur\delt@\edef\@ngle{\repdecn@mb\v@leur}\figptell-14::#1;#2,#3(\@ngle,#4)%
    \advance\v@leur\delt@\edef\@ngle{\repdecn@mb\v@leur}\figptell-15::#1;#2,#3(\@ngle,#4)%
    \advance\v@leur\delt@\edef\@ngle{\repdecn@mb\v@leur}\figptell-16::#1;#2,#3(\@ngle,#4)%
    \figptscontrolDD-18[-13,-14,-15,-16]%
    \PSwrit@cmd{-18}{}{\fwf@g}\PSwrit@cmd{-17}{}{\fwf@g}%
    \PSwrit@cmd{-16}{\c@mcurveto}{\fwf@g}%
    \figptcopyDD-13:/-16/\bcl@rellBz#1;#2,#3(#4)\fi}
\ctr@ld@f\def\Ps@rell#1;#2,#3(#4,#5,#6){\ellBB@x#1;#2,#3(#4,#5,#6)%
    \f@gnewpath{\v@lmin=#2\unit@\v@lmin=\ptT@ptps\v@lmin%
    \v@lmax=#3\unit@\v@lmax=\ptT@ptps\v@lmax\BdingB@xfalse%
    \PSwrit@cmd{#1}%
    {#6\space\repdecn@mb{\v@lmin}\space\repdecn@mb{\v@lmax}\space #4\space #5\space ellipse}{\fwf@g}}%
    \global\Use@llipsetrue}
\ctr@ln@m\psarcell
\ctr@ld@f\def\psarcellDD#1;#2,#3(#4,#5,#6){{\ifcurr@ntPS\ifps@cri%
    \PSc@mment{psarcellDD Center=#1 ; XRad=#2, YRad=#3 (Ang1=#4, Ang2=#5, Inclination=#6)}%
    \iffillm@de\Ps@rell#1;#2,#3(#4,#5,#6)%
    \f@gfill%
    \else\Ps@rell#1;#2,#3(#4,#5,#6)\f@gstroke\fi%
    \PSc@mment{End psarcellDD}\fi\fi}}
\ctr@ld@f\def\psarcellTD#1;#2,#3(#4,#5,#6){{\ifcurr@ntPS\ifps@cri\s@uvc@ntr@l\et@tpsarcellTD%
    \PSc@mment{psarcellTD Center=#1 ; XRad=#2, YRad=#3 (Ang1=#4, Ang2=#5, Inclination=#6)}%
    \setc@ntr@l{2}\figpttraC -8:=#1/#2,0,0/\figpttraC -7:=#1/0,#3,0/%
    \figvectC -4(0,0,1)\figptsrot -8=-8,-7/#1,#6,-4/\psarcellPATD#1,-8,-7(#4,#5)%
    \PSc@mment{End psarcellTD}\resetc@ntr@l\et@tpsarcellTD\fi\fi}}
\ctr@ln@m\psarcellPA
\ctr@ld@f\def\psarcellPADD#1,#2,#3(#4,#5){{\ifcurr@ntPS\ifps@cri\s@uvc@ntr@l\et@tpsarcellPADD%
    \PSc@mment{psarcellPADD Center=#1,PtAxis1=#2,PtAxis2=#3 (Ang1=#4, Ang2=#5)}%
    \setc@ntr@l{2}\figvectPDD-1[#1,#2]\vecunit@DD{-1}{-1}\v@lX=\ptT@unit@\result@t%
    \edef\XR@d{\repdecn@mb{\v@lX}}\Figg@tXY{-1}\arct@n\v@lmin(\v@lX,\v@lY)%
    \v@lmin=\rdT@deg\v@lmin\edef\Inclin@{\repdecn@mb{\v@lmin}}%
    \figgetdist\YR@d[#1,#3]\psarcellDD#1;\XR@d,\YR@d(#4,#5,\Inclin@)%
    \PSc@mment{End psarcellPADD}\resetc@ntr@l\et@tpsarcellPADD\fi\fi}}
\ctr@ld@f\def\psarcellPATD#1,#2,#3(#4,#5){{\ifcurr@ntPS\ifps@cri\s@uvc@ntr@l\et@tpsarcellPATD%
    \PSc@mment{psarcellPATD Center=#1,PtAxis1=#2,PtAxis2=#3 (Ang1=#4, Ang2=#5)}%
    \iffillm@de\Ps@rellPATD#1,#2,#3(#4,#5)%
    \f@gfill%
    \else\Ps@rellPATD#1,#2,#3(#4,#5)\f@gstroke\fi%
    \PSc@mment{End psarcellPATD}\resetc@ntr@l\et@tpsarcellPATD\fi\fi}}
\ctr@ld@f\def\Ps@rellPATD#1,#2,#3(#4,#5){\let\c@lprojSP=\relax%
    \setc@ntr@l{2}\figvectPTD-1[#1,#2]\figvectPTD-2[#1,#3]\c@lNbarcs{#4}{#5}%
    \v@leur=#4pt\c@lptellP{#1}{-1}{-2}\Figptpr@j-5:/-3/%
    \f@gnewpath\PSwrit@cmdS{-5}{\c@mmoveto}{\fwf@g}{\X@un}{\Y@un}%
    \edef\C@nt@r{#1}\s@mme=\z@\bcl@rellPATD}
\ctr@ld@f\def\bcl@rellPATD{\relax%
    \ifnum\s@mme<\p@rtent\advance\s@mme\@ne%
    \advance\v@leur\delt@\c@lptellP{\C@nt@r}{-1}{-2}\Figptpr@j-4:/-3/%
    \advance\v@leur\delt@\c@lptellP{\C@nt@r}{-1}{-2}\Figptpr@j-6:/-3/%
    \advance\v@leur\delt@\c@lptellP{\C@nt@r}{-1}{-2}\Figptpr@j-3:/-3/%
    \v@lX=\z@\v@lY=\z@\Figtr@nptDD{-5}{-5}\Figtr@nptDD{2}{-3}%
    \divide\v@lX\@vi\divide\v@lY\@vi%
    \Figtr@nptDD{3}{-4}\Figtr@nptDD{-1.5}{-6}\v@lmin=\v@lX\v@lmax=\v@lY%
    \v@lX=\z@\v@lY=\z@\Figtr@nptDD{2}{-5}\Figtr@nptDD{-5}{-3}%
    \divide\v@lX\@vi\divide\v@lY\@vi\Figtr@nptDD{-1.5}{-4}\Figtr@nptDD{3}{-6}%
    \BdingB@xfalse%
    \Figp@intregDD-4:(\v@lmin,\v@lmax)\PSwrit@cmdS{-4}{}{\fwf@g}{\X@de}{\Y@de}%
    \Figp@intregDD-4:(\v@lX,\v@lY)\PSwrit@cmdS{-4}{}{\fwf@g}{\X@tr}{\Y@tr}%
    \BdingB@xtrue\PSwrit@cmdS{-3}{\c@mcurveto}{\fwf@g}{\X@qu}{\Y@qu}%
    \B@zierBB@x{1}{\Y@un}(\X@un,\X@de,\X@tr,\X@qu)%
    \B@zierBB@x{2}{\X@un}(\Y@un,\Y@de,\Y@tr,\Y@qu)%
    \edef\X@un{\X@qu}\edef\Y@un{\Y@qu}\figptcopyDD-5:/-3/\bcl@rellPATD\fi}
\ctr@ld@f\def\c@lNbarcs#1#2{%
    \delt@=#2pt\advance\delt@-#1pt\maxim@m{\v@lmax}{\delt@}{-\delt@}%
    \v@leur=\v@lmax\divide\v@leur45 \p@rtentiere{\p@rtent}{\v@leur}\advance\p@rtent\@ne%
    \s@mme=\p@rtent\multiply\s@mme\thr@@\divide\delt@\s@mme}
\ctr@ld@f\def\psarcellPP#1,#2,#3[#4,#5]{{\ifcurr@ntPS\ifps@cri\s@uvc@ntr@l\et@tpsarcellPP%
    \PSc@mment{psarcellPP Center=#1,PtAxis1=#2,PtAxis2=#3 [Point1=#4, Point2=#5]}%
    \setc@ntr@l{2}\figvectP-2[#1,#3]\vecunit@{-2}{-2}\v@lmin=\result@t%
    \invers@{\v@lmax}{\v@lmin}%
    \figvectP-1[#1,#2]\vecunit@{-1}{-1}\v@leur=\result@t%
    \v@leur=\repdecn@mb{\v@lmax}\v@leur\edef\AsB@{\repdecn@mb{\v@leur}}
    \c@lAngle{#1}{#4}{\v@lmin}\edef\@ngdeb{\repdecn@mb{\v@lmin}}%
    \c@lAngle{#1}{#5}{\v@lmax}\ifdim\v@lmin>\v@lmax\advance\v@lmax\DePI@deg\fi%
    \edef\@ngfin{\repdecn@mb{\v@lmax}}\psarcellPA#1,#2,#3(\@ngdeb,\@ngfin)%
    \PSc@mment{End psarcellPP}\resetc@ntr@l\et@tpsarcellPP\fi\fi}}
\ctr@ld@f\def\c@lAngle#1#2#3{\figvectP-3[#1,#2]%
    \c@lproscal\delt@[-3,-1]\c@lproscal\v@leur[-3,-2]%
    \v@leur=\AsB@\v@leur\arct@n#3(\delt@,\v@leur)#3=\rdT@deg#3}
\ctr@ln@w{newif}\if@rrowratio\@rrowratiotrue
\ctr@ln@w{newif}\if@rrowhfill
\ctr@ln@w{newif}\if@rrowhout
\ctr@ld@f\def\Psset@rrowhe@d#1=#2|{\keln@mun#1|%
    \def\n@mref{a}\ifx\l@debut\n@mref\pssetarrowheadangle{#2}\else
    \def\n@mref{f}\ifx\l@debut\n@mref\pssetarrowheadfill{#2}\else
    \def\n@mref{l}\ifx\l@debut\n@mref\pssetarrowheadlength{#2}\else
    \def\n@mref{o}\ifx\l@debut\n@mref\pssetarrowheadout{#2}\else
    \def\n@mref{r}\ifx\l@debut\n@mref\pssetarrowheadratio{#2}\else
    \immediate\write16{*** Unknown attribute: \BS@ psset arrowhead(..., #1=...)}%
    \fi\fi\fi\fi\fi}
\ctr@ln@m\@rrowheadangle
\ctr@ln@m\C@AHANG \ctr@ln@m\S@AHANG \ctr@ln@m\UNSS@N
\ctr@ld@f\def\pssetarrowheadangle#1{\edef\@rrowheadangle{#1}{\c@ssin{\C@}{\S@}{#1}%
    \xdef\C@AHANG{\C@}\xdef\S@AHANG{\S@}\v@lmax=\S@ pt%
    \invers@{\v@leur}{\v@lmax}\maxim@m{\v@leur}{\v@leur}{-\v@leur}%
    \xdef\UNSS@N{\the\v@leur}}}
\ctr@ld@f\def\pssetarrowheadfill#1{\expandafter\set@rrowhfill#1:}
\ctr@ld@f\def\set@rrowhfill#1#2:{\if#1n\@rrowhfillfalse\else\@rrowhfilltrue\fi}
\ctr@ld@f\def\pssetarrowheadout#1{\expandafter\set@rrowhout#1:}
\ctr@ld@f\def\set@rrowhout#1#2:{\if#1n\@rrowhoutfalse\else\@rrowhouttrue\fi}
\ctr@ln@m\@rrowheadlength
\ctr@ld@f\def\pssetarrowheadlength#1{\edef\@rrowheadlength{#1}\@rrowratiofalse}
\ctr@ln@m\@rrowheadratio
\ctr@ld@f\def\pssetarrowheadratio#1{\edef\@rrowheadratio{#1}\@rrowratiotrue}
\ctr@ln@m\defaultarrowheadlength
\ctr@ld@f\def\psresetarrowhead{%
    \pssetarrowheadangle{\defaultarrowheadangle}%
    \pssetarrowheadfill{\defaultarrowheadfill}%
    \pssetarrowheadout{\defaultarrowheadout}%
    \pssetarrowheadratio{\defaultarrowheadratio}%
    \d@fm@cdim\defaultarrowheadlength{\defaulth@rdahlength}
    \pssetarrowheadlength{\defaultarrowheadlength}}
\ctr@ld@f\def\defaultarrowheadratio{0.1}
\ctr@ld@f\def\defaultarrowheadangle{20}
\ctr@ld@f\def\defaultarrowheadfill{no}
\ctr@ld@f\def\defaultarrowheadout{no}
\ctr@ld@f\def\defaulth@rdahlength{8pt}
\ctr@ln@m\psarrow
\ctr@ld@f\def\psarrowDD[#1,#2]{{\ifcurr@ntPS\ifps@cri\s@uvc@ntr@l\et@tpsarrow%
    \PSc@mment{psarrowDD [Pt1,Pt2]=[#1,#2]}\pssetfillmode{no}%
    \psarrowheadDD[#1,#2]\setc@ntr@l{2}\psline[#1,-3]%
    \PSc@mment{End psarrowDD}\resetc@ntr@l\et@tpsarrow\fi\fi}}
\ctr@ld@f\def\psarrowTD[#1,#2]{{\ifcurr@ntPS\ifps@cri\s@uvc@ntr@l\et@tpsarrowTD%
    \PSc@mment{psarrowTD [Pt1,Pt2]=[#1,#2]}\resetc@ntr@l{2}%
    \Figptpr@j-5:/#1/\Figptpr@j-6:/#2/\let\c@lprojSP=\relax\psarrowDD[-5,-6]%
    \PSc@mment{End psarrowTD}\resetc@ntr@l\et@tpsarrowTD\fi\fi}}
\ctr@ln@m\psarrowhead
\ctr@ld@f\def\psarrowheadDD[#1,#2]{{\ifcurr@ntPS\ifps@cri\s@uvc@ntr@l\et@tpsarrowheadDD%
    \if@rrowhfill\def\@hangle{-\@rrowheadangle}\else\def\@hangle{\@rrowheadangle}\fi%
    \if@rrowratio%
    \if@rrowhout\def\@hratio{-\@rrowheadratio}\else\def\@hratio{\@rrowheadratio}\fi%
    \PSc@mment{psarrowheadDD Ratio=\@hratio, Angle=\@hangle, [Pt1,Pt2]=[#1,#2]}%
    \Ps@rrowhead\@hratio,\@hangle[#1,#2]%
    \else%
    \if@rrowhout\def\@hlength{-\@rrowheadlength}\else\def\@hlength{\@rrowheadlength}\fi%
    \PSc@mment{psarrowheadDD Length=\@hlength, Angle=\@hangle, [Pt1,Pt2]=[#1,#2]}%
    \Ps@rrowheadfd\@hlength,\@hangle[#1,#2]%
    \fi%
    \PSc@mment{End psarrowheadDD}\resetc@ntr@l\et@tpsarrowheadDD\fi\fi}}
\ctr@ld@f\def\psarrowheadTD[#1,#2]{{\ifcurr@ntPS\ifps@cri\s@uvc@ntr@l\et@tpsarrowheadTD%
    \PSc@mment{psarrowheadTD [Pt1,Pt2]=[#1,#2]}\resetc@ntr@l{2}%
    \Figptpr@j-5:/#1/\Figptpr@j-6:/#2/\let\c@lprojSP=\relax\psarrowheadDD[-5,-6]%
    \PSc@mment{End psarrowheadTD}\resetc@ntr@l\et@tpsarrowheadTD\fi\fi}}
\ctr@ld@f\def\Ps@rrowhead#1,#2[#3,#4]{\v@leur=#1\p@\maxim@m{\v@leur}{\v@leur}{-\v@leur}%
    \ifdim\v@leur>\Cepsil@n{
    \PSc@mment{ps@rrowhead Ratio=#1, Angle=#2, [Pt1,Pt2]=[#3,#4]}\v@leur=\UNSS@N%
    \v@leur=\curr@ntwidth\v@leur\v@leur=\ptpsT@pt\v@leur\delt@=.5\v@leur
    \setc@ntr@l{2}\figvectPDD-3[#4,#3]%
    \Figg@tXY{-3}\v@lX=#1\v@lX\v@lY=#1\v@lY\Figv@ctCreg-3(\v@lX,\v@lY)%
    \vecunit@{-4}{-3}\mili@u=\result@t%
    \ifdim#2pt>\z@\v@lXa=-\C@AHANG\delt@%
     \edef\c@ef{\repdecn@mb{\v@lXa}}\figpttraDD-3:=-3/\c@ef,-4/\fi%
    \edef\c@ef{\repdecn@mb{\delt@}}%
    \v@lXa=\mili@u\v@lXa=\C@AHANG\v@lXa%
    \v@lYa=\ptpsT@pt\p@\v@lYa=\curr@ntwidth\v@lYa\v@lYa=\sDcc@ngle\v@lYa%
    \advance\v@lXa-\v@lYa\gdef\sDcc@ngle{0}%
    \ifdim\v@lXa>\v@leur\edef\c@efendpt{\repdecn@mb{\v@leur}}%
    \else\edef\c@efendpt{\repdecn@mb{\v@lXa}}\fi%
    \Figg@tXY{-3}\v@lmin=\v@lX\v@lmax=\v@lY%
    \v@lXa=\C@AHANG\v@lmin\v@lYa=\S@AHANG\v@lmax\advance\v@lXa\v@lYa%
    \v@lYa=-\S@AHANG\v@lmin\v@lX=\C@AHANG\v@lmax\advance\v@lYa\v@lX%
    \setc@ntr@l{1}\Figg@tXY{#4}\advance\v@lX\v@lXa\advance\v@lY\v@lYa%
    \setc@ntr@l{2}\Figp@intregDD-2:(\v@lX,\v@lY)%
    \v@lXa=\C@AHANG\v@lmin\v@lYa=-\S@AHANG\v@lmax\advance\v@lXa\v@lYa%
    \v@lYa=\S@AHANG\v@lmin\v@lX=\C@AHANG\v@lmax\advance\v@lYa\v@lX%
    \setc@ntr@l{1}\Figg@tXY{#4}\advance\v@lX\v@lXa\advance\v@lY\v@lYa%
    \setc@ntr@l{2}\Figp@intregDD-1:(\v@lX,\v@lY)%
    \ifdim#2pt<\z@\fillm@detrue\psline[-2,#4,-1]
    \else\figptstraDD-3=#4,-2,-1/\c@ef,-4/\psline[-2,-3,-1]\fi
    \ifdim#1pt>\z@\figpttraDD-3:=#4/\c@efendpt,-4/\else\figptcopyDD-3:/#4/\fi%
    \PSc@mment{End ps@rrowhead}}\fi}
\ctr@ld@f\def\sDcc@ngle{0}
\ctr@ld@f\def\Ps@rrowheadfd#1,#2[#3,#4]{{%
    \PSc@mment{ps@rrowheadfd Length=#1, Angle=#2, [Pt1,Pt2]=[#3,#4]}%
    \setc@ntr@l{2}\figvectPDD-1[#3,#4]\n@rmeucDD{\v@leur}{-1}\v@leur=\ptT@unit@\v@leur%
    \invers@{\v@leur}{\v@leur}\v@leur=#1\v@leur\edef\R@tio{\repdecn@mb{\v@leur}}%
    \Ps@rrowhead\R@tio,#2[#3,#4]\PSc@mment{End ps@rrowheadfd}}}
\ctr@ln@m\psarrowBezier
\ctr@ld@f\def\psarrowBezierDD[#1,#2,#3,#4]{{\ifcurr@ntPS\ifps@cri\s@uvc@ntr@l\et@tpsarrowBezierDD%
    \PSc@mment{psarrowBezierDD Control points=#1,#2,#3,#4}\setc@ntr@l{2}%
    \if@rrowratio\c@larclengthDD\v@leur,10[#1,#2,#3,#4]\else\v@leur=\z@\fi%
    \Ps@rrowB@zDD\v@leur[#1,#2,#3,#4]%
    \PSc@mment{End psarrowBezierDD}\resetc@ntr@l\et@tpsarrowBezierDD\fi\fi}}
\ctr@ld@f\def\psarrowBezierTD[#1,#2,#3,#4]{{\ifcurr@ntPS\ifps@cri\s@uvc@ntr@l\et@tpsarrowBezierTD%
    \PSc@mment{psarrowBezierTD Control points=#1,#2,#3,#4}\resetc@ntr@l{2}%
    \Figptpr@j-7:/#1/\Figptpr@j-8:/#2/\Figptpr@j-9:/#3/\Figptpr@j-10:/#4/%
    \let\c@lprojSP=\relax\ifnum\curr@ntproj<\tw@\psarrowBezierDD[-7,-8,-9,-10]%
    \else\f@gnewpath\PSwrit@cmd{-7}{\c@mmoveto}{\fwf@g}%
    \if@rrowratio\c@larclengthDD\mili@u,10[-7,-8,-9,-10]\else\mili@u=\z@\fi%
    \p@rtent=\NBz@rcs\advance\p@rtent\m@ne\subB@zierTD\p@rtent[#1,#2,#3,#4]%
    \f@gstroke%
    \advance\v@lmin\p@rtent\delt@
    \v@leur=\v@lmin\advance\v@leur0.33333 \delt@\edef\unti@rs{\repdecn@mb{\v@leur}}%
    \v@leur=\v@lmin\advance\v@leur0.66666 \delt@\edef\deti@rs{\repdecn@mb{\v@leur}}%
    \figptcopyDD-8:/-10/\c@lsubBzarc\unti@rs,\deti@rs[#1,#2,#3,#4]%
    \figptcopyDD-8:/-4/\figptcopyDD-9:/-3/\Ps@rrowB@zDD\mili@u[-7,-8,-9,-10]\fi%
    \PSc@mment{End psarrowBezierTD}\resetc@ntr@l\et@tpsarrowBezierTD\fi\fi}}
\ctr@ld@f\def\c@larclengthDD#1,#2[#3,#4,#5,#6]{{\p@rtent=#2\figptcopyDD-5:/#3/%
    \delt@=\p@\divide\delt@\p@rtent\c@rre=\z@\v@leur=\z@\s@mme=\z@%
    \loop\ifnum\s@mme<\p@rtent\advance\s@mme\@ne\advance\v@leur\delt@%
    \edef\T@{\repdecn@mb{\v@leur}}\figptBezierDD-6::\T@[#3,#4,#5,#6]%
    \figvectPDD-1[-5,-6]\n@rmeucDD{\mili@u}{-1}\advance\c@rre\mili@u%
    \figptcopyDD-5:/-6/\repeat\global\result@t=\ptT@unit@\c@rre}#1=\result@t}
\ctr@ld@f\def\Ps@rrowB@zDD#1[#2,#3,#4,#5]{{\pssetfillmode{no}%
    \if@rrowratio\delt@=\@rrowheadratio#1\else\delt@=\@rrowheadlength pt\fi%
    \v@leur=\C@AHANG\delt@\edef\R@dius{\repdecn@mb{\v@leur}}%
    \FigptintercircB@zDD-5::0,\R@dius[#5,#4,#3,#2]%
    \pssetarrowheadlength{\repdecn@mb{\delt@}}\psarrowheadDD[-5,#5]%
    \let\n@rmeuc=\n@rmeucDD\figgetdist\R@dius[#5,-3]%
    \FigptintercircB@zDD-6::0,\R@dius[#5,#4,#3,#2]%
    \figptBezierDD-5::0.33333[#5,#4,#3,#2]\figptBezierDD-3::0.66666[#5,#4,#3,#2]%
    \figptscontrolDD-5[-6,-5,-3,#2]\psBezierDD1[-6,-5,-4,#2]}}
\ctr@ln@m\psarrowcirc
\ctr@ld@f\def\psarrowcircDD#1;#2(#3,#4){{\ifcurr@ntPS\ifps@cri\s@uvc@ntr@l\et@tpsarrowcircDD%
    \PSc@mment{psarrowcircDD Center=#1 ; Radius=#2 (Ang1=#3,Ang2=#4)}%
    \pssetfillmode{no}\Pscirc@rrowhead#1;#2(#3,#4)%
    \setc@ntr@l{2}\figvectPDD -4[#1,-3]\vecunit@{-4}{-4}%
    \Figg@tXY{-4}\arct@n\v@lmin(\v@lX,\v@lY)%
    \v@lmin=\rdT@deg\v@lmin\v@leur=#4pt\advance\v@leur-\v@lmin%
    \maxim@m{\v@leur}{\v@leur}{-\v@leur}%
    \ifdim\v@leur>\DemiPI@deg\relax\ifdim\v@lmin<#4pt\advance\v@lmin\DePI@deg%
    \else\advance\v@lmin-\DePI@deg\fi\fi\edef\ar@ngle{\repdecn@mb{\v@lmin}}%
    \ifdim#3pt<#4pt\psarccirc#1;#2(#3,\ar@ngle)\else\psarccirc#1;#2(\ar@ngle,#3)\fi%
    \PSc@mment{End psarrowcircDD}\resetc@ntr@l\et@tpsarrowcircDD\fi\fi}}
\ctr@ld@f\def\psarrowcircTD#1,#2,#3;#4(#5,#6){{\ifcurr@ntPS\ifps@cri\s@uvc@ntr@l\et@tpsarrowcircTD%
    \PSc@mment{psarrowcircTD Center=#1,P1=#2,P2=#3 ; Radius=#4 (Ang1=#5, Ang2=#6)}%
    \resetc@ntr@l{2}\c@lExtAxes#1,#2,#3(#4)\let\c@lprojSP=\relax%
    \figvectPTD-11[#1,-4]\figvectPTD-12[#1,-5]\c@lNbarcs{#5}{#6}%
    \if@rrowratio\v@lmax=\degT@rd\v@lmax\edef\D@lpha{\repdecn@mb{\v@lmax}}\fi%
    \advance\p@rtent\m@ne\mili@u=\z@%
    \v@leur=#5pt\c@lptellP{#1}{-11}{-12}\Figptpr@j-9:/-3/%
    \f@gnewpath\PSwrit@cmdS{-9}{\c@mmoveto}{\fwf@g}{\X@un}{\Y@un}%
    \edef\C@nt@r{#1}\s@mme=\z@\bcl@rcircTD\f@gstroke%
    \advance\v@leur\delt@\c@lptellP{#1}{-11}{-12}\Figptpr@j-5:/-3/%
    \advance\v@leur\delt@\c@lptellP{#1}{-11}{-12}\Figptpr@j-6:/-3/%
    \advance\v@leur\delt@\c@lptellP{#1}{-11}{-12}\Figptpr@j-10:/-3/%
    \figptscontrolDD-8[-9,-5,-6,-10]%
    \if@rrowratio\c@lcurvradDD0.5[-9,-8,-7,-10]\advance\mili@u\result@t%
    \maxim@m{\mili@u}{\mili@u}{-\mili@u}\mili@u=\ptT@unit@\mili@u%
    \mili@u=\D@lpha\mili@u\advance\p@rtent\@ne\divide\mili@u\p@rtent\fi%
    \Ps@rrowB@zDD\mili@u[-9,-8,-7,-10]%
    \PSc@mment{End psarrowcircTD}\resetc@ntr@l\et@tpsarrowcircTD\fi\fi}}
\ctr@ld@f\def\bcl@rcircTD{\relax%
    \ifnum\s@mme<\p@rtent\advance\s@mme\@ne%
    \advance\v@leur\delt@\c@lptellP{\C@nt@r}{-11}{-12}\Figptpr@j-5:/-3/%
    \advance\v@leur\delt@\c@lptellP{\C@nt@r}{-11}{-12}\Figptpr@j-6:/-3/%
    \advance\v@leur\delt@\c@lptellP{\C@nt@r}{-11}{-12}\Figptpr@j-10:/-3/%
    \figptscontrolDD-8[-9,-5,-6,-10]\BdingB@xfalse%
    \PSwrit@cmdS{-8}{}{\fwf@g}{\X@de}{\Y@de}\PSwrit@cmdS{-7}{}{\fwf@g}{\X@tr}{\Y@tr}%
    \BdingB@xtrue\PSwrit@cmdS{-10}{\c@mcurveto}{\fwf@g}{\X@qu}{\Y@qu}%
    \if@rrowratio\c@lcurvradDD0.5[-9,-8,-7,-10]\advance\mili@u\result@t\fi%
    \B@zierBB@x{1}{\Y@un}(\X@un,\X@de,\X@tr,\X@qu)%
    \B@zierBB@x{2}{\X@un}(\Y@un,\Y@de,\Y@tr,\Y@qu)%
    \edef\X@un{\X@qu}\edef\Y@un{\Y@qu}\figptcopyDD-9:/-10/\bcl@rcircTD\fi}
\ctr@ld@f\def\Pscirc@rrowhead#1;#2(#3,#4){{%
    \PSc@mment{pscirc@rrowhead Center=#1 ; Radius=#2 (Ang1=#3,Ang2=#4)}%
    \v@leur=#2\unit@\edef\s@glen{\repdecn@mb{\v@leur}}\v@lY=\z@\v@lX=\v@leur%
    \resetc@ntr@l{2}\Figv@ctCreg-3(\v@lX,\v@lY)\figpttraDD-5:=#1/1,-3/%
    \figptrotDD-5:=-5/#1,#4/%
    \figvectPDD-3[#1,-5]\Figg@tXY{-3}\v@leur=\v@lX%
    \ifdim#3pt<#4pt\v@lX=\v@lY\v@lY=-\v@leur\else\v@lX=-\v@lY\v@lY=\v@leur\fi%
    \Figv@ctCreg-3(\v@lX,\v@lY)\vecunit@{-3}{-3}%
    \if@rrowratio\v@leur=#4pt\advance\v@leur-#3pt\maxim@m{\mili@u}{-\v@leur}{\v@leur}%
    \mili@u=\degT@rd\mili@u\v@leur=\s@glen\mili@u\edef\s@glen{\repdecn@mb{\v@leur}}%
    \mili@u=#2\mili@u\mili@u=\@rrowheadratio\mili@u\else\mili@u=\@rrowheadlength pt\fi%
    \figpttraDD-6:=-5/\s@glen,-3/\v@leur=#2pt\v@leur=2\v@leur%
    \invers@{\v@leur}{\v@leur}\c@rre=\repdecn@mb{\v@leur}\mili@u
    \mili@u=\c@rre\mili@u=\repdecn@mb{\c@rre}\mili@u%
    \v@leur=\p@\advance\v@leur-\mili@u
    \invers@{\mili@u}{2\v@leur}\delt@=\c@rre\delt@=\repdecn@mb{\mili@u}\delt@%
    \xdef\sDcc@ngle{\repdecn@mb{\delt@}}
    \sqrt@{\mili@u}{\v@leur}\arct@n\v@leur(\mili@u,\c@rre)%
    \v@leur=\rdT@deg\v@leur
    \ifdim#3pt<#4pt\v@leur=-\v@leur\fi%
    \if@rrowhout\v@leur=-\v@leur\fi\edef\cor@ngle{\repdecn@mb{\v@leur}}%
    \figptrotDD-6:=-6/-5,\cor@ngle/\psarrowheadDD[-6,-5]%
    \PSc@mment{End pscirc@rrowhead}}}
\ctr@ln@m\psarrowcircP
\ctr@ld@f\def\psarrowcircPDD#1;#2[#3,#4]{{\ifcurr@ntPS\ifps@cri%
    \PSc@mment{psarrowcircPDD Center=#1; Radius=#2, [P1=#3,P2=#4]}%
    \s@uvc@ntr@l\et@tpsarrowcircPDD\Ps@ngleparam#1;#2[#3,#4]%
    \ifdim\v@leur>\z@\ifdim\v@lmin>\v@lmax\advance\v@lmax\DePI@deg\fi%
    \else\ifdim\v@lmin<\v@lmax\advance\v@lmin\DePI@deg\fi\fi%
    \edef\@ngdeb{\repdecn@mb{\v@lmin}}\edef\@ngfin{\repdecn@mb{\v@lmax}}%
    \psarrowcirc#1;\r@dius(\@ngdeb,\@ngfin)%
    \PSc@mment{End psarrowcircPDD}\resetc@ntr@l\et@tpsarrowcircPDD\fi\fi}}
\ctr@ld@f\def\psarrowcircPTD#1;#2[#3,#4,#5]{{\ifcurr@ntPS\ifps@cri\s@uvc@ntr@l\et@tpsarrowcircPTD%
    \PSc@mment{psarrowcircPTD Center=#1; Radius=#2, [P1=#3,P2=#4,P3=#5]}%
    \figgetangleTD\@ngfin[#1,#3,#4,#5]\v@leur=#2pt%
    \maxim@m{\mili@u}{-\v@leur}{\v@leur}\edef\r@dius{\repdecn@mb{\mili@u}}%
    \ifdim\v@leur<\z@\v@lmax=\@ngfin pt\advance\v@lmax-\DePI@deg%
    \edef\@ngfin{\repdecn@mb{\v@lmax}}\fi\psarrowcircTD#1,#3,#5;\r@dius(0,\@ngfin)%
    \PSc@mment{End psarrowcircPTD}\resetc@ntr@l\et@tpsarrowcircPTD\fi\fi}}
\ctr@ld@f\def\psaxes#1(#2){{\ifcurr@ntPS\ifps@cri\s@uvc@ntr@l\et@tpsaxes%
    \PSc@mment{psaxes Origin=#1 Range=(#2)}\an@lys@xes#2,:\resetc@ntr@l{2}%
    \ifx\t@xt@\empty\ifTr@isDim\ps@xes#1(0,#2,0,#2,0,#2)\else\ps@xes#1(0,#2,0,#2)\fi%
    \else\ps@xes#1(#2)\fi\PSc@mment{End psaxes}\resetc@ntr@l\et@tpsaxes\fi\fi}}
\ctr@ld@f\def\an@lys@xes#1,#2:{\def\t@xt@{#2}}
\ctr@ln@m\ps@xes
\ctr@ld@f\def\ps@xesDD#1(#2,#3,#4,#5){%
    \figpttraC-5:=#1/#2,0/\figpttraC-6:=#1/#3,0/\psarrowDD[-5,-6]%
    \figpttraC-5:=#1/0,#4/\figpttraC-6:=#1/0,#5/\psarrowDD[-5,-6]}
\ctr@ld@f\def\ps@xesTD#1(#2,#3,#4,#5,#6,#7){%
    \figpttraC-7:=#1/#2,0,0/\figpttraC-8:=#1/#3,0,0/\psarrowTD[-7,-8]%
    \figpttraC-7:=#1/0,#4,0/\figpttraC-8:=#1/0,#5,0/\psarrowTD[-7,-8]%
    \figpttraC-7:=#1/0,0,#6/\figpttraC-8:=#1/0,0,#7/\psarrowTD[-7,-8]}
\ctr@ln@m\newGr@FN
\ctr@ld@f\def\newGr@FNPDF#1{\s@mme=\Gr@FNb\advance\s@mme\@ne\xdef\Gr@FNb{\number\s@mme}}
\ctr@ld@f\def\newGr@FNDVI#1{\newGr@FNPDF{}\xdef#1{\jobname GI\Gr@FNb.anx}}
\ctr@ld@f\def\psbeginfig#1{\newGr@FN\DefGIfilen@me\gdef\@utoFN{0}%
    \def\t@xt@{#1}\relax\ifx\t@xt@\empty\psupdatem@detrue%
    \gdef\@utoFN{1}\Psb@ginfig\DefGIfilen@me\else\expandafter\Psb@ginfigNu@#1 :\fi}
\ctr@ld@f\def\Psb@ginfigNu@#1 #2:{\def\t@xt@{#1}\relax\ifx\t@xt@\empty\def\t@xt@{#2}%
    \ifx\t@xt@\empty\psupdatem@detrue\gdef\@utoFN{1}\Psb@ginfig\DefGIfilen@me%
    \else\Psb@ginfigNu@#2:\fi\else\Psb@ginfig{#1}\fi}
\ctr@ln@m\PSfilen@me \ctr@ln@m\auxfilen@me
\ctr@ld@f\def\Psb@ginfig#1{\ifcurr@ntPS\else%
    \edef\PSfilen@me{#1}\edef\auxfilen@me{\jobname.anx}%
    \ifpsupdatem@de\ps@critrue\else\openin\frf@g=\PSfilen@me\relax%
    \ifeof\frf@g\ps@critrue\else\ps@crifalse\fi\closein\frf@g\fi%
    \curr@ntPStrue\c@ldefproj\expandafter\setupd@te\defaultupdate:%
    \ifps@cri\initb@undb@x%
    \immediate\openout\fwf@g=\auxfilen@me\initpss@ttings\fi%
    \fi}
\ctr@ld@f\def\Gr@FNb{0}
\ctr@ld@f\def\figforTeXFileno{\Gr@FNb}
\ctr@ld@f\def\figforTeXFigno{0 }
\ctr@ld@f\def\figforTeXnextFigno{1 }
\ctr@ld@f\edef\DefGIfilen@me{\jobname GI.anx}
\ctr@ld@f\def\initpss@ttings{\psreset{arrowhead,curve,first,flowchart,mesh,second,third}%
    \Use@llipsefalse}
\ctr@ld@f\def\B@zierBB@x#1#2(#3,#4,#5,#6){{\c@rre=\t@n\epsil@n
    \v@lmax=#4\advance\v@lmax-#5\v@lmax=\thr@@\v@lmax\advance\v@lmax#6\advance\v@lmax-#3%
    \mili@u=#4\mili@u=-\tw@\mili@u\advance\mili@u#3\advance\mili@u#5%
    \v@lmin=#4\advance\v@lmin-#3\maxim@m{\v@leur}{-\v@lmax}{\v@lmax}%
    \maxim@m{\delt@}{-\mili@u}{\mili@u}\maxim@m{\v@leur}{\v@leur}{\delt@}%
    \maxim@m{\delt@}{-\v@lmin}{\v@lmin}\maxim@m{\v@leur}{\v@leur}{\delt@}%
    \ifdim\v@leur>\c@rre\invers@{\v@leur}{\v@leur}\edef\Uns@rM@x{\repdecn@mb{\v@leur}}%
    \v@lmax=\Uns@rM@x\v@lmax\mili@u=\Uns@rM@x\mili@u\v@lmin=\Uns@rM@x\v@lmin%
    \maxim@m{\v@leur}{-\v@lmax}{\v@lmax}\ifdim\v@leur<\c@rre%
    \maxim@m{\v@leur}{-\mili@u}{\mili@u}\ifdim\v@leur<\c@rre\else%
    \invers@{\mili@u}{\mili@u}\v@leur=-0.5\v@lmin%
    \v@leur=\repdecn@mb{\mili@u}\v@leur\m@jBBB@x{\v@leur}{#1}{#2}(#3,#4,#5,#6)\fi%
    \else\delt@=\repdecn@mb{\mili@u}\mili@u\v@leur=\repdecn@mb{\v@lmax}\v@lmin%
    \advance\delt@-\v@leur\ifdim\delt@<\z@\else\invers@{\v@lmax}{\v@lmax}%
    \edef\Uns@rAp{\repdecn@mb{\v@lmax}}\sqrt@{\delt@}{\delt@}%
    \v@leur=-\mili@u\advance\v@leur\delt@\v@leur=\Uns@rAp\v@leur%
    \m@jBBB@x{\v@leur}{#1}{#2}(#3,#4,#5,#6)%
    \v@leur=-\mili@u\advance\v@leur-\delt@\v@leur=\Uns@rAp\v@leur%
    \m@jBBB@x{\v@leur}{#1}{#2}(#3,#4,#5,#6)\fi\fi\fi}}
\ctr@ld@f\def\m@jBBB@x#1#2#3(#4,#5,#6,#7){{\relax\ifdim#1>\z@\ifdim#1<\p@%
    \edef\T@{\repdecn@mb{#1}}\v@lX=\p@\advance\v@lX-#1\edef\UNmT@{\repdecn@mb{\v@lX}}%
    \v@lX=#4\v@lY=#5\v@lZ=#6\v@lXa=#7\v@lX=\UNmT@\v@lX\advance\v@lX\T@\v@lY%
    \v@lY=\UNmT@\v@lY\advance\v@lY\T@\v@lZ\v@lZ=\UNmT@\v@lZ\advance\v@lZ\T@\v@lXa%
    \v@lX=\UNmT@\v@lX\advance\v@lX\T@\v@lY\v@lY=\UNmT@\v@lY\advance\v@lY\T@\v@lZ%
    \v@lX=\UNmT@\v@lX\advance\v@lX\T@\v@lY%
    \ifcase#2\or\v@lY=#3\or\v@lY=\v@lX\v@lX=#3\fi\b@undb@x{\v@lX}{\v@lY}\fi\fi}}
\ctr@ld@f\def\PsB@zier#1[#2]{{\f@gnewpath%
    \s@mme=\z@\def\list@num{#2,0}\extrairelepremi@r\p@int\de\list@num%
    \PSwrit@cmdS{\p@int}{\c@mmoveto}{\fwf@g}{\X@un}{\Y@un}\p@rtent=#1\bclB@zier}}
\ctr@ld@f\def\bclB@zier{\relax%
    \ifnum\s@mme<\p@rtent\advance\s@mme\@ne\BdingB@xfalse%
    \extrairelepremi@r\p@int\de\list@num\PSwrit@cmdS{\p@int}{}{\fwf@g}{\X@de}{\Y@de}%
    \extrairelepremi@r\p@int\de\list@num\PSwrit@cmdS{\p@int}{}{\fwf@g}{\X@tr}{\Y@tr}%
    \BdingB@xtrue%
    \extrairelepremi@r\p@int\de\list@num\PSwrit@cmdS{\p@int}{\c@mcurveto}{\fwf@g}{\X@qu}{\Y@qu}%
    \B@zierBB@x{1}{\Y@un}(\X@un,\X@de,\X@tr,\X@qu)%
    \B@zierBB@x{2}{\X@un}(\Y@un,\Y@de,\Y@tr,\Y@qu)%
    \edef\X@un{\X@qu}\edef\Y@un{\Y@qu}\bclB@zier\fi}
\ctr@ln@m\psBezier
\ctr@ld@f\def\psBezierDD#1[#2]{\ifcurr@ntPS\ifps@cri%
    \PSc@mment{psBezierDD N arcs=#1, Control points=#2}%
    \iffillm@de\PsB@zier#1[#2]%
    \f@gfill%
    \else\PsB@zier#1[#2]\f@gstroke\fi%
    \PSc@mment{End psBezierDD}\fi\fi}
\ctr@ln@m\et@tpsBezierTD
\ctr@ld@f\def\psBezierTD#1[#2]{\ifcurr@ntPS\ifps@cri\s@uvc@ntr@l\et@tpsBezierTD%
    \PSc@mment{psBezierTD N arcs=#1, Control points=#2}%
    \iffillm@de\PsB@zierTD#1[#2]%
    \f@gfill%
    \else\PsB@zierTD#1[#2]\f@gstroke\fi%
    \PSc@mment{End psBezierTD}\resetc@ntr@l\et@tpsBezierTD\fi\fi}
\ctr@ld@f\def\PsB@zierTD#1[#2]{\ifnum\curr@ntproj<\tw@\PsB@zier#1[#2]\else\PsB@zier@TD#1[#2]\fi}
\ctr@ld@f\def\PsB@zier@TD#1[#2]{{\f@gnewpath%
    \s@mme=\z@\def\list@num{#2,0}\extrairelepremi@r\p@int\de\list@num%
    \let\c@lprojSP=\relax\setc@ntr@l{2}\Figptpr@j-7:/\p@int/%
    \PSwrit@cmd{-7}{\c@mmoveto}{\fwf@g}%
    \loop\ifnum\s@mme<#1\advance\s@mme\@ne\extrairelepremi@r\p@intun\de\list@num%
    \extrairelepremi@r\p@intde\de\list@num\extrairelepremi@r\p@inttr\de\list@num%
    \subB@zierTD\NBz@rcs[\p@int,\p@intun,\p@intde,\p@inttr]\edef\p@int{\p@inttr}\repeat}}
\ctr@ld@f\def\subB@zierTD#1[#2,#3,#4,#5]{\delt@=\p@\divide\delt@\NBz@rcs\v@lmin=\z@%
    {\Figg@tXY{-7}\edef\X@un{\the\v@lX}\edef\Y@un{\the\v@lY}%
    \s@mme=\z@\loop\ifnum\s@mme<#1\advance\s@mme\@ne%
    \v@leur=\v@lmin\advance\v@leur0.33333 \delt@\edef\unti@rs{\repdecn@mb{\v@leur}}%
    \v@leur=\v@lmin\advance\v@leur0.66666 \delt@\edef\deti@rs{\repdecn@mb{\v@leur}}%
    \advance\v@lmin\delt@\edef\trti@rs{\repdecn@mb{\v@lmin}}%
    \figptBezierTD-8::\trti@rs[#2,#3,#4,#5]\Figptpr@j-8:/-8/%
    \c@lsubBzarc\unti@rs,\deti@rs[#2,#3,#4,#5]\BdingB@xfalse%
    \PSwrit@cmdS{-4}{}{\fwf@g}{\X@de}{\Y@de}\PSwrit@cmdS{-3}{}{\fwf@g}{\X@tr}{\Y@tr}%
    \BdingB@xtrue\PSwrit@cmdS{-8}{\c@mcurveto}{\fwf@g}{\X@qu}{\Y@qu}%
    \B@zierBB@x{1}{\Y@un}(\X@un,\X@de,\X@tr,\X@qu)%
    \B@zierBB@x{2}{\X@un}(\Y@un,\Y@de,\Y@tr,\Y@qu)%
    \edef\X@un{\X@qu}\edef\Y@un{\Y@qu}\figptcopyDD-7:/-8/\repeat}}
\ctr@ld@f\def\NBz@rcs{2}
\ctr@ld@f\def\c@lsubBzarc#1,#2[#3,#4,#5,#6]{\figptBezierTD-5::#1[#3,#4,#5,#6]%
    \figptBezierTD-6::#2[#3,#4,#5,#6]\Figptpr@j-4:/-5/\Figptpr@j-5:/-6/%
    \figptscontrolDD-4[-7,-4,-5,-8]}
\ctr@ln@m\pscirc
\ctr@ld@f\def\pscircDD#1(#2){\ifcurr@ntPS\ifps@cri\PSc@mment{pscircDD Center=#1 (Radius=#2)}%
    \psarccircDD#1;#2(0,360)\PSc@mment{End pscircDD}\fi\fi}
\ctr@ld@f\def\pscircTD#1,#2,#3(#4){\ifcurr@ntPS\ifps@cri%
    \PSc@mment{pscircTD Center=#1,P1=#2,P2=#3 (Radius=#4)}%
    \psarccircTD#1,#2,#3;#4(0,360)\PSc@mment{End pscircTD}\fi\fi}
\ctr@ln@m\p@urcent
{\catcode`\%=12\gdef\p@urcent{
\ctr@ld@f\def\PSc@mment#1{\ifpsdebugmode\immediate\write\fwf@g{\p@urcent\space#1}\fi}
\ctr@ln@m\acc@louv \ctr@ln@m\acc@lfer
{\catcode`\[=1\catcode`\{=12\gdef\acc@louv[{}}
{\catcode`\]=2\catcode`\}=12\gdef\acc@lfer{}]]
\ctr@ld@f\def\PSdict@{\ifUse@llipse%
    \immediate\write\fwf@g{/ellipsedict 9 dict def ellipsedict /mtrx matrix put}%
    \immediate\write\fwf@g{/ellipse \acc@louv ellipsedict begin}%
    \immediate\write\fwf@g{ /endangle exch def /startangle exch def}%
    \immediate\write\fwf@g{ /yrad exch def /xrad exch def}%
    \immediate\write\fwf@g{ /rotangle exch def /y exch def /x exch def}%
    \immediate\write\fwf@g{ /savematrix mtrx currentmatrix def}%
    \immediate\write\fwf@g{ x y translate rotangle rotate xrad yrad scale}%
    \immediate\write\fwf@g{ 0 0 1 startangle endangle arc}%
    \immediate\write\fwf@g{ savematrix setmatrix end\acc@lfer def}%
    \fi\PShe@der{EndProlog}}
\ctr@ld@f\def\Pssetc@rve#1=#2|{\keln@mun#1|%
    \def\n@mref{r}\ifx\l@debut\n@mref\pssetroundness{#2}\else
    \immediate\write16{*** Unknown attribute: \BS@ psset curve(..., #1=...)}%
    \fi}
\ctr@ln@m\curv@roundness
\ctr@ld@f\def\pssetroundness#1{\edef\curv@roundness{#1}}
\ctr@ld@f\def\defaultroundness{0.2} 
\ctr@ln@m\pscurve
\ctr@ld@f\def\pscurveDD[#1]{{\ifcurr@ntPS\ifps@cri\PSc@mment{pscurveDD Points=#1}%
    \s@uvc@ntr@l\et@tpscurveDD%
    \iffillm@de\Psc@rveDD\curv@roundness[#1]%
    \f@gfill%
    \else\Psc@rveDD\curv@roundness[#1]\f@gstroke\fi%
    \PSc@mment{End pscurveDD}\resetc@ntr@l\et@tpscurveDD\fi\fi}}
\ctr@ld@f\def\pscurveTD[#1]{{\ifcurr@ntPS\ifps@cri%
    \PSc@mment{pscurveTD Points=#1}\s@uvc@ntr@l\et@tpscurveTD\let\c@lprojSP=\relax%
    \iffillm@de\Psc@rveTD\curv@roundness[#1]%
    \f@gfill%
    \else\Psc@rveTD\curv@roundness[#1]\f@gstroke\fi%
    \PSc@mment{End pscurveTD}\resetc@ntr@l\et@tpscurveTD\fi\fi}}
\ctr@ld@f\def\Psc@rveDD#1[#2]{%
    \def\list@num{#2}\extrairelepremi@r\Ak@\de\list@num%
    \extrairelepremi@r\Ai@\de\list@num\extrairelepremi@r\Aj@\de\list@num%
    \f@gnewpath\PSwrit@cmdS{\Ai@}{\c@mmoveto}{\fwf@g}{\X@un}{\Y@un}%
    \setc@ntr@l{2}\figvectPDD -1[\Ak@,\Aj@]%
    \@ecfor\Ak@:=\list@num\do{\figpttraDD-2:=\Ai@/#1,-1/\BdingB@xfalse%
       \PSwrit@cmdS{-2}{}{\fwf@g}{\X@de}{\Y@de}%
       \figvectPDD -1[\Ai@,\Ak@]\figpttraDD-2:=\Aj@/-#1,-1/%
       \PSwrit@cmdS{-2}{}{\fwf@g}{\X@tr}{\Y@tr}\BdingB@xtrue%
       \PSwrit@cmdS{\Aj@}{\c@mcurveto}{\fwf@g}{\X@qu}{\Y@qu}%
       \B@zierBB@x{1}{\Y@un}(\X@un,\X@de,\X@tr,\X@qu)%
       \B@zierBB@x{2}{\X@un}(\Y@un,\Y@de,\Y@tr,\Y@qu)%
       \edef\X@un{\X@qu}\edef\Y@un{\Y@qu}\edef\Ai@{\Aj@}\edef\Aj@{\Ak@}}}
\ctr@ld@f\def\Psc@rveTD#1[#2]{\ifnum\curr@ntproj<\tw@\Psc@rvePPTD#1[#2]\else\Psc@rveCPTD#1[#2]\fi}
\ctr@ld@f\def\Psc@rvePPTD#1[#2]{\setc@ntr@l{2}%
    \def\list@num{#2}\extrairelepremi@r\Ak@\de\list@num\Figptpr@j-5:/\Ak@/%
    \extrairelepremi@r\Ai@\de\list@num\Figptpr@j-3:/\Ai@/%
    \extrairelepremi@r\Aj@\de\list@num\Figptpr@j-4:/\Aj@/%
    \f@gnewpath\PSwrit@cmdS{-3}{\c@mmoveto}{\fwf@g}{\X@un}{\Y@un}%
    \figvectPDD -1[-5,-4]%
    \@ecfor\Ak@:=\list@num\do{\Figptpr@j-5:/\Ak@/\figpttraDD-2:=-3/#1,-1/%
       \BdingB@xfalse\PSwrit@cmdS{-2}{}{\fwf@g}{\X@de}{\Y@de}%
       \figvectPDD -1[-3,-5]\figpttraDD-2:=-4/-#1,-1/%
       \PSwrit@cmdS{-2}{}{\fwf@g}{\X@tr}{\Y@tr}\BdingB@xtrue%
       \PSwrit@cmdS{-4}{\c@mcurveto}{\fwf@g}{\X@qu}{\Y@qu}%
       \B@zierBB@x{1}{\Y@un}(\X@un,\X@de,\X@tr,\X@qu)%
       \B@zierBB@x{2}{\X@un}(\Y@un,\Y@de,\Y@tr,\Y@qu)%
       \edef\X@un{\X@qu}\edef\Y@un{\Y@qu}\figptcopyDD-3:/-4/\figptcopyDD-4:/-5/}}
\ctr@ld@f\def\Psc@rveCPTD#1[#2]{\setc@ntr@l{2}%
    \def\list@num{#2}\extrairelepremi@r\Ak@\de\list@num%
    \extrairelepremi@r\Ai@\de\list@num\extrairelepremi@r\Aj@\de\list@num%
    \Figptpr@j-7:/\Ai@/%
    \f@gnewpath\PSwrit@cmd{-7}{\c@mmoveto}{\fwf@g}%
    \figvectPTD -9[\Ak@,\Aj@]%
    \@ecfor\Ak@:=\list@num\do{\figpttraTD-10:=\Ai@/#1,-9/%
       \figvectPTD -9[\Ai@,\Ak@]\figpttraTD-11:=\Aj@/-#1,-9/%
       \subB@zierTD\NBz@rcs[\Ai@,-10,-11,\Aj@]\edef\Ai@{\Aj@}\edef\Aj@{\Ak@}}}
\ctr@ld@f\def\psendfig{\ifcurr@ntPS\ifps@cri\immediate\closeout\fwf@g%
    \immediate\openout\fwf@g=\PSfilen@me\relax%
    \ifPDFm@ke\PSBdingB@x\else%
    \immediate\write\fwf@g{\p@urcent\string!PS-Adobe-2.0 EPSF-2.0}%
    \PShe@der{Creator\string: TeX (fig4tex.tex)}%
    \PShe@der{Title\string: \PSfilen@me}%
    \PShe@der{CreationDate\string: \the\day/\the\month/\the\year}%
    \PSBdingB@x%
    \PShe@der{EndComments}\PSdict@\fi%
    \immediate\write\fwf@g{\c@mgsave}%
    \openin\frf@g=\auxfilen@me\c@pypsfile\fwf@g\frf@g\closein\frf@g%
    \immediate\write\fwf@g{\c@mgrestore}%
    \PSc@mment{End of file.}\immediate\closeout\fwf@g%
    \immediate\openout\fwf@g=\auxfilen@me\immediate\closeout\fwf@g%
    \immediate\write16{File \PSfilen@me\space created.}\fi\fi\curr@ntPSfalse\ps@critrue}
\ctr@ld@f\def\PShe@der#1{\immediate\write\fwf@g{\p@urcent\p@urcent#1}}
\ctr@ld@f\def\PSBdingB@x{{\v@lX=\ptT@ptps\c@@rdXmin\v@lY=\ptT@ptps\c@@rdYmin%
     \v@lXa=\ptT@ptps\c@@rdXmax\v@lYa=\ptT@ptps\c@@rdYmax%
     \PShe@der{BoundingBox\string: \repdecn@mb{\v@lX}\space\repdecn@mb{\v@lY}%
     \space\repdecn@mb{\v@lXa}\space\repdecn@mb{\v@lYa}}}}
\ctr@ld@f\def\psfcconnect[#1]{{\ifcurr@ntPS\ifps@cri\PSc@mment{psfcconnect Points=#1}%
    \pssetfillmode{no}\s@uvc@ntr@l\et@tpsfcconnect\resetc@ntr@l{2}%
    \fcc@nnect@[#1]\resetc@ntr@l\et@tpsfcconnect\PSc@mment{End psfcconnect}\fi\fi}}
\ctr@ld@f\def\fcc@nnect@[#1]{\let\N@rm=\n@rmeucDD\def\list@num{#1}%
    \extrairelepremi@r\Ai@\de\list@num\edef\pr@m{\Ai@}\v@leur=\z@\p@rtent=\@ne\c@llgtot%
    \ifcase\fclin@typ@\edef\list@num{[\pr@m,#1,\Ai@}\expandafter\pscurve\list@num]%
    \else\ifdim\fclin@r@d\p@>\z@\Pslin@conge[#1]\else\psline[#1]\fi\fi%
    \v@leur=\@rrowp@s\v@leur\edef\list@num{#1,\Ai@,0}%
    \extrairelepremi@r\Ai@\de\list@num\mili@u=\epsil@n\c@llgpart%
    \advance\mili@u-\epsil@n\advance\mili@u-\delt@\advance\v@leur-\mili@u%
    \ifcase\fclin@typ@\invers@\mili@u\delt@%
    \ifnum\@rrowr@fpt>\z@\advance\delt@-\v@leur\v@leur=\delt@\fi%
    \v@leur=\repdecn@mb\v@leur\mili@u\edef\v@lt{\repdecn@mb\v@leur}%
    \extrairelepremi@r\Ak@\de\list@num%
    \figvectPDD-1[\pr@m,\Aj@]\figpttraDD-6:=\Ai@/\curv@roundness,-1/%
    \figvectPDD-1[\Ak@,\Ai@]\figpttraDD-7:=\Aj@/\curv@roundness,-1/%
    \delt@=\@rrowheadlength\p@\delt@=\C@AHANG\delt@\edef\R@dius{\repdecn@mb{\delt@}}%
    \ifcase\@rrowr@fpt%
    \FigptintercircB@zDD-8::\v@lt,\R@dius[\Ai@,-6,-7,\Aj@]\psarrowheadDD[-5,-8]\else%
    \FigptintercircB@zDD-8::\v@lt,\R@dius[\Aj@,-7,-6,\Ai@]\psarrowheadDD[-8,-5]\fi%
    \else\advance\delt@-\v@leur%
    \p@rtentiere{\p@rtent}{\delt@}\edef\C@efun{\the\p@rtent}%
    \p@rtentiere{\p@rtent}{\v@leur}\edef\C@efde{\the\p@rtent}%
    \figptbaryDD-5:[\Ai@,\Aj@;\C@efun,\C@efde]\ifcase\@rrowr@fpt%
    \delt@=\@rrowheadlength\unit@\delt@=\C@AHANG\delt@\edef\t@ille{\repdecn@mb{\delt@}}%
    \figvectPDD-2[\Ai@,\Aj@]\vecunit@{-2}{-2}\figpttraDD-5:=-5/\t@ille,-2/\fi%
    \psarrowheadDD[\Ai@,-5]\fi}
\ctr@ld@f\def\c@llgtot{\@ecfor\Aj@:=\list@num\do{\figvectP-1[\Ai@,\Aj@]\N@rm\delt@{-1}%
    \advance\v@leur\delt@\advance\p@rtent\@ne\edef\Ai@{\Aj@}}}
\ctr@ld@f\def\c@llgpart{\extrairelepremi@r\Aj@\de\list@num\figvectP-1[\Ai@,\Aj@]\N@rm\delt@{-1}%
    \advance\mili@u\delt@\ifdim\mili@u<\v@leur\edef\pr@m{\Ai@}\edef\Ai@{\Aj@}\c@llgpart\fi}
\ctr@ld@f\def\Pslin@conge[#1]{\ifnum\p@rtent>\tw@{\def\list@num{#1}%
    \extrairelepremi@r\Ai@\de\list@num\extrairelepremi@r\Aj@\de\list@num%
    \figptcopy-6:/\Ai@/\figvectP-3[\Ai@,\Aj@]\vecunit@{-3}{-3}\v@lmax=\result@t%
    \@ecfor\Ak@:=\list@num\do{\figvectP-4[\Aj@,\Ak@]\vecunit@{-4}{-4}%
    \minim@m\v@lmin\v@lmax\result@t\v@lmax=\result@t%
    \det@rm\delt@[-3,-4]\maxim@m\mili@u{\delt@}{-\delt@}\ifdim\mili@u>\Cepsil@n%
    \ifdim\delt@>\z@\figgetangleDD\Angl@[\Aj@,\Ak@,\Ai@]\else%
    \figgetangleDD\Angl@[\Aj@,\Ai@,\Ak@]\fi%
    \v@leur=\PI@deg\advance\v@leur-\Angl@\p@\divide\v@leur\tw@%
    \edef\Angl@{\repdecn@mb\v@leur}\c@ssin{\C@}{\S@}{\Angl@}\v@leur=\fclin@r@d\unit@%
    \v@leur=\S@\v@leur\mili@u=\C@\p@\invers@\mili@u\mili@u%
    \v@leur=\repdecn@mb{\mili@u}\v@leur%
    \minim@m\v@leur\v@leur\v@lmin\edef\t@ille{\repdecn@mb{\v@leur}}%
    \figpttra-5:=\Aj@/-\t@ille,-3/\psline[-6,-5]\figpttra-6:=\Aj@/\t@ille,-4/%
    \figvectNVDD-3[-3]\figvectNVDD-8[-4]\inters@cDD-7:[-5,-3;-6,-8]%
    \ifdim\delt@>\z@\psarccircP-7;\fclin@r@d[-5,-6]\else\psarccircP-7;\fclin@r@d[-6,-5]\fi%
    \else\psline[-6,\Aj@]\figptcopy-6:/\Aj@/\fi
    \edef\Ai@{\Aj@}\edef\Aj@{\Ak@}\figptcopy-3:/-4/}\psline[-6,\Aj@]}\else\psline[#1]\fi}
\ctr@ld@f\def\psfcnode[#1]#2{{\ifcurr@ntPS\ifps@cri\PSc@mment{psfcnode Points=#1}%
    \s@uvc@ntr@l\et@tpsfcnode\resetc@ntr@l{2}%
    \def\t@xt@{#2}\ifx\t@xt@\empty\def\g@tt@xt{\setbox\Gb@x=\hbox{\Figg@tT{\p@int}}}%
    \else\def\g@tt@xt{\setbox\Gb@x=\hbox{#2}}\fi%
    \v@lmin=\h@rdfcXp@dd\advance\v@lmin\Xp@dd\unit@\multiply\v@lmin\tw@%
    \v@lmax=\h@rdfcYp@dd\advance\v@lmax\Yp@dd\unit@\multiply\v@lmax\tw@%
    \Figv@ctCreg-8(\unit@,-\unit@)\def\list@num{#1}%
    \delt@=\curr@ntwidth bp\divide\delt@\tw@%
    \fcn@de\PSc@mment{End psfcnode}\resetc@ntr@l\et@tpsfcnode\fi\fi}}
\ctr@ld@f\def\d@butn@de{\g@tt@xt\v@lX=\wd\Gb@x%
    \v@lY=\ht\Gb@x\advance\v@lY\dp\Gb@x\advance\v@lX\v@lmin\advance\v@lY\v@lmax}
\ctr@ld@f\def\fcn@deE{%
    \@ecfor\p@int:=\list@num\do{\d@butn@de\v@lX=\unssqrttw@\v@lX\v@lY=\unssqrttw@\v@lY%
    \ifdim\thickn@ss\p@>\z@
    \v@lXa=\v@lX\advance\v@lXa\delt@\v@lXa=\ptT@unit@\v@lXa\edef\XR@d{\repdecn@mb\v@lXa}%
    \v@lYa=\v@lY\advance\v@lYa\delt@\v@lYa=\ptT@unit@\v@lYa\edef\YR@d{\repdecn@mb\v@lYa}%
    \arct@n\v@leur(\v@lXa,\v@lYa)\v@leur=\rdT@deg\v@leur\edef\@nglde{\repdecn@mb\v@leur}%
    {\c@lptellDD-2::\p@int;\XR@d,\YR@d(\@nglde)}
    \advance\v@leur-\PI@deg\edef\@nglun{\repdecn@mb\v@leur}%
    {\c@lptellDD-3::\p@int;\XR@d,\YR@d(\@nglun)}%
    \figptstra-6=-3,-2,\p@int/\thickn@ss,-8/\pssetfillmode{yes}\us@secondC@lor%
    \psline[-2,-3,-6,-5]\psarcell-4;\XR@d,\YR@d(\@nglun,\@nglde,0)\fi
    \v@lX=\ptT@unit@\v@lX\v@lY=\ptT@unit@\v@lY%
    \edef\XR@d{\repdecn@mb\v@lX}\edef\YR@d{\repdecn@mb\v@lY}%
    \pssetfillmode{yes}\us@thirdC@lor\psarcell\p@int;\XR@d,\YR@d(0,360,0)%
    \pssetfillmode{no}\us@primarC@lor\psarcell\p@int;\XR@d,\YR@d(0,360,0)}}
\ctr@ld@f\def\fcn@deL{\delt@=\ptT@unit@\delt@\edef\t@ille{\repdecn@mb\delt@}%
    \@ecfor\p@int:=\list@num\do{\Figg@tXYa{\p@int}\d@butn@de%
    \ifdim\v@lX>\v@lY\itis@Ktrue\else\itis@Kfalse\fi%
    \advance\v@lXa-\v@lX\Figp@intreg-1:(\v@lXa,\v@lYa)%
    \advance\v@lXa\v@lX\advance\v@lYa-\v@lY\Figp@intreg-2:(\v@lXa,\v@lYa)%
    \advance\v@lXa\v@lX\advance\v@lYa\v@lY\Figp@intreg-3:(\v@lXa,\v@lYa)%
    \advance\v@lXa-\v@lX\advance\v@lYa\v@lY\Figp@intreg-4:(\v@lXa,\v@lYa)%
    \ifdim\thickn@ss\p@>\z@\Figg@tXYa{\p@int}\pssetfillmode{yes}\us@secondC@lor
    \c@lpt@xt{-1}{-4}\c@lpt@xt@\v@lXa\v@lYa\v@lX\v@lY\c@rre\delt@%
    \Figp@intregDD-9:(\v@lZ,\v@lYa)\Figp@intregDD-11:(\v@lZa,\v@lYa)%
    \c@lpt@xt{-4}{-3}\c@lpt@xt@\v@lYa\v@lXa\v@lY\v@lX\delt@\c@rre%
    \Figp@intregDD-12:(\v@lXa,\v@lZ)\Figp@intregDD-10:(\v@lXa,\v@lZa)%
    \ifitis@K\figptstra-7=-9,-10,-11/\thickn@ss,-8/\psline[-9,-11,-5,-6,-7]\else%
    \figptstra-7=-10,-11,-12/\thickn@ss,-8/\psline[-10,-12,-5,-6,-7]\fi\fi
    \pssetfillmode{yes}\us@thirdC@lor\psline[-1,-2,-3,-4]%
    \pssetfillmode{no}\us@primarC@lor\psline[-1,-2,-3,-4,-1]}}
\ctr@ld@f\def\c@lpt@xt#1#2{\figvectN-7[#1,#2]\vecunit@{-7}{-7}\figpttra-5:=#1/\t@ille,-7/%
    \figvectP-7[#1,#2]\Figg@tXY{-7}\c@rre=\v@lX\delt@=\v@lY\Figg@tXY{-5}}
\ctr@ld@f\def\c@lpt@xt@#1#2#3#4#5#6{\v@lZ=#6\invers@{\v@lZ}{\v@lZ}\v@leur=\repdecn@mb{#5}\v@lZ%
    \v@lZ=#2\advance\v@lZ-#4\mili@u=\repdecn@mb{\v@leur}\v@lZ%
    \v@lZ=#3\advance\v@lZ\mili@u\v@lZa=-\v@lZ\advance\v@lZa\tw@#1}
\ctr@ld@f\def\fcn@deR{\@ecfor\p@int:=\list@num\do{\Figg@tXYa{\p@int}\d@butn@de%
    \advance\v@lXa-0.5\v@lX\advance\v@lYa-0.5\v@lY\Figp@intreg-1:(\v@lXa,\v@lYa)%
    \advance\v@lXa\v@lX\Figp@intreg-2:(\v@lXa,\v@lYa)%
    \advance\v@lYa\v@lY\Figp@intreg-3:(\v@lXa,\v@lYa)%
    \advance\v@lXa-\v@lX\Figp@intreg-4:(\v@lXa,\v@lYa)%
    \ifdim\thickn@ss\p@>\z@\pssetfillmode{yes}\us@secondC@lor
    \Figv@ctCreg-5(-\delt@,-\delt@)\figpttra-9:=-1/1,-5/%
    \Figv@ctCreg-5(\delt@,-\delt@)\figpttra-10:=-2/1,-5/%
    \Figv@ctCreg-5(\delt@,\delt@)\figpttra-11:=-3/1,-5/%
    \figptstra-7=-9,-10,-11/\thickn@ss,-8/\psline[-9,-11,-5,-6,-7]\fi
    \pssetfillmode{yes}\us@thirdC@lor\psline[-1,-2,-3,-4]%
    \pssetfillmode{no}\us@primarC@lor\psline[-1,-2,-3,-4,-1]}}
\ctr@ln@m\@rrowp@s
\ctr@ln@m\Xp@dd     \ctr@ln@m\Yp@dd
\ctr@ln@m\fclin@r@d \ctr@ln@m\thickn@ss
\ctr@ld@f\def\Pssetfl@wchart#1=#2|{\keln@mtr#1|%
    \def\n@mref{arr}\ifx\l@debut\n@mref\expandafter\keln@mtr\l@suite|%
     \def\n@mref{owp}\ifx\l@debut\n@mref\edef\@rrowp@s{#2}\else
     \def\n@mref{owr}\ifx\l@debut\n@mref\setfcr@fpt#2|\else
     \immediate\write16{*** Unknown attribute: \BS@ psset flowchart(..., #1=...)}%
     \fi\fi\else%
    \def\n@mref{lin}\ifx\l@debut\n@mref\setfccurv@#2|\else
    \def\n@mref{pad}\ifx\l@debut\n@mref\edef\Xp@dd{#2}\edef\Yp@dd{#2}\else
    \def\n@mref{rad}\ifx\l@debut\n@mref\edef\fclin@r@d{#2}\else
    \def\n@mref{sha}\ifx\l@debut\n@mref\setfcshap@#2|\else
    \def\n@mref{thi}\ifx\l@debut\n@mref\edef\thickn@ss{#2}\else
    \def\n@mref{xpa}\ifx\l@debut\n@mref\edef\Xp@dd{#2}\else
    \def\n@mref{ypa}\ifx\l@debut\n@mref\edef\Yp@dd{#2}\else
    \immediate\write16{*** Unknown attribute: \BS@ psset flowchart(..., #1=...)}%
    \fi\fi\fi\fi\fi\fi\fi\fi}
\ctr@ln@m\@rrowr@fpt \ctr@ln@m\fclin@typ@
\ctr@ld@f\def\setfcr@fpt#1#2|{\if#1e\def\@rrowr@fpt{1}\else\def\@rrowr@fpt{0}\fi}
\ctr@ld@f\def\setfccurv@#1#2|{\if#1c\def\fclin@typ@{0}\else\def\fclin@typ@{1}\fi}
\ctr@ln@m\h@rdfcXp@dd \ctr@ln@m\h@rdfcYp@dd
\ctr@ln@m\fcn@de \ctr@ln@m\fcsh@pe
\ctr@ld@f\def\setfcshap@#1#2|{%
    \if#1e\let\fcn@de=\fcn@deE\def\h@rdfcXp@dd{4pt}\def\h@rdfcYp@dd{4pt}%
     \edef\fcsh@pe{ellipse}\else%
    \if#1l\let\fcn@de=\fcn@deL\def\h@rdfcXp@dd{4pt}\def\h@rdfcYp@dd{4pt}%
     \edef\fcsh@pe{lozenge}\else%
          \let\fcn@de=\fcn@deR\def\h@rdfcXp@dd{6pt}\def\h@rdfcYp@dd{6pt}%
     \edef\fcsh@pe{rectangle}\fi\fi}
\ctr@ld@f\def\psline[#1]{{\ifcurr@ntPS\ifps@cri\PSc@mment{psline Points=#1}%
    \let\pslign@=\Pslign@P\Pslin@{#1}\PSc@mment{End psline}\fi\fi}}
\ctr@ld@f\def\pslineF#1{{\ifcurr@ntPS\ifps@cri\PSc@mment{pslineF Filename=#1}%
    \let\pslign@=\Pslign@F\Pslin@{#1}\PSc@mment{End pslineF}\fi\fi}}
\ctr@ld@f\def\pslineC(#1){{\ifcurr@ntPS\ifps@cri\PSc@mment{pslineC}%
    \let\pslign@=\Pslign@C\Pslin@{#1}\PSc@mment{End pslineC}\fi\fi}}
\ctr@ld@f\def\Pslin@#1{\iffillm@de\pslign@{#1}%
    \f@gfill%
    \else\pslign@{#1}\ifx\derp@int\premp@int%
    \f@gclosestroke%
    \else\f@gstroke\fi\fi}
\ctr@ld@f\def\Pslign@P#1{\def\list@num{#1}\extrairelepremi@r\p@int\de\list@num%
    \edef\premp@int{\p@int}\f@gnewpath%
    \PSwrit@cmd{\p@int}{\c@mmoveto}{\fwf@g}%
    \@ecfor\p@int:=\list@num\do{\PSwrit@cmd{\p@int}{\c@mlineto}{\fwf@g}%
    \edef\derp@int{\p@int}}}
\ctr@ld@f\def\Pslign@F#1{\s@uvc@ntr@l\et@tPslign@F\setc@ntr@l{2}\openin\frf@g=#1\relax%
    \ifeof\frf@g\message{*** File #1 not found !}\end\else%
    \read\frf@g to\tr@c\edef\premp@int{\tr@c}\expandafter\extr@ctCF\tr@c:%
    \f@gnewpath\PSwrit@cmd{-1}{\c@mmoveto}{\fwf@g}%
    \loop\read\frf@g to\tr@c\ifeof\frf@g\mored@tafalse\else\mored@tatrue\fi%
    \ifmored@ta\expandafter\extr@ctCF\tr@c:\PSwrit@cmd{-1}{\c@mlineto}{\fwf@g}%
    \edef\derp@int{\tr@c}\repeat\fi\closein\frf@g\resetc@ntr@l\et@tPslign@F}
\ctr@ln@m\extr@ctCF
\ctr@ld@f\def\extr@ctCFDD#1 #2:{\v@lX=#1\unit@\v@lY=#2\unit@\Figp@intregDD-1:(\v@lX,\v@lY)}
\ctr@ld@f\def\extr@ctCFTD#1 #2 #3:{\v@lX=#1\unit@\v@lY=#2\unit@\v@lZ=#3\unit@%
    \Figp@intregTD-1:(\v@lX,\v@lY,\v@lZ)}
\ctr@ld@f\def\Pslign@C#1{\s@uvc@ntr@l\et@tPslign@C\setc@ntr@l{2}%
    \def\list@num{#1}\extrairelepremi@r\p@int\de\list@num%
    \edef\premp@int{\p@int}\f@gnewpath%
    \expandafter\Pslign@C@\p@int:\PSwrit@cmd{-1}{\c@mmoveto}{\fwf@g}%
    \@ecfor\p@int:=\list@num\do{\expandafter\Pslign@C@\p@int:%
    \PSwrit@cmd{-1}{\c@mlineto}{\fwf@g}\edef\derp@int{\p@int}}%
    \resetc@ntr@l\et@tPslign@C}
\ctr@ld@f\def\Pslign@C@#1 #2:{{\def\t@xt@{#1}\ifx\t@xt@\empty\Pslign@C@#2:
    \else\extr@ctCF#1 #2:\fi}}
\ctr@ln@m\c@ntrolmesh
\ctr@ld@f\def\Pssetm@sh#1=#2|{\keln@mun#1|%
    \def\n@mref{d}\ifx\l@debut\n@mref\pssetmeshdiag{#2}\else
    \immediate\write16{*** Unknown attribute: \BS@ psset mesh(..., #1=...)}%
    \fi}
\ctr@ld@f\def\pssetmeshdiag#1{\edef\c@ntrolmesh{#1}}
\ctr@ld@f\def\defaultmeshdiag{0}    
\ctr@ld@f\def\psmesh#1,#2[#3,#4,#5,#6]{{\ifcurr@ntPS\ifps@cri%
    \PSc@mment{psmesh N1=#1, N2=#2, Quadrangle=[#3,#4,#5,#6]}%
    \s@uvc@ntr@l\et@tpsmesh\Pss@tsecondSt\setc@ntr@l{2}%
    \ifnum#1>\@ne\Psmeshp@rt#1[#3,#4,#5,#6]\fi%
    \ifnum#2>\@ne\Psmeshp@rt#2[#4,#5,#6,#3]\fi%
    \ifnum\c@ntrolmesh>\z@\Psmeshdi@g#1,#2[#3,#4,#5,#6]\fi%
    \ifnum\c@ntrolmesh<\z@\Psmeshdi@g#2,#1[#4,#5,#6,#3]\fi\Psrest@reSt%
    \psline[#3,#4,#5,#6,#3]\PSc@mment{End psmesh}\resetc@ntr@l\et@tpsmesh\fi\fi}}
\ctr@ld@f\def\Psmeshp@rt#1[#2,#3,#4,#5]{{\l@mbd@un=\@ne\l@mbd@de=#1\loop%
    \ifnum\l@mbd@un<#1\advance\l@mbd@de\m@ne\figptbary-1:[#2,#3;\l@mbd@de,\l@mbd@un]%
    \figptbary-2:[#5,#4;\l@mbd@de,\l@mbd@un]\psline[-1,-2]\advance\l@mbd@un\@ne\repeat}}
\ctr@ld@f\def\Psmeshdi@g#1,#2[#3,#4,#5,#6]{\figptcopy-2:/#3/\figptcopy-3:/#6/%
    \l@mbd@un=\z@\l@mbd@de=#1\loop\ifnum\l@mbd@un<#1%
    \advance\l@mbd@un\@ne\advance\l@mbd@de\m@ne\figptcopy-1:/-2/\figptcopy-4:/-3/%
    \figptbary-2:[#3,#4;\l@mbd@de,\l@mbd@un]%
    \figptbary-3:[#6,#5;\l@mbd@de,\l@mbd@un]\Psmeshdi@gp@rt#2[-1,-2,-3,-4]\repeat}
\ctr@ld@f\def\Psmeshdi@gp@rt#1[#2,#3,#4,#5]{{\l@mbd@un=\z@\l@mbd@de=#1\loop%
    \ifnum\l@mbd@un<#1\figptbary-5:[#2,#5;\l@mbd@de,\l@mbd@un]%
    \advance\l@mbd@de\m@ne\advance\l@mbd@un\@ne%
    \figptbary-6:[#3,#4;\l@mbd@de,\l@mbd@un]\psline[-5,-6]\repeat}}
\ctr@ln@m\psnormal
\ctr@ld@f\def\psnormalDD#1,#2[#3,#4]{{\ifcurr@ntPS\ifps@cri%
    \PSc@mment{psnormal Length=#1, Lambda=#2 [Pt1,Pt2]=[#3,#4]}%
    \s@uvc@ntr@l\et@tpsnormal\resetc@ntr@l{2}\figptendnormal-6::#1,#2[#3,#4]%
    \figptcopyDD-5:/-1/\psarrow[-5,-6]%
    \PSc@mment{End psnormal}\resetc@ntr@l\et@tpsnormal\fi\fi}}
\ctr@ld@f\def\psreset#1{\trtlis@rg{#1}{\Psreset@}}
\ctr@ld@f\def\Psreset@#1|{\keln@mde#1|%
    \def\n@mref{ar}\ifx\l@debut\n@mref\psresetarrowhead\else
    \def\n@mref{cu}\ifx\l@debut\n@mref\psset curve(roundness=\defaultroundness)\else
    \def\n@mref{fi}\ifx\l@debut\n@mref\psset (color=\defaultcolor,dash=\defaultdash,%
         fill=\defaultfill,join=\defaultjoin,width=\defaultwidth)\else
    \def\n@mref{fl}\ifx\l@debut\n@mref\psset flowchart(arrowp=\defaultfcarrowposition,%
    arrowr=\defaultfcarrowrefpt,line=\defaultfcline,xpadd=\defaultfcxpadding,%
    ypadd=\defaultfcypadding,radius=\defaultfcradius,shape=\defaultfcshape,%
    thick=\defaultfcthickness)\else
    \def\n@mref{me}\ifx\l@debut\n@mref\psset mesh(diag=\defaultmeshdiag)\else
    \def\n@mref{se}\ifx\l@debut\n@mref\psresetsecondsettings\else
    \def\n@mref{th}\ifx\l@debut\n@mref\psset third(color=\defaultthirdcolor)\else
    \immediate\write16{*** Unknown keyword #1 (\BS@ psreset).}%
    \fi\fi\fi\fi\fi\fi\fi}
\ctr@ld@f\def\psset#1(#2){\def\t@xt@{#1}\ifx\t@xt@\empty\trtlis@rg{#2}{\Pssetf@rst}
    \else\keln@mde#1|%
    \def\n@mref{ar}\ifx\l@debut\n@mref\trtlis@rg{#2}{\Psset@rrowhe@d}\else
    \def\n@mref{cu}\ifx\l@debut\n@mref\trtlis@rg{#2}{\Pssetc@rve}\else
    \def\n@mref{fi}\ifx\l@debut\n@mref\trtlis@rg{#2}{\Pssetf@rst}\else
    \def\n@mref{fl}\ifx\l@debut\n@mref\trtlis@rg{#2}{\Pssetfl@wchart}\else
    \def\n@mref{me}\ifx\l@debut\n@mref\trtlis@rg{#2}{\Pssetm@sh}\else
    \def\n@mref{se}\ifx\l@debut\n@mref\trtlis@rg{#2}{\Pssets@cond}\else
    \def\n@mref{th}\ifx\l@debut\n@mref\trtlis@rg{#2}{\Pssetth@rd}\else
    \immediate\write16{*** Unknown keyword: \BS@ psset #1(...)}%
    \fi\fi\fi\fi\fi\fi\fi\fi}
\ctr@ld@f\def\pssetdefault#1(#2){\ifcurr@ntPS\immediate\write16{*** \BS@ pssetdefault is ignored
    inside a \BS@ psbeginfig-\BS@ psendfig block.}%
    \immediate\write16{*** It must be called before \BS@ psbeginfig.}\else%
    \def\t@xt@{#1}\ifx\t@xt@\empty\trtlis@rg{#2}{\Pssd@f@rst}\else\keln@mde#1|%
    \def\n@mref{ar}\ifx\l@debut\n@mref\trtlis@rg{#2}{\Pssd@@rrowhe@d}\else
    \def\n@mref{cu}\ifx\l@debut\n@mref\trtlis@rg{#2}{\Pssd@c@rve}\else
    \def\n@mref{fi}\ifx\l@debut\n@mref\trtlis@rg{#2}{\Pssd@f@rst}\else
    \def\n@mref{fl}\ifx\l@debut\n@mref\trtlis@rg{#2}{\Pssd@fl@wchart}\else
    \def\n@mref{me}\ifx\l@debut\n@mref\trtlis@rg{#2}{\Pssd@m@sh}\else
    \def\n@mref{se}\ifx\l@debut\n@mref\trtlis@rg{#2}{\Pssd@s@cond}\else
    \def\n@mref{th}\ifx\l@debut\n@mref\trtlis@rg{#2}{\Pssd@th@rd}\else
    \immediate\write16{*** Unknown keyword: \BS@ pssetdefault #1(...)}%
    \fi\fi\fi\fi\fi\fi\fi\fi\initpss@ttings\fi}
\ctr@ld@f\def\Pssd@f@rst#1=#2|{\keln@mun#1|%
    \def\n@mref{c}\ifx\l@debut\n@mref\edef\defaultcolor{#2}\else
    \def\n@mref{d}\ifx\l@debut\n@mref\edef\defaultdash{#2}\else
    \def\n@mref{f}\ifx\l@debut\n@mref\edef\defaultfill{#2}\else
    \def\n@mref{j}\ifx\l@debut\n@mref\edef\defaultjoin{#2}\else
    \def\n@mref{u}\ifx\l@debut\n@mref\edef\defaultupdate{#2}\pssetupdate{#2}\else
    \def\n@mref{w}\ifx\l@debut\n@mref\edef\defaultwidth{#2}\else
    \immediate\write16{*** Unknown attribute: \BS@ pssetdefault (..., #1=...)}%
    \fi\fi\fi\fi\fi\fi}
\ctr@ld@f\def\Pssd@@rrowhe@d#1=#2|{\keln@mun#1|%
    \def\n@mref{a}\ifx\l@debut\n@mref\edef\defaultarrowheadangle{#2}\else
    \def\n@mref{f}\ifx\l@debut\n@mref\edef\defaultarrowheadangle{#2}\else
    \def\n@mref{l}\ifx\l@debut\n@mref\y@tiunit{#2}\ifunitpr@sent%
     \edef\defaulth@rdahlength{#2}\else\edef\defaulth@rdahlength{#2pt}%
     \message{*** \BS@ pssetdefault (..., #1=#2, ...) : unit is missing, pt is assumed.}%
     \fi\else
    \def\n@mref{o}\ifx\l@debut\n@mref\edef\defaultarrowheadout{#2}\else
    \def\n@mref{r}\ifx\l@debut\n@mref\edef\defaultarrowheadratio{#2}\else
    \immediate\write16{*** Unknown attribute: \BS@ pssetdefault arrowhead(..., #1=...)}%
    \fi\fi\fi\fi\fi}
\ctr@ld@f\def\Pssd@c@rve#1=#2|{\keln@mun#1|%
    \def\n@mref{r}\ifx\l@debut\n@mref\edef\defaultroundness{#2}\else%
    \immediate\write16{*** Unknown attribute: \BS@ pssetdefault curve(..., #1=...)}%
    \fi}
\ctr@ld@f\def\Pssd@fl@wchart#1=#2|{\keln@mtr#1|%
    \def\n@mref{arr}\ifx\l@debut\n@mref\expandafter\keln@mtr\l@suite|%
     \def\n@mref{owp}\ifx\l@debut\n@mref\edef\defaultfcarrowposition{#2}\else
     \def\n@mref{owr}\ifx\l@debut\n@mref\edef\defaultfcarrowrefpt{#2}\else
     \immediate\write16{*** Unknown attribute: \BS@ pssetdefault flowchart(..., #1=...)}%
     \fi\fi\else%
    \def\n@mref{lin}\ifx\l@debut\n@mref\edef\defaultfcline{#2}\else
    \def\n@mref{pad}\ifx\l@debut\n@mref\edef\defaultfcxpadding{#2}%
                    \edef\defaultfcypadding{#2}\else
    \def\n@mref{rad}\ifx\l@debut\n@mref\edef\defaultfcradius{#2}\else
    \def\n@mref{sha}\ifx\l@debut\n@mref\edef\defaultfcshape{#2}\else
    \def\n@mref{thi}\ifx\l@debut\n@mref\edef\defaultfcthickness{#2}\else
    \def\n@mref{xpa}\ifx\l@debut\n@mref\edef\defaultfcxpadding{#2}\else
    \def\n@mref{ypa}\ifx\l@debut\n@mref\edef\defaultfcypadding{#2}\else
    \immediate\write16{*** Unknown attribute: \BS@ pssetdefault flowchart(..., #1=...)}%
    \fi\fi\fi\fi\fi\fi\fi\fi}
\ctr@ld@f\def\defaultfcarrowposition{0.5}
\ctr@ld@f\def\defaultfcarrowrefpt{start}
\ctr@ld@f\def\defaultfcline{polygon}
\ctr@ld@f\def\defaultfcradius{0}
\ctr@ld@f\def\defaultfcshape{rectangle}
\ctr@ld@f\def\defaultfcthickness{0}
\ctr@ld@f\def\defaultfcxpadding{0}
\ctr@ld@f\def\defaultfcypadding{0}
\ctr@ld@f\def\Pssd@m@sh#1=#2|{\keln@mun#1|%
    \def\n@mref{d}\ifx\l@debut\n@mref\edef\defaultmeshdiag{#2}\else%
    \immediate\write16{*** Unknown attribute: \BS@ pssetdefault mesh(..., #1=...)}%
    \fi}
\ctr@ld@f\def\Pssd@s@cond#1=#2|{\keln@mun#1|%
    \def\n@mref{c}\ifx\l@debut\n@mref\edef\defaultsecondcolor{#2}\else%
    \def\n@mref{d}\ifx\l@debut\n@mref\edef\defaultseconddash{#2}\else%
    \def\n@mref{w}\ifx\l@debut\n@mref\edef\defaultsecondwidth{#2}\else%
    \immediate\write16{*** Unknown attribute: \BS@ pssetdefault second(..., #1=...)}%
    \fi\fi\fi}
\ctr@ld@f\def\Pssd@th@rd#1=#2|{\keln@mun#1|%
    \def\n@mref{c}\ifx\l@debut\n@mref\edef\defaultthirdcolor{#2}\else%
    \immediate\write16{*** Unknown attribute: \BS@ pssetdefault third(..., #1=...)}%
    \fi}
\ctr@ln@w{newif}\iffillm@de
\ctr@ld@f\def\pssetfillmode#1{\expandafter\setfillm@de#1:}
\ctr@ld@f\def\setfillm@de#1#2:{\if#1n\fillm@defalse\else\fillm@detrue\fi}
\ctr@ld@f\def\defaultfill{no}     
\ctr@ln@w{newif}\ifpsupdatem@de
\ctr@ld@f\def\pssetupdate#1{\ifcurr@ntPS\immediate\write16{*** \BS@ pssetupdate is ignored inside a
     \BS@ psbeginfig-\BS@ psendfig block.}%
    \immediate\write16{*** It must be called before \BS@ psbeginfig.}%
    \else\expandafter\setupd@te#1:\fi}
\ctr@ld@f\def\setupd@te#1#2:{\if#1n\psupdatem@defalse\else\psupdatem@detrue\fi}
\ctr@ld@f\def\defaultupdate{no}     
\ctr@ln@m\curr@ntcolor \ctr@ln@m\curr@ntcolorc@md
\ctr@ld@f\def\Pssetc@lor#1{\ifps@cri\result@tent=\@ne\expandafter\c@lnbV@l#1 :%
    \def\curr@ntcolor{}\def\curr@ntcolorc@md{}%
    \ifcase\result@tent\or\pssetgray{#1}\or\or\pssetrgb{#1}\or\pssetcmyk{#1}\fi\fi}
\ctr@ln@m\curr@ntcolorc@mdStroke
\ctr@ld@f\def\pssetcmyk#1{\ifps@cri\def\curr@ntcolor{#1}\def\curr@ntcolorc@md{\c@msetcmykcolor}%
    \def\curr@ntcolorc@mdStroke{\c@msetcmykcolorStroke}%
    \ifcurr@ntPS\PSc@mment{pssetcmyk Color=#1}\us@primarC@lor\fi\fi}
\ctr@ld@f\def\pssetrgb#1{\ifps@cri\def\curr@ntcolor{#1}\def\curr@ntcolorc@md{\c@msetrgbcolor}%
    \def\curr@ntcolorc@mdStroke{\c@msetrgbcolorStroke}%
    \ifcurr@ntPS\PSc@mment{pssetrgb Color=#1}\us@primarC@lor\fi\fi}
\ctr@ld@f\def\pssetgray#1{\ifps@cri\def\curr@ntcolor{#1}\def\curr@ntcolorc@md{\c@msetgray}%
    \def\curr@ntcolorc@mdStroke{\c@msetgrayStroke}%
    \ifcurr@ntPS\PSc@mment{pssetgray Gray level=#1}\us@primarC@lor\fi\fi}
\ctr@ln@m\fillc@md
\ctr@ld@f\def\us@primarC@lor{\immediate\write\fwf@g{\d@fprimarC@lor}%
    \let\fillc@md=\prfillc@md}
\ctr@ld@f\def\prfillc@md{\d@fprimarC@lor\space\c@mfill}
\ctr@ld@f\def\defaultcolor{0}       
\ctr@ld@f\def\c@lnbV@l#1 #2:{\def\t@xt@{#1}\relax\ifx\t@xt@\empty\c@lnbV@l#2:
    \else\c@lnbV@l@#1 #2:\fi}
\ctr@ld@f\def\c@lnbV@l@#1 #2:{\def\t@xt@{#2}\ifx\t@xt@\empty%
    \def\t@xt@{#1}\ifx\t@xt@\empty\advance\result@tent\m@ne\fi
    \else\advance\result@tent\@ne\c@lnbV@l@#2:\fi}
\ctr@ld@f\def\Blackcmyk{0 0 0 1}
\ctr@ld@f\def\Whitecmyk{0 0 0 0}
\ctr@ld@f\def\Cyancmyk{1 0 0 0}
\ctr@ld@f\def\Magentacmyk{0 1 0 0}
\ctr@ld@f\def\Yellowcmyk{0 0 1 0}
\ctr@ld@f\def\Redcmyk{0 1 1 0}
\ctr@ld@f\def\Greencmyk{1 0 1 0}
\ctr@ld@f\def\Bluecmyk{1 1 0 0}
\ctr@ld@f\def\Graycmyk{0 0 0 0.50}
\ctr@ld@f\def\BrickRedcmyk{0 0.89 0.94 0.28} 
\ctr@ld@f\def\Browncmyk{0 0.81 1 0.60} 
\ctr@ld@f\def\ForestGreencmyk{0.91 0 0.88 0.12} 
\ctr@ld@f\def\Goldenrodcmyk{ 0 0.10 0.84 0} 
\ctr@ld@f\def\Marooncmyk{0 0.87 0.68 0.32} 
\ctr@ld@f\def\Orangecmyk{0 0.61 0.87 0} 
\ctr@ld@f\def\Purplecmyk{0.45 0.86 0 0} 
\ctr@ld@f\def\RoyalBluecmyk{1. 0.50 0 0} 
\ctr@ld@f\def\Violetcmyk{0.79 0.88 0 0} 
\ctr@ld@f\def\Blackrgb{0 0 0}
\ctr@ld@f\def\Whitergb{1 1 1}
\ctr@ld@f\def\Redrgb{1 0 0}
\ctr@ld@f\def\Greenrgb{0 1 0}
\ctr@ld@f\def\Bluergb{0 0 1}
\ctr@ld@f\def\Cyanrgb{0 1 1}
\ctr@ld@f\def\Magentargb{1 0 1}
\ctr@ld@f\def\Yellowrgb{1 1 0}
\ctr@ld@f\def\Grayrgb{0.5 0.5 0.5}
\ctr@ld@f\def\Chocolatergb{0.824 0.412 0.118}
\ctr@ld@f\def\DarkGoldenrodrgb{0.722 0.525 0.043}
\ctr@ld@f\def\DarkOrangergb{1 0.549 0}
\ctr@ld@f\def\Firebrickrgb{0.698 0.133 0.133}
\ctr@ld@f\def\ForestGreenrgb{0.133 0.545 0.133}
\ctr@ld@f\def\Goldrgb{1 0.843 0}
\ctr@ld@f\def\HotPinkrgb{1 0.412 0.706}
\ctr@ld@f\def\Maroonrgb{0.690 0.188 0.376}
\ctr@ld@f\def\Pinkrgb{1 0.753 0.796}
\ctr@ld@f\def\RoyalBluergb{0.255 0.412 0.882}
\ctr@ld@f\def\Pssetf@rst#1=#2|{\keln@mun#1|%
    \def\n@mref{c}\ifx\l@debut\n@mref\Pssetc@lor{#2}\else
    \def\n@mref{d}\ifx\l@debut\n@mref\pssetdash{#2}\else
    \def\n@mref{f}\ifx\l@debut\n@mref\pssetfillmode{#2}\else
    \def\n@mref{j}\ifx\l@debut\n@mref\pssetjoin{#2}\else
    \def\n@mref{u}\ifx\l@debut\n@mref\pssetupdate{#2}\else
    \def\n@mref{w}\ifx\l@debut\n@mref\pssetwidth{#2}\else
    \immediate\write16{*** Unknown attribute: \BS@ psset (..., #1=...)}%
    \fi\fi\fi\fi\fi\fi}
\ctr@ln@m\curr@ntdash
\ctr@ld@f\def\s@uvdash#1{\edef#1{\curr@ntdash}}
\ctr@ld@f\def\defaultdash{1}        
\ctr@ld@f\def\pssetdash#1{\ifps@cri\edef\curr@ntdash{#1}\ifcurr@ntPS\expandafter\Pssetd@sh#1 :\fi\fi}
\ctr@ld@f\def\Pssetd@shI#1{\PSc@mment{pssetdash Index=#1}\ifcase#1%
    \or\immediate\write\fwf@g{[] 0 \c@msetdash}
    \or\immediate\write\fwf@g{[6 2] 0 \c@msetdash}
    \or\immediate\write\fwf@g{[4 2] 0 \c@msetdash}
    \or\immediate\write\fwf@g{[2 2] 0 \c@msetdash}
    \or\immediate\write\fwf@g{[1 2] 0 \c@msetdash}
    \or\immediate\write\fwf@g{[2 4] 0 \c@msetdash}
    \or\immediate\write\fwf@g{[3 5] 0 \c@msetdash}
    \or\immediate\write\fwf@g{[3 3] 0 \c@msetdash}
    \or\immediate\write\fwf@g{[3 5 1 5] 0 \c@msetdash}
    \or\immediate\write\fwf@g{[6 4 2 4] 0 \c@msetdash}
    \fi}
\ctr@ld@f\def\Pssetd@sh#1 #2:{{\def\t@xt@{#1}\ifx\t@xt@\empty\Pssetd@sh#2:
    \else\def\t@xt@{#2}\ifx\t@xt@\empty\Pssetd@shI{#1}\else\s@mme=\@ne\def\debutp@t{#1}%
    \an@lysd@sh#2:\ifodd\s@mme\edef\debutp@t{\debutp@t\space\finp@t}\def\finp@t{0}\fi%
    \PSc@mment{pssetdash Pattern=#1 #2}%
    \immediate\write\fwf@g{[\debutp@t] \finp@t\space\c@msetdash}\fi\fi}}
\ctr@ld@f\def\an@lysd@sh#1 #2:{\def\t@xt@{#2}\ifx\t@xt@\empty\def\finp@t{#1}\else%
    \edef\debutp@t{\debutp@t\space#1}\advance\s@mme\@ne\an@lysd@sh#2:\fi}
\ctr@ln@m\curr@ntwidth
\ctr@ld@f\def\s@uvwidth#1{\edef#1{\curr@ntwidth}}
\ctr@ld@f\def\defaultwidth{0.4}     
\ctr@ld@f\def\pssetwidth#1{\ifps@cri\edef\curr@ntwidth{#1}\ifcurr@ntPS%
    \PSc@mment{pssetwidth Width=#1}\immediate\write\fwf@g{#1 \c@msetlinewidth}\fi\fi}
\ctr@ln@m\curr@ntjoin
\ctr@ld@f\def\pssetjoin#1{\ifps@cri\edef\curr@ntjoin{#1}\ifcurr@ntPS\expandafter\Pssetj@in#1:\fi\fi}
\ctr@ld@f\def\Pssetj@in#1#2:{\PSc@mment{pssetjoin join=#1}%
    \if#1r\def\t@xt@{1}\else\if#1b\def\t@xt@{2}\else\def\t@xt@{0}\fi\fi%
    \immediate\write\fwf@g{\t@xt@\space\c@msetlinejoin}}
\ctr@ld@f\def\defaultjoin{miter}   
\ctr@ld@f\def\Pssets@cond#1=#2|{\keln@mun#1|%
    \def\n@mref{c}\ifx\l@debut\n@mref\Pssets@condcolor{#2}\else%
    \def\n@mref{d}\ifx\l@debut\n@mref\pssetseconddash{#2}\else%
    \def\n@mref{w}\ifx\l@debut\n@mref\pssetsecondwidth{#2}\else%
    \immediate\write16{*** Unknown attribute: \BS@ psset second(..., #1=...)}%
    \fi\fi\fi}
\ctr@ln@m\curr@ntseconddash
\ctr@ld@f\def\pssetseconddash#1{\edef\curr@ntseconddash{#1}}
\ctr@ld@f\def\defaultseconddash{4}  
\ctr@ln@m\curr@ntsecondwidth
\ctr@ld@f\def\pssetsecondwidth#1{\edef\curr@ntsecondwidth{#1}}
\ctr@ld@f\edef\defaultsecondwidth{\defaultwidth} 
\ctr@ld@f\def\psresetsecondsettings{%
    \pssetseconddash{\defaultseconddash}\pssetsecondwidth{\defaultsecondwidth}%
    \Pssets@condcolor{\defaultsecondcolor}}
\ctr@ln@m\sec@ndcolor \ctr@ln@m\sec@ndcolorc@md
\ctr@ld@f\def\Pssets@condcolor#1{\ifps@cri\result@tent=\@ne\expandafter\c@lnbV@l#1 :%
    \def\sec@ndcolor{}\def\sec@ndcolorc@md{}%
    \ifcase\result@tent\or\pssetsecondgray{#1}\or\or\pssetsecondrgb{#1}%
    \or\pssetsecondcmyk{#1}\fi\fi}
\ctr@ln@m\sec@ndcolorc@mdStroke
\ctr@ld@f\def\pssetsecondcmyk#1{\def\sec@ndcolor{#1}\def\sec@ndcolorc@md{\c@msetcmykcolor}%
    \def\sec@ndcolorc@mdStroke{\c@msetcmykcolorStroke}}
\ctr@ld@f\def\pssetsecondrgb#1{\def\sec@ndcolor{#1}\def\sec@ndcolorc@md{\c@msetrgbcolor}%
    \def\sec@ndcolorc@mdStroke{\c@msetrgbcolorStroke}}
\ctr@ld@f\def\pssetsecondgray#1{\def\sec@ndcolor{#1}\def\sec@ndcolorc@md{\c@msetgray}%
    \def\sec@ndcolorc@mdStroke{\c@msetgrayStroke}}
\ctr@ld@f\def\us@secondC@lor{\immediate\write\fwf@g{\d@fsecondC@lor}%
    \let\fillc@md=\sdfillc@md}
\ctr@ld@f\def\sdfillc@md{\d@fsecondC@lor\space\c@mfill}
\ctr@ld@f\edef\defaultsecondcolor{\defaultcolor} 
\ctr@ld@f\def\Pss@tsecondSt{%
    \s@uvdash{\typ@dash}\pssetdash{\curr@ntseconddash}%
    \s@uvwidth{\typ@width}\pssetwidth{\curr@ntsecondwidth}\us@secondC@lor}
\ctr@ld@f\def\Psrest@reSt{\pssetwidth{\typ@width}\pssetdash{\typ@dash}\us@primarC@lor}
\ctr@ld@f\def\Pssetth@rd#1=#2|{\keln@mun#1|%
    \def\n@mref{c}\ifx\l@debut\n@mref\Pssetth@rdcolor{#2}\else%
    \immediate\write16{*** Unknown attribute: \BS@ psset third(..., #1=...)}%
    \fi}
\ctr@ln@m\th@rdcolor \ctr@ln@m\th@rdcolorc@md
\ctr@ld@f\def\Pssetth@rdcolor#1{\ifps@cri\result@tent=\@ne\expandafter\c@lnbV@l#1 :%
    \def\th@rdcolor{}\def\th@rdcolorc@md{}%
    \ifcase\result@tent\or\Pssetth@rdgray{#1}\or\or\Pssetth@rdrgb{#1}%
    \or\Pssetth@rdcmyk{#1}\fi\fi}
\ctr@ln@m\th@rdcolorc@mdStroke
\ctr@ld@f\def\Pssetth@rdcmyk#1{\def\th@rdcolor{#1}\def\th@rdcolorc@md{\c@msetcmykcolor}%
    \def\th@rdcolorc@mdStroke{\c@msetcmykcolorStroke}}
\ctr@ld@f\def\Pssetth@rdrgb#1{\def\th@rdcolor{#1}\def\th@rdcolorc@md{\c@msetrgbcolor}%
    \def\th@rdcolorc@mdStroke{\c@msetrgbcolorStroke}}
\ctr@ld@f\def\Pssetth@rdgray#1{\def\th@rdcolor{#1}\def\th@rdcolorc@md{\c@msetgray}%
    \def\th@rdcolorc@mdStroke{\c@msetgrayStroke}}
\ctr@ld@f\def\us@thirdC@lor{\immediate\write\fwf@g{\d@fthirdC@lor}%
    \let\fillc@md=\thfillc@md}
\ctr@ld@f\def\thfillc@md{\d@fthirdC@lor\space\c@mfill}
\ctr@ld@f\def\defaultthirdcolor{1}  
\ctr@ld@f\def\pstrimesh#1[#2,#3,#4]{{\ifcurr@ntPS\ifps@cri%
    \PSc@mment{pstrimesh Type=#1, Triangle=[#2,#3,#4]}%
    \s@uvc@ntr@l\et@tpstrimesh\ifnum#1>\@ne\Pss@tsecondSt\setc@ntr@l{2}%
    \Pstrimeshp@rt#1[#2,#3,#4]\Pstrimeshp@rt#1[#3,#4,#2]%
    \Pstrimeshp@rt#1[#4,#2,#3]\Psrest@reSt\fi\psline[#2,#3,#4,#2]%
    \PSc@mment{End pstrimesh}\resetc@ntr@l\et@tpstrimesh\fi\fi}}
\ctr@ld@f\def\Pstrimeshp@rt#1[#2,#3,#4]{{\l@mbd@un=\@ne\l@mbd@de=#1\loop\ifnum\l@mbd@de>\@ne%
    \advance\l@mbd@de\m@ne\figptbary-1:[#2,#3;\l@mbd@de,\l@mbd@un]%
    \figptbary-2:[#2,#4;\l@mbd@de,\l@mbd@un]\psline[-1,-2]%
    \advance\l@mbd@un\@ne\repeat}}
\initpr@lim\initpss@ttings\initPDF@rDVI
\ctr@ln@w{newbox}\figBoxA
\ctr@ln@w{newbox}\figBoxB
\ctr@ln@w{newbox}\figBoxC
\catcode`\@=12


\figinit{.8cm} 
\figpt 1:(-5,0)\figpt 2:(13,0)\figpt 3:(0,-9)\figpt 4:(0,14)

\figpt5:(-1,-8)\figpt 6:(10,3)
\figpt70:(-2,13)\figpt 80:(12,-1)
\figpt90:(8,-3)\figpt 91:(8,9)
\figpt92:(-1,7)\figpt 93:(10,7)

\figpt 50: (0,11)
\figpt 51: (0,-7)
\figpt 52: (7,0)
\figpt 53: (11,0)
\figpt 54: (0,7)

\figpt 40: (0,0)
\figpt 41: (4,7)
\figpt 42: (8,3)
\figpt 43: (8,1)

\figpt 44: (3,3)
\psbeginfig{}
\psarrow[1,2]\psarrow[3,4]\psline[5,6]\psline[70,80]\psline[90,91]

\psset(width=1,fillmode=yes,color=0.7)\psline[40,50,41,42,43,52,40]

\psendfig 
\figvisu{\figBoxA} {Figure $1$}{ \figwritec [44] {Acceptable couples
$\bigr{(}a,\xi \bigr{)}$} \figwritee 80: $\rfb{serge1}$ (.1)
\figwritee 91: $\rfb{serge4}$
(.1) \figwritee 2 : $a$ (.1) \figwritee 4 : $\xi$ (.1) \figwritee
6:$\rfb{serge3}$ (.1) \figsetmark{$\bullet$}
\figwrites 53 :$2 \bigr{(} \gamma_1 -
\gamma_2 \bigr{)}$ (.1) \figwrites 52 :$\frac{ 2 \gamma_2
e^{-\tau}}{2 + 3 \gamma_1}$ (.1) \figwritew 50 :$2 \bigr{(}\gamma_1
- \gamma_2 \bigr{)}$ (.1)
\figwritew 51 :$- 2 \gamma_2 e^{-\tau}$
(.1)
}
\centerline{\box\figBoxA}

According to this figure, it is clear that for $a$ and $\xi$ small
enough, (\ref{serge1}) to (\ref{serge4}) are valid. Note further
that due to \eqref{serge5} if $k$ goes to $\infty$ or to 0, then
$\gamma_1$ must tend to zero, and therefore $\gamma_2$ as well and
the maximal value $a_0$ of $a$ goes   to zero.
 \end{proof}

From Proposition \ref{proExp} and the energy equivalence
(\ref{equiv}) we deduce estimate (\ref{expestimate}).

\begin{remark}\label{explicit}
 {\rm
We can make explicit the relation between $k$ and $a_0$ by choosing
the constants $\xi, \gamma_1, \gamma_2$ in the definitions (\ref{energy}), (\ref{Lyapunov}) of the energy $E(\cdot)$
and of the Lyapunov functional ${\cal E}(\cdot)$ in such a way that conditions (\ref{serge1})--(\ref{serge5}) are satisfied.
Moreover, we need to fix $\epsilon >0$ in the estimate (\ref{stimaM}) such that (\ref{scelta}) holds.
For instance, fix
$$\epsilon=\frac{1}{C(P)},\quad \xi=2a\,.$$
Now, choose
$$\gamma_1=\min\Big
\{ \frac{1}{3},
(
2\Vert m\Vert_\infty +C_0(P)+1)^{-1},
 k\Big (
 k^2\Big ( \Vert m\Vert_\infty^2
\frac{2}{\delta} +\frac{(n-1)^2}{2}C(P)\Big )+\Vert m\Vert_\infty
\Big
)^{-1}
\Big
\}\,,$$
and $\gamma_2=\frac{\gamma_1}{2}\,.$

The choice of $\gamma_1 \le (
2\Vert m\Vert_\infty +C_0(P)+1)^{-1}$ ensures the equivalence between the energy $E(\cdot)$ and the Lyapunov functional ${\cal E}
(\cdot )
.$ Moreover, with the above choices
of $\gamma_1$ and $\gamma_2$  conditions (\ref{serge3}) and
(\ref{serge5}) are satisfied for any $a>0.$

The remaing conditions are satisfied
for all $0\le a<a_0,$
with
$$a_0=\min\Big
\{
\frac{\gamma_1}{3},
\frac{1}{2}(\Vert m\Vert_\infty^2+\frac{1}{2}(n-1)^2C_0(P))^{-1}\Big
\},
$$
that is
$$
\begin{array}{l}
\displaystyle{
a_0=\min\Big\{
\frac{1}{9}, \frac{1}{3}
(2
\Vert m\Vert_\infty +
C_0(P)+
1
)^{-1},
\frac{k}{3}
\Big (
 k^2\Big ( \Vert m\Vert_\infty^2
\frac{2}{\delta} +\frac{(n-1)^2}{2}C(P)\Big )+\Vert m\Vert_\infty
\Big
)^{-1},
}
\\
\\
\hspace{9cm}
\displaystyle{
\frac{1}{2}(\Vert m\Vert_\infty^2+\frac{1}{2}(n-1)^2C_0(P))^{-1}\Big \}}
\end{array}
$$
Note that $a_0$ depends only on the geometry of the domain $\Omega$ and on $k.$
Moreover, observe that
 $a_0\rightarrow 0$
if $k\rightarrow 0$ and, also, if $k\rightarrow +\infty.$
This is in agreement with the result of \cite{DLP} (cfr. (\ref{forDLP})).
}\end{remark}
\section*{Acknowledgments}
The authors thank the referees for their helpful suggestions and
comments.

\end{document}